\documentclass[fleqn]{article}
\usepackage[utf8]{inputenc}
\usepackage{enumerate,amsthm,amssymb,xcolor,hyperref,comment}
\usepackage[margin=1in]{geometry}

\usepackage[leqno]{amsmath}

\makeatletter
\newcommand{\leqnomode}{\tagsleft@true}
\newcommand{\reqnomode}{\tagsleft@false}
\makeatother

\newtheorem{lemma}{Lemma}
\newtheorem{theorem}{Theorem}
\newcommand{\sset}[1]{\left\{#1\right\}}
\newcommand{\spc}{(G, S, X_0, X, Y_0, Y, f)}
\newcommand{\sspc}{(G, S, X_0, X, Y, Y^*, f)}
\newcommand{\esspc}{(G, S, X_0,X, Y^*, f)}

\def\longbox#1{\parbox{0.85\textwidth}{#1}}

\title{Four-coloring $P_6$-free graphs.\\ II. Finding an excellent precoloring}
\author{
Maria Chudnovsky\thanks{Supported by NSF grants DMS-1550991 and US Army Research Office grant W911NF-16-1-0404.}\\
Princeton University, Princeton, NJ 08544
\\
\\
Sophie Spirkl\\
Princeton University, Princeton, NJ 08544
\\
\\
Mingxian Zhong\\
Columbia University, New York, NY 10027}
\begin{document}
\maketitle
\begin{abstract} 
This is the second paper in a series of two. The goal of the series is to   
give a polynomial time algorithm for the  \textsc{$4$-coloring problem} and 
the \textsc{$4$-precoloring    extension} problem restricted to the class of 
graphs with no  induced six-vertex path, thus proving a conjecture of Huang. 
Combined with previously known 
results this completes the classification of the complexity of the $4$-coloring 
problem for graphs with a connected forbidden induced subgraph.

In this paper we give a polynomial time algorithm that starts with a
$4$-precoloring of a graph with no induced six-vertex path, and outputs
a polynomial-size collection of so-called excellent precolorings.
Excellent precolorings are easier to handle than general ones, and, in
addition, in order to determine whether the initial precoloring can be
extended to the whole graph, it is enough to answer the same question
for each of the excellent precolorings in the collection.  The first
paper in the series deals with excellent precolorings, thus providing
a complete solution to the problem.
\end{abstract}

\section{Introduction} \label{sec:intro}

All graphs in this paper are finite and simple.
We use $[k]$ to denote the set $\sset{1, \dots, k}$. Let $G$ be a graph. A
\emph{$k$-coloring} of $G$ is a function
$f:V(G) \rightarrow [k]$.   A $k$-coloring is {\em proper} if for every edge 
$uv \in E(G)$,
$f(u) \neq f(v)$, and $G$ is \emph{$k$-colorable} if $G$ has a
proper $k$-coloring. The \textsc{$k$-coloring problem} is the problem of
deciding, given a graph $G$, if $G$ is $k$-colorable. This
problem is well-known to be $NP$-hard for all $k \geq 3$.

Let $G$ be a graph. For $X \subseteq V(G)$ we denote by
$G|X$ the subgraph induced by $G$ on $X$, and by
$G \setminus X$ the graph $G|(V(G) \setminus X)$.
We say that $X$ is {\em connected} if $G|X$ is connected.
If $X=\{x\}$ we write $G \setminus x$ to mean $G \setminus \{x\}$.
For disjoint subsets $A,B \subset V(G)$ we say that $A$ is \emph{complete} to $B$ if every vertex of $A$ is adjacent to every vertex of $B$, and that 
$A$ is \emph{anticomplete} to $B$ if every vertex of $A$ is non-adjacent to
every vertex of $B$. If $A=\{a\}$ we write $a$ is complete (or anticomplete)
to $B$ to mean that $\{a\}$ is complete (or anticomplete) to $B$.
If $a$ is not complete and not anticomplete to $B$,
we say that $a$ is \emph{mixed} on $B$. Finally, if $H$ is an induced subgraph
of $G$ and $a \in V(G) \setminus V(H)$, we say that $a$ is \emph{complete to, 
anticomplete to}  or \emph{mixed on} $H$ if $a$ is complete to, anticomplete 
to or mixed on 
$V(H)$, respectively. For
$X \subseteq V(G)$, we say that $e \in E(G)$ is \emph{an edge of $X$}
if both endpoints of $e$ are in $X$.
For $v \in V(G)$ we write $N_G(v)$ (or $N(v)$ when there is no danger of confusion) to mean the set of vertices of $G$ that are adjacent to  $v$. Observe that since $G$ is simple, $v \not \in N(v)$. For $X \subseteq V(G)$ we define
$N(X)=(\bigcup_{v \in X} N(v)) \setminus X$.

A function $L: V(G) \rightarrow 2^{[k]}$ that assigns a subset of
$[k]$ to each vertex of a graph $G$ is a \emph{$k$-list assignment}
for $G$. For a $k$-list assignment $L$, a function
$f: V(G) \rightarrow [k]$ is an $L$-coloring if $f$ is a $k$-coloring
of $G$ and $f(v) \in L(v)$ for all $v \in V(G)$. We say that
$G$ is {\em $L$-colorable}, or that $(G,L)$ is {\em colorable},
if $G$ has a proper $L$-coloring. The \textsc{$k$-list
  coloring problem} is the problem of deciding, given a graph $G$ and
a $k$-list assignment $L$, if $G$ is $L$-colorable. Since this
generalizes the $k$-coloring problem, it is $NP$-hard for all
$k \geq 3$.

A \emph{$k$-precoloring} $(G, X, f)$ of a graph $G$ is a function
$f: X \rightarrow [k]$ for a set $X \subseteq V(G)$ such that $f$ is a
proper $k$-coloring of $G|X$. Equivalently, a $k$-precoloring is a
$k$-list assignment $L$ in which $|L(v)| \in \sset{1, k}$ for all
$v \in V(G)$. A \emph{$k$-precoloring extension} for $(G, X, f)$ is a
proper $k$-coloring $g$ of $G$ such that $g|_X = f|_X$, and the
\textsc{$k$-precoloring extension problem} is the problem of deciding,
given a graph $G$ and a $k$-precoloring $(G, X, f)$, if $(G, X, f)$
has a $k$-precoloring extension.

We denote by $P_t$ the path with $t$ vertices.  
Given a path $P$, its \emph{interior} is the set of vertices that
have degree two in $P$; the interior of $P$ is denoted by $P^*$.
A  \emph{path in a graph  $G$} is  a sequence $v_1-\ldots -v_t$ of pairwise 
distinct vertices where for $i,j \in [t]$, $v_i$ is adjacent to $v_j$ if and 
only if $|i-j|=1$; the {\em length} of this path is $t$.  We denote by 
$V(P)$ the set $\{v_1, \ldots, v_t\}$,
and if $a, b \in V(P)$, say $a=v_i$ and $b=v_j$ and $i<j$, then $a-P-b$ is the 
path $v_i-v_{i+1}-\ldots -v_j$, and $b-P-a$ is the path $v_j-v_{j-1}-\ldots-v_i$.
 For $v \in V(P)$, the
{\em neighbors of $v$ in $P$} are the neighbors of $v$ in $G|V(P)$.
A {\em $P_t$ in $G$} is a path
of length $t$ in $G$. A graph is \emph{$P_t$-free} if there is no $P_t$ in 
$G$.

Throughout the paper by ``polynomial time'' or ``polynomial size'' we mean 
running time,  or size, that are polynomial in $|V(G)|$, where $G$ is the input graph.
Since the  \textsc{$k$-coloring problem} and the \textsc{$k$-precoloring extension problem} are $NP$-hard for $k \geq 3$, their
restrictions to graphs with a forbidden induced subgraph have been extensively 
studied; see  \cite{ICM,gps} for a survey of known results. In particular,
the following is known (given a graph $H$, we say that a graph $G$ is 
{\em $H$-free} if no  induced subgraph of $G$ is isomorphic to $H$):

\begin{theorem}[\cite{gps}] Let $H$ be a (fixed) graph, and let $k>2$. If
the \textsc{$k$-coloring problem} can be solved in polynomial time when restricted to the class of $H$-free graphs, then every connected component of $H$ is a path.
\end{theorem}

Thus if we assume that $H$ is connected, then the question of determining the 
complexity of $k$-coloring $H$-free graphs is reduced to 
studying the complexity of coloring graphs with 
certain induced paths excluded,
and a significant body of work has been produced on this topic.
Below we list a few such results.

\begin{theorem}[\cite{c1}] \label{3colP7}
The \textsc{3-coloring problem} can be
  solved in polynomial time for the class of $P_7$-free graphs.
\end{theorem}

\begin{theorem}[\cite{hoang}] The \textsc{$k$-coloring problem} can be
  solved in polynomial time for the class of $P_5$-free graphs.
\end{theorem}

\begin{theorem}[\cite{huang}] The \textsc{4-coloring problem} is
  $NP$-complete for the class of $P_7$-free graphs.
\end{theorem}

\begin{theorem}[\cite{huang}] For all $k \geq 5$, the
  \textsc{$k$-coloring problem} is $NP$-complete for the class of
  $P_6$-free graphs.
\end{theorem}

The only cases for which the complexity of $k$-coloring $P_t$-free
graphs is not known are $k=4$, $t=6$, and $k=3$, $t \geq 8$.
This is the second paper in a series of two.
The main result of the series is the following:
\begin{theorem} 
\label{main}
The \textsc{4-precoloring extension problem} can be
  solved in polynomial time for the class of $P_6$-free graphs.
\end{theorem}
Theorem~\ref{main}  proves  a conjecture of Huang \cite{huang}, thus resolving the former open case above,  and completes the 
classification of the complexity of the \textsc{$4$-coloring problem} for graphs with a connected forbidden induced subgraph.

A \emph{starred precoloring} is a $7$-tuple
$\sspc$ such that 
\begin{enumerate}[(A)]
\item $f: S \cup X_0 \rightarrow \sset{1,2,3,4}$ is a proper coloring
  of $G|(S \cup X_0)$;
\item $V(G) = S \cup X_0 \cup X  \cup Y \cup Y^*$;
\item $G|S$ is connected and no vertex in $V(G) \setminus S$ is
  complete to $S$;
\item every vertex in $Y$ has a neighbor in $S$; 
\item for every vertex $x\in X$, $|f(N(x) \cap S)| \geq 2$;
\item $Y$ is anticomplete to $Y^*$; 
\item no vertex in $X$ is mixed on a  component of $G|Y^*$;
\item for every component $C$ of $G|Y^*$, there is a vertex in
  $S \cup X_0 \cup X$ complete to  $V(C)$.
\end{enumerate}

A starred precoloring is {\em excellent} if $Y=\emptyset$;
we write it is as a $6$-tuple $(G,S,X_0,X,Y^*,f)$.
The set $S$ is called the \emph{seed} of the starred precoloring. 
We define $L_P(v) = L_{S, f}(v) = \sset{1,2,3,4} \setminus (f(N(v) \cap S))$. 
A \emph{precoloring extension} of a starred precoloring is a function $f' : V(G) \setminus (S \cup X_0) \rightarrow \sset{1,2,3,4}$ such that $f' \cup f$ is a proper coloring of $G$. 
The main result of the first paper of the series \cite{part1} is 

\begin{theorem}
\label{excellent}
For every positive integer $C$, there exists a polynomial-time algorithm with the following specifications.
\\
\\
{\bf Input:}  An excellent starred precoloring $P=\esspc$ of a 
$P_6$-free graph $G$ with $|S| \leq C$.
\\
\\
{\bf Output:} A precoloring extension of $P$ or a determination
that none exists.
\end{theorem}

In this paper, we reduce the  \textsc{4-precoloring extension problem}
for $P_6$-free graphs to the case handled by Theorem~\ref{excellent}.
Our main result is the following. 
\begin{theorem}
\label{Yaxioms}
There exists an integer $C>0$ and a polynomial-time algorithm with the
following specifications.
\\
\\
{\bf Input:} A 4-precoloring $(G,X_0,f)$ of a $P_6$-free graph $G$.
\\
\\
{\bf Output:} A collection $\mathcal{L}$ of excellent starred
precolorings of $G$ such that
\begin{enumerate}
\item $|\mathcal{L}| \leq |V(G)|^C$,
\item for every $(G',S',X_0',X',Y^*,f') \in \mathcal{L}$ 
\begin{itemize}
\item $|S'| \leq C$, 
\item $X_0 \subseteq S' \cup X_0'$,
\item $G'$ is an induced subgraph of $G$, and
\item $f'|_{X_0}=f|_{X_0}$.
\end{itemize}
\item if we know for every $P \in \mathcal{L}$ whether $P$ has a
  precoloring extension, then we can decide in polynomial time if
  $(G, X_0, f)$ has a 4-precoloring extension; and
\item given a precoloring extension for every $P \in \mathcal{L}$ such
  that $P$ has a precoloring extension, we can compute a 4-precoloring
  extension for $(G, X_0, f)$ in polynomial time, if one exists.
\end{enumerate}
\end{theorem}

Clearly together  Theorem \ref{Yaxioms} and Theorem \ref{excellent} imply 
Theorem \ref{main}, and, as an immediate corollary, we obtain that 
the \textsc{$4$-coloring problem} for $P_6$-free graphs is also solvable in
polynomial time. In contrast, the
\textsc{$4$-list coloring problem} restricted to $P_6$-free graphs is
$NP$-hard as proved by Golovach, Paulusma, and Song \cite{gps}.

The proof of Theorem~\ref{Yaxioms} consists of several steps. 
At each step we replace the problem that 
we are trying to solve by a polynomially sized collection of simpler problems,
where by ``simpler'' we mean ``closer to being an excellent starred 
precoloring''. The strategy at every step is to ``guess'' (by exhaustively enumerating) a bounded number of vertices that have certain key properties, and their colors, add these vertices to the seed, and show that the resulting precoloring is  better than the one we started with. In this process we make sure that 
the size of the seed remains bounded, so that we can apply Theorem~\ref{excellent}.

This paper is organized as follows. In Section~\ref{sec:def}, we
introduce a few helpful definitions and lemmas. In Section~\ref{sec:axioms}, we
transform an instance of the \textsc{4-precoloring extension problem}
into a polynomial number of subproblems, each of which is 
a starred precoloring that has additional properties. In particular,
we ensure that every component of $G|(Y \cup Y^*)$ that contains a vertex 
$v$ with $L_P(v)=[4]$  is completely contained in $Y^*$.
In Section~\ref{sec:axiomsy}, we replace each of the subproblems produced
in Section~\ref{sec:axioms} by yet another polynomially-sized collection of
problems and prove Theorem \ref{Yaxioms}. 

\subsection{Definitions} \label{sec:def}

For two functions $f, f'$ with $f : X \rightarrow \sset{1, \dots, k}$
and $f' : Y \rightarrow \sset{1, \dots, k}$ such that
$f|_{X \cap Y} \equiv f'|_{X \cap Y}$, we define their \emph{union}
$f \cup f' : X \cup Y \rightarrow \sset{1, \dots, k}$ as
$(f \cup f')(z) = f(z)$ if $z \in X$, and $(f \cup f')(z) = f'(z)$ if
$z \in Y$. For a set $X$ with $X=\{x\}$, we do not distinguish between
$X$ and $x$ (when there is no danger of confusion).

A \emph{seeded precoloring} of a graph $G$ is a 7-tuple $$P = \spc$$
such that
\begin{itemize}
\item the function $f: (S \cup X_0) \rightarrow \sset{1, 2, 3, 4}$ is
  a proper coloring of $G|(S \cup X_0)$;
\item $V(G) = S \cup X_0 \cup X \cup Y_0 \cup Y$; and 
\item $S, X_0, X, Y_0, Y$ are pairwise disjoint. 
\end{itemize}
The set $S$ is called the \emph{seed} of the seeded precoloring. A
\emph{precoloring extension} of a seeded precoloring $P$ is a
4-precoloring extension of $(G, S \cup X_0, f)$. For a seeded
precoloring $P$ and a collection $\mathcal{L}$ of seeded precolorings, we
say that $\mathcal{L}$ is an \emph{equivalent collection} for $P$ if $P$ has
a precoloring extension if and only if at least one of the seeded
precolorings in $\mathcal{L}$ does, and, given a precoloring extension of a
member of $\mathcal{L}$, we can construct a precoloring of $P$ in polynomial time.

Given a seeded precoloring $P = \spc$ and a seeded precoloring $P'$,
we say that $P'$ is a \emph{normal subcase} of $P$ if
$P' = (G \setminus Z, S', X_0', X', Y_0', Y', f')$ such that
\begin{itemize}
\item $Z \subseteq Y_0$; 
\item $G|S'$ connected and $S \subseteq S'$;
\item every vertex in $X_0' \cap Y_0$ has a neighbor in $S'$; 
\item $S \subseteq S' \subseteq  S \cup X \cup Y_0 \cup Y$;
\item  $X_0 \subseteq X_0' \subseteq  X_0 \cup X \cup Y_0 \cup Y$, and
\item there is a function $g: (S' \cup X_0') \setminus (S \cup X_0) \rightarrow \sset{1,2,3,4}$ such that $f' = f \cup g$. 
\end{itemize}

For a seeded precoloring $P = \spc$ of a graph $G$, we let
$L_P(v) = L_{S,f}(v) = f(v)$ for $v \in S \cup X_0$, and
$L_P(v) = L_{S,f}(v) = \sset{1, 2, 3, 4} \setminus f(N(v) \cap S)$
otherwise. For a set $Z \subseteq V(G)$ and list
$L \subseteq \sset{1,2,3,4}$, we let $Z_L$ denote the set of vertices
$v$ in $Z$ with $L_P(v)=L$.

We finish this section with two useful theorems.

\begin{lemma}
\label{lem:mixed}
Let $G$ be a graph and let $X \subseteq V(G)$ be connected.  If 
$v \in V(G) \setminus X$ is mixed on $X$, the there is an edge
$xy$ of $X$ such that $v$ is adjacent to $x$ and not to $y$.
\end{lemma}

\begin{proof}
Since $v$ is mixed on $X$, both the sets $N(v) \cap X$ and $X \setminus N(v)$
are non-empty. Now since $X$ is connected, there exist $x \in N(v) \cap X$
and $y \in X \setminus N(v)$ such that $x$ is adjacent to $y$, as required.
This proves Lemma~\ref{lem:mixed}.
\end{proof}

\begin{theorem}[\cite{edwards}]
\label{Edwards}
There is a polynomial time algorithm that  tests, for  graph $H$ and a list 
assignment $L$ with $|L(v)| \leq 2$ for every $v \in V(H)$, if
$(H,L)$ is colorable, and finds a coloring if one exists.
\end{theorem}

\section{Establishing the Axioms on $Y_0$} 

\label{sec:axioms} Given a $P_6$-free graph $G$ and a precoloring
$(G, A, f)$, our goal is to construct a polynomial number of seeded
precolorings $P = \spc$ satisfying the following axioms, and such that
if we can decide for each of them if it has a precoloring extension,
then we can decide if $(G, A, f)$ has a 4-precoloring extension, and
construct one if it exists.
\begin{enumerate}[(i)]
\item $G \setminus X_0$ is connected. \label{it:conn}
\item $S$ is connected and no vertex in $V(G) \setminus S$ is
  complete to $S$. \label{it:seed}
\item $Y_0 = V(G) \setminus (N(S) \cup X_0 \cup S)$. \label{it:y0}
\item No vertex $V(G) \setminus (Y_0 \cup X_0)$ is mixed on an edge of
  $Y_0$. \label{it:mixedy0}
\item If $|L_{S,f}(v)| = 1$ and $v \not\in S$, then $v \in X_0$; if
  $|L_{S,f}(v)| = 2$, then $v \in X$; if $|L_{S,f}(v)| = 3$, then $v \in Y$;
  and if $|L_{S,f}(v)| = 4$, then $v \in Y_0$. \label{it:lists}
\item There is a color $c \in \sset{1,2,3,4}$ such for every vertex
  $y \in Y$ with a neighbor in $Y_0$, $f(N(y) \cap S) =
  \sset{c}$. We let $L = \sset{1,2,3,4} \setminus \sset{c}$. \label{it:123star}
\item With $L$ as in \eqref{it:123star}, we let $Y_L^*$ be the subset
  of $Y_L$ of vertices that are in connected components of
  $G|(Y_0 \cup Y_L)$ containing a vertex of $Y_0$. Then no vertex of
  $Y \setminus Y_L^*$ has a neighbor in $Y_0 \cup Y_L^*$, and no vertex in $X$ 
is mixed  on an edge of $Y_0 \cup Y_L^*$. \label{it:mixedyl}
\item With $Y_L^*$ as in \eqref{it:mixedyl}, for every component $C$ of $G|(Y_0 \cup Y_L^*)$, there is a vertex $v$ in $X$ complete to $C$. \label{it:complete}
\end{enumerate}

We begin by establishing the first axiom. 
\begin{lemma} \label{lem:conn}
Given a 4-precoloring $(G, X_0, f)$ of a $P_6$-free
  graph $G$, there is an algorithm with running time $O(|V(G)|^{2})$
  that outputs a collection $\mathcal{L}$ of seeded precolorings such
  that:
\begin{itemize}
\item $|\mathcal{L}| \leq |V(G)|$;
\item every $P' \in \mathcal{L}$ is of the form
  $P' = (G|(V(C) \cup X_0), \emptyset, X_0, \emptyset, V(C), \emptyset, f)$ for a
component $C$ of $G \setminus X_0$;  
\item every $P' \in \mathcal{L}$ satisfies \eqref{it:conn} 
\item $(G, X_0, f)$ has a 4-precoloring extension if and only if each
  of the seeded precolorings $P' \in \mathcal{L}$ has a precoloring extension; and
\item given a precoloring extension for each of the seeded precolorings
  $P' \in \mathcal{L}$, we can compute a 4-precoloring extension for
  $(G, X_0, f)$ in polynomial time.
\end{itemize}
\end{lemma}
\begin{proof}
  For each connected component $C$ of $G \setminus X_0$, the algorithm
  outputs the seeded precoloring
  $(G|(V(C) \cup X_0), \emptyset, X_0, \emptyset, V(C), \emptyset,
  f)$.
  Since the coloring is fixed on $X_0$, it follows that $(G, X_0, f)$
  has a 4-precoloring extension if and only if the 4-precoloring on
  $X_0$ can be extended to every connected component $C$ of
  $G \setminus X_0$. This implies the statement of the lemma.
\end{proof}

The next lemma is used to arrange the following axioms, which we
restate:
\begin{enumerate}
\item[\eqref{it:seed}] $S$ is connected and no vertex in
  $V(G) \setminus S$ is complete to $S$.
\item[\eqref{it:y0}] $Y_0 = V(G) \setminus (N(S) \cup X_0 \cup S)$.
\end{enumerate}

\begin{lemma}
\label{lem:seedy0}
  There is a constant $C$ such that the following holds. Let
  $P = (G, \emptyset, X_0, \emptyset, Y_0, \emptyset, f)$ be a seeded
  precoloring of a $P_6$-free graph $G$ with $P$ satisfying
  \eqref{it:conn}. Then there is an algorithm with running time
  $O(|V(G)|^{C})$ that outputs an equivalent collection $\mathcal{L}$
  for $P$ such that
\begin{itemize}
\item $|\mathcal{L}| \leq |V(G)|^{C}$;
\item every $P' \in \mathcal{L}$ is a normal subcase of $G$;
\item every $P' = (G', S', X_0', X', Y_0', Y', f') \in \mathcal{L}$
  with seed $S'$ satisfies $|S'| \leq C$; and
\item every $P' \in \mathcal{L}$ satisfies \eqref{it:conn},
  \eqref{it:seed} and \eqref{it:y0}.
\end{itemize}
Moreover, for every $P' \in \mathcal{L}$, given a precoloring
extension of $P'$, we can compute a precoloring extension for $P$ in
polynomial time.
\end{lemma}
\begin{proof}
  If $|V(G) \setminus X_0| \leq 5$, we enumerate all possible colorings. Now let
  $v \in V(G) \setminus X_0$, and let $S' = \sset{v}$. While there is
  a vertex $w$ in $V(G) \setminus S'$ complete to $S'$, we add $w$ to
  $S'$. Let $S$ denote the set $S'$ when this procedure terminates.
  If either $|S| \geq 5$ or
  $(G|(S \cup X_0), \emptyset, X_0, S, \emptyset, \emptyset, f)$ has
  no precoloring extension, then we output that $P$ has no precoloring
  extension. Otherwise, we construct $\mathcal{L}$ as follows. For
  every proper coloring $f'$ of $G|S$ such that $f \cup f'$ is a
  proper coloring of $G|(S \cup X_0)$, we add
  $$P' = (G, S, X_0 \setminus S, N(S) \setminus X_0, V(G) \setminus
  (X_0 \cup S \cup N(S)), \emptyset, f \cup f')$$
  to $\mathcal{L}$. Since $|S| \leq 4$, it follows that the first
  three bullets hold, and \eqref{it:y0} holds for $P'$ by the
  definition of $P'$. Since $X_0$ is unchanged, it follows that
  \eqref{it:conn} holds. Since $S$ is a maximal clique, we have that
  \eqref{it:seed} holds for $P'$. This concludes the proof.
\end{proof}

The next four lemmas are technical tools that we use several times in the 
course of the proof. They are used to show that if we start with a seeded 
precoloring that has certain properties, and then move to its normal subcase, 
then these properties are preserved (or at least can be restored with a simple 
modification).

For a seeded precoloring $P=\spc$, a {\em type} is a subset of $S$.
For  $v \in V(G) \setminus (S \cup X_0)$, 
the \emph{type of $v$}, denoted by
$T_P(v) = T_S(v)$, is
$N(v) \cap S$. For a
type $T$ and a set $A$, we let $A(T) = \sset{v \in A: T_P(v) = T}$.

\begin{lemma}
\label{types}
Let $P=\spc$ be a seeded precoloring of a $P_6$-free graph $G$ satisfying
\eqref{it:seed} and \eqref{it:y0}, and let
Let $T,T' \subseteq S$ with $|f(T)|=|f(T')|=1$ and such that $f(T) \neq f(T')$.
Let $y,y' \in N(Y_0)$ such that $T(y)=T$ and $T(y')=T'$. Let $z,z' \in Y_0$
be such that $yz$ and $y'z'$ are edges, and suppose that $z$ is non-adjacent to $z'$ and that $y$ is non-adjacent to $y'$. Then either $yz'$ or $y'z$ is an edge.
\end{lemma}

\begin{proof}
Suppose both the pairs  $yz'$ and $y'z$ are non-adjacent. Since $P$ satisfies
\eqref{it:seed} and \eqref{it:y0}, it follows that $G|S$ is connected
and both $y,y'$ have neighbors in $S$. 
Let $Q$ be a shortest path from $y$ to $y'$ with interior in $S$.
Since $|f(T)|=|f(T')=1$ and $f(T) \neq f(T')$, it follows that 
$T \cap T' = \emptyset$, and so $|Q^*|>1$. But now $z-y-Q-y'-z'$ is 
a path of length at least six in $G$, a contradiction. This proves
Lemma~\ref{types}. 
\end{proof}

\begin{lemma}
\label{sublem:s'Y0}
Let $P=\spc$ be a seeded precoloring of a $P_6$-free graph satisfying
\eqref{it:seed}, \eqref{it:y0} and \eqref{it:mixedy0}, and let 
$P' = (G', S', X_0', X', Y_0', Y', f')$ be a 
normal subcase of $P$ satisfying \eqref{it:y0}. 
Then no $v \in Y_0 \setminus (S' \cup Y_0')$ has both a neighbor 
in $S'$ and a neighbor in $Y_0'$.
\end{lemma}

\begin{proof}
Suppose such $v$ exists.  Let $y \in Y_0$ be a neighbor of $v$.
Since $P'$ is a normal subcase of $P$, 
$P'$ satisfies \eqref{it:seed}.
Since $v$ has both a neighbor in $Y_0'$ and 
a neighbor in $S'$, and since $P'$ satisfies \eqref{it:y0}, it follows that 
$v \in X' \cup Y' \cup X_0'$.  
Since $v \in Y_0$, it follows that $v$ is anticomplete to 
$S$. Therefore $v$ has a 
neighbor in  $S' \setminus S \subseteq X \cup Y \cup Y_0$.
Since $P'$ satisfies \eqref{it:seed}, there is  a path $Q$ from $v$ to a 
vertex  $s$ of  $S$ with $Q^* \subseteq S'$. Then $V(Q) \setminus \{v\}$ is 
anticomplete to $Y_0'$.  Let  $R$ be the maximal 
subpath of $v-Q-s$, with  $v \in V(R)$, such that $V(R) \subseteq Y_0$. Then 
$s \not \in V(R)$, and there is a unique vertex $t \in V(Q) \setminus V(R)$ 
with a neighbor in $V(R)$. Since $t \in N(Y_0)$, it follows that 
$t \not \in S \cup Y_0$, and so 
$t \in X \cup Y$.
 But  $t$ is mixed on $V(R) \cup \{y\} \subseteq Y_0$,
contrary to the fact that $P$ satisfies \eqref{it:mixedy0}.
This proves Lemma~\ref{sublem:s'Y0}.

\end{proof}

\begin{lemma} 
  \label{sublem:it:conn} 
There is a constant $C$ such that the following holds. Let $P = \spc$ be a seeded precoloring of a $P_6$-free graph $G$ with $P$ satisfying  \eqref{it:conn},
and let $P' = (G', S', X_0', X', Y_0', Y', f')$ be a normal subcase of $P$
satisfying \eqref{it:y0} and \eqref{it:mixedy0}. 
Then there is an algorithm with running time $O(|V(G)|^{C})$ that outputs 
an equivalent collection $\mathcal{L}$ for $P'$, such that
$|\mathcal{L}| \leq 1$, and if $\mathcal{L}=\{P''\}$, then
\begin{itemize}
\item there is $Z \subseteq Y_0'$ such that
$P''=(G' \setminus Z, S',X_0',Y_0' \setminus Z, Y',f)$ and
$P''$ is a normal subcase of $P'$;
\item $P''$ satisfies \eqref{it:conn}---\eqref{it:mixedy0};
\item if $P'$ satisfies \eqref{it:lists}, then $P''$ satisfies \eqref{it:lists}.
\end{itemize}
Moreover, given a precoloring extension of $P''$, we can compute a precoloring extension for $P$ in polynomial time.
\end{lemma}

\begin{proof}
   Since $P'$ is a normal subcase of
$P$, it follows that $P'$ satisfies \eqref{it:seed}.
We may assume that $P'$ does not satisfy \eqref{it:conn}, for otherwise we can set $\mathcal{L}=\{P'\}$.
  Now let $C$ be a connected component of
  $G' \setminus X_0'$ with $S' \cap V(C) = \emptyset$. 
It follows that $V(C) \subseteq Y_0'$ and $C$ is a
component of $G|Y_0'$.

Let  $x \in N(V(C)) \cap (X_0' \setminus X_0)$. Since $P$ satisfies
  \eqref{it:conn}, such a vertex $x$ exists.
By  Lemma~\ref{sublem:s'Y0},
$x \in X \cup Y$. Since $P'$ satisfies \eqref{it:mixedy0},
it follows from Lemma~\ref{lem:mixed} that 
$x$ is complete to $V(C)$. Let $f'(x)=c$.
Then in every precoloring extension $d$ of $P'$ we have $d(v) \neq c$
for every $v \in V(C)$.

 Let $A = \sset{v \in X_0': f'(v) \neq c}$. By 
Theorem~\ref{3colP7}
  and since $G$ is $P_6$-free, we can decide in polynomial time if
  $(G'|(V(C) \cup A), A, f'|_A)$ has a precoloring extension
  with colors in $\sset{1,2,3,4} \setminus \sset{c}$. If not, then $P'$
  has no precoloring extension, and we set $\mathcal{L}=\emptyset$. If
  $(G'|(V(C) \cup A), A, f'|_A)$ has a precoloring
  extension using only colors in $\sset{1,2,3,4} \setminus \sset{c}$, then
  $P'$ has a precoloring extension  if and only if
  $(G' \setminus V(C), S',X_0',X',Y_0' \setminus V(C), Y',f')$   does. 

We repeat this process a polynomial number of times until
  $G' \setminus X_0'$ is connected, and output the resulting
  seeded precoloring $P''=(G'',S',X_0',X',Y_0'',Y',f')$ 
satisfying \eqref{it:conn}.  Since $Y_0'' \subseteq Y_0'$, and the
other sets of $P''$ remain the same as in $P'$, it follows that the
$P''$ satisfies \eqref{it:seed}--\eqref{it:mixedy0}, and if
$P'$ satisfies \eqref{it:lists}, then so does $P''$. 
This proves Lemma~\ref{sublem:it:conn}.
\end{proof}

\begin{lemma}
\label{sublem:it:mixedyl}
Let $P=\spc$ be a seeded precoloring of a $P_6$-free graph satisfying
\eqref{it:seed}, \eqref{it:y0} and \eqref{it:mixedy0}, and let 
$P' = (G', S', X_0', X', Y_0', Y', f')$ be a 
normal subcase of $P$ satisfying \eqref{it:y0}. Then  
$P'$  satisfies  \eqref{it:mixedy0}.
Moreover, if $P$ satisfies \eqref{it:123star}, then $P'$  satisfies  
\eqref{it:123star}.
\end{lemma}

\begin{proof}
Since $P'$ is a normal subcase of $P$, 
$P'$ satisfies \eqref{it:seed}.
First we show that $P'$ satisfies \eqref{it:mixedy0}.
Suppose not, then there
exists $v \in V(G) \setminus X_0'$ mixed on an edge $xy$ of $Y_0'$,
say $v$ is adjacent to $y$ and not to $x$.  It follows 
that $v \in X' \cup Y'$, and since $P$ satisfies \eqref{it:mixedy0}, 
$v \in Y_0$.   Therefore $v$ has a neighbor in 
$S'$, contrary to Lemma~\ref{sublem:s'Y0}. This proves that $P'$ satisfies
\eqref{it:mixedy0}.

Next assume that $P$ satisfies \eqref{it:123star}.
We show that $P'$ satisfies \eqref{it:123star}.
Let $L$ as in \eqref{it:123star} applied to  $P$.
Suppose there exists 
$y \in N(Y_0')$ with $L_{P'}(y) \neq L$ and $|L_{P'}(y)|=3$. 
Since $P$ satisfies \eqref{it:123star}, it follows that 
$y \in Y_0 \setminus Y_0'$,
and $y$ has a neighbor $s \in S'$, contrary to Lemma~\ref{sublem:s'Y0}. 
This proves that $P'$ satisfies
\eqref{it:123star}.

This completes the proof of Lemma~\ref{sublem:it:mixedyl}.

\end{proof}

The next lemma is another technical tool,  used to establish axioms
\eqref{it:mixedy0} and \eqref{it:mixedyl}.

\begin{lemma} \label{sublem:it:mixedy0}
There is a function $q: \mathbb{N} \rightarrow \mathbb{N}$ such that the following holds. Let $P = \spc$ be a seeded precoloring of a $P_6$-free graph $G$ with $P$ satisfying \eqref{it:conn}, \eqref{it:seed} and \eqref{it:y0}. Let  
$L \subseteq [4]$ with $|L|=3$, let $c_4$ be the unique element of
$[4] \setminus L$. Let $R \subseteq Y_0 \cup Y_L$ such that
$Y_0 \subseteq R$.
Assume further that if 
$t \in  (X \cup Y) \setminus R$ has a 
neighbor in $R$, then for every $z \in R$, $L_{P}(t) \neq L_{P}(z)$, 
and that there is no path $t-z_1-z_2-z_3$ with $t \in (X \cup Y) \setminus R$
and $z_1,z_2,z_3 \in R$.
Then there is an algorithm with running time $O(|V(G)|^{q(|S|)})$ that outputs an equivalent collection $\mathcal{L}$ for $P$ such that
\begin{itemize}
\item $|\mathcal{L}| \leq |V(G)|^{q(|S|)}$;
\item every $P' \in \mathcal{L}$ is a normal subcase of $P$;
\item every $P' = (G', S', X_0', X', Y_0', Y', f') \in \mathcal{L}$ with seed $S'$ satisfies $|S'| \leq q(|S|)$;  
\item every $P' \in \mathcal{L}$ satisfies \eqref{it:seed} and \eqref{it:y0}.
\item no vertex of $(X' \cup Y') \setminus R$ is mixed on an edge of 
$(Y' \cup Y_0') \cap R$.
\end{itemize}
Moreover, for every $P' \in \mathcal{L}$, given a precoloring extension of $P'$, we can compute a precoloring extension of $P$ in polynomial time. 
\end{lemma}

\begin{proof}
If $G$ contains a $K_5$, then $P$ has no precoloring extension; 
we output $\mathcal{L}=\emptyset$ and stop. Thus from now on we assume that
$G$ has no clique of size five.
Let $Y_0^{5}=R$ and let $Z^{5}=(X \cup Y) \setminus R$.
Let $\mathcal{T}^{5}$ be the set of types of vertices in $Z^{5}$, and set $j=4$.

Let $\mathcal{Q}_{j}$ be the set of $|\mathcal{T}^j|$-tuples
$(S^{j,T})_{T \in \mathcal{T}^{j+1}}$, where each $S^{j,T}  \subseteq Z^{j+1}(T)$
and $S^{j,T}$ is constructed as follows (starting with $S^{j,T}=\emptyset$):
\begin{itemize}
\item If $R=Y_0$ or  $c_4 \in f(T)$ proceed as follows. 
While there is a vertex $z \in Z^{j+1}(T)$ complete to $S^{j,T}$ and such that
there is  clique
$\sset{a_1, \dots, a_j} \subseteq {Y_0}^{j+1}$
with $N(z) \cap \sset{a_1, \dots, a_j} = \sset{a_1}$,
choose such $z$ with $N(z) \cap R$  maximal 
and add it to $S^{j,T}$.
\item If $R \neq Y_0$ and  $c_4 \not \in f(T)$,
while there is $z \in Z^{j+1}(T)$ complete to 
$S^{j,T}$  such that there is  clique
$\sset{a_1, \dots, a_j} \subseteq {Y_0}^{j+1}$
with $N(z) \cap \sset{a_1, \dots, a_j} = \sset{a_1}$,
add $z$ to $S^{j,T}$.
Let $X_0(z)$ be the set of all  $z' \in Z^{j+1}(T)$ such that
\begin{itemize}
\item $z'$ is complete to $S^{j,T} \setminus \{z\}$
\item there is a clique
$\sset{b_1, \dots, b_j} \subseteq {Y_0}^{j+1}$
such that $N(z') \cap \sset{b_1, \dots, b_j} = \sset{b_1}$, and  
\item $N(z') \cap R$ is a proper subset of $N(z) \cap R$.
\end{itemize}
When no such vertex $z$ exists, let 
$X_0^{j,T}=\bigcup_{z \in S^{j,T}}X_0(z)$.
Define $f'(z')=c_4$ for every $z' \in X_0^{j,T}$ (observe that since $c_4 \not\in f(T)$, it follows that $c_4 \in L_P(z')$).
\end{itemize}

Since $G$ has no clique of size five, it follows that $|S^{j,T}| \leq 4$ for all
$T$. Let $Q \in \mathcal{Q}_j$; write $Q=(S^{j,T})_{T \in \mathcal{T}^{j+1}}$.
Let $S^j=S^{j,Q}=\bigcup_{T \in \mathcal{T}^{j+1}}S^{j,T}$.
Let   $Y_0^{j} = Y_0^{j,Q}=Y_0^{j+1} \setminus N(S^{j})$,
$X_0^{j}=X_0^{j,Q}=\bigcup_{T \in \mathcal{T}^{j+1}}X_0^{j,T}$.
$Z^{j}=Z^{j,Q}=(Z^{j+1} \setminus X_0^{j}) \cup (Y_0^{j+1} \setminus Y_0^{j})$
and let $\mathcal{T}^{j}$ be the set of types of $Z^{j}$ (in $P$). 
If $j>2$, decrease $j$ by $1$ and repeat the construction 
above, to obtain a new set $\mathcal{Q}_{j-1}$; repeat this  for each 
$Q \in \mathcal{Q}_j$.

Suppose $j=2$. Then $Q$ was constructed by fixing
$Q_4 \in \mathcal{Q}_4$, constructing $\mathcal{Q}_3$ (with $Q_4$ fixed),
fixing $Q_3 \in \mathcal{Q}_3$, constructing $\mathcal{Q}_2$ (with
$Q_3$ fixed), and finally fixing  $Q \in \mathcal{Q}_2$. Write $Q_2=Q$. 
For consistency of notation we write $Q_5=\emptyset$, $Z^5=Z^{5,Q_5}$
and $Y_0^5=Y_0^{5,Q_5}$.
Let $S'=S \cup \bigcup_{j=2}^4 S^{j,Q_j}$. If $R \neq Y_0$,  
let $X_0'=X_0 \cup \bigcup_{j=2}^4 X_0^{j,Q_j}$; if $R=Y_0$, let $X_0'=X_0$.

For every function $f': S' \setminus S \rightarrow \sset{1,2,3,4}$ such 
that  $f \cup f'$ is a proper coloring of $G|(S' \cup X_0')$, let
 $$P_{f',Q} = (G, S', X_0', Z^{2,Q} \cap X, Y_0^{2,Q}, Z^{2,Q} \cap Y, f \cup f').$$ 
Let $\mathcal{L}$ be the set of all $P_{Q,f'}$ as above.
Observe that $S'$ is obtained from $S$ by adding a clique of size at most four 
for  each type in $\mathcal{T}^j$
at each of the  three steps
  ($j=4,3,2$), and since $|\mathcal{T}^j| \leq 2^{|S|}$ for
every $j$, it follows that  $|S \cup S'| \leq |S|+12 \times 2^{|S|}$. Since
  $|S' \setminus S| \leq 12 \times 2^{|S|}$, it follows that 
$|\mathcal{L}| \leq (4|V(G)|)^{12 \times 2^{|S|}}$.

In the remainder of the proof we show that every $P_{Q,f'} \in \mathcal{L}$ 
satisfies the required properties. 

\vspace*{-0.4cm}
  \begin{equation}\vspace*{-0.4cm} \label{eq:connectedseed}
    \longbox{\emph{$S \cup \bigcup_{k=j}^4S^k$ is connected for 
every $j \in \{2, \ldots, 4\}$. In particular $S'$ is connected. }}
  \end{equation}

Since for every $j$, we have that $S^{j,Q_j} \subseteq Z^{j+1}$, it follows that
every vertex of $S^{j,Q_j}$  has a neighbor in $S \cup \bigcup_{k=j+1}^4S^{k,Q_k}$,
and \eqref{eq:connectedseed} follows.

\bigskip

  \vspace*{-0.4cm}\begin{equation}\vspace*{-0.4cm} \label{eq:no3pathk4}
    \longbox{\emph{Let $j \in \{2, \ldots, 5\}$.
There is no path $z-a-b-c$ with $z \in Z^{j,Q_j}$ and
$a, b, c \in Y_0^{j,Q_j}$.}}
  \end{equation}

  Suppose for a contradiction that there exist $j$ and $z$ violating
\eqref{eq:no3pathk4}; we may assume $z$ is chosen with $j$ maximum.
By assumption $j\neq 5$ and $z \in   Y_0^{j,Q_j} \setminus Y_0^{j+1,Q_{j+1}}$.
It follows that  $z$ has a neighbor $z' \in S^{j,Q_j}$ and that
$z$ is anticomplete to $S \cup \bigcup_{k=j+1}^4S^{k,Q_k}$. 
Since $z' \in S^{j,Q_j} \subseteq Z^{j+1,Q_j}$, it follows that 
$z'$ has a neighbor $s \in S \cup \bigcup_{k=j+1}^4S^{k,Q_k}$.   
But now $s-z'-z-a-b-c$ is a $P_6$ in $G$, a
  contradiction. This proves \eqref{eq:no3pathk4}.
  
  \vspace*{-0.4cm}\begin{equation}\vspace*{-0.4cm} \label{eq:mixedk4}
  \longbox{\emph{Let $j \in \{2, \ldots, 4\}$.
No vertex $z \in Z^{j,Q_j}$  has exactly one neighbor in a clique
  		$\sset{a_1, \dots, a_j} \subseteq Y_0^{j,Q_j}$.}} 
  \end{equation}
  
  Suppose for a contradiction that there exist $j$ and $z$ violating
\eqref{eq:mixedk4}; we may assume that $z$ is chosen with $j$ maximum.
Write $Q_j=(S^{j,T})$.
Let   $\sset{a_1, \dots, a_j} \subseteq Y_0^{j,Q_j}$ be a clique
  with $N(z) \cap \sset{a_1, \dots, a_j} = \sset{a_1}$.
  
  Suppose first that $z \in R$. Let $k$ be maximum such that
$z \in Z^{k,Q_k}$. Then 
$z \not \in Z^{k+1,Q_{k+1}}$, and thus $z \in Y_0^{k+1,Q_{k+1}}$,  
  $z$ has a neighbor $z' \in S^{k,Q_k}$, and $z$ is 
anticomplete to $S \cup \bigcup_{l=k+1}^4S^{l,Q_l}$. It follows that
$z' \in Z^{k+1,Q_{k+1}}$. But now
  $z'-z-a_1-a_j$ is  a path with $z, a_1, a_j \in Y_0^{k+1,Q_{k+1}}$
  contrary to \eqref{eq:no3pathk4}. This proves that $z \not \in R$.

It follows that 
$z \in Z^{j+1,Q_{j+1}} \cap (X \cup Y)$, and in particular $z$ has a neighbor
in $S$.  Let $T=T_P(z)$. It follows
that $S^{j,T} \neq \emptyset$; let $z' \in S^{j,T}$ be the 
first vertex that was added to $S^{j,T}$ that is non-adjacent to $z$
(such a vertex exists by the definition of $S^{j,T}$). 
Then $L_P(z)=L_P(z')$. Since $z' \in S^{j,Q_j}$, it follows that $z'$
is anticomplete to $Y_0^{j,Q_j}$. 
Since $a_1 \in Y_0^{j,Q_j} \subseteq Y_0^{j+1,Q_{j+1}}$, it follows that 
$z$ has a
  neighbor in $Y_0^{j+1,Q_{j+1}}$ non-adjacent to $z'$, and hence (by the choice
  of $z'$ if $Y_0=R$, and since $z \not \in X_0(z')$ if $Y_0 \neq R$),  it follows 
that 
$z'$ has a neighbor $a' \in Y_0^{j+1}$ that is
  non-adjacent to $z$. 
  
  Suppose first that $a'$ is complete to $\sset{a_1, \dots,
    a_j}$. Since $G$ contains no clique of size five, it follows that
  $j<4$.  But now $N(z') \cap \sset{a', a_1, \dots, a_j} = \sset{a'}$,
  contrary to the maximality of $j$.

  Suppose next that $a'$ is mixed on $\sset{a_1, \dots, a_j}$. Let $x$
  be a neighbor and $y$ be a non-neighbor of $a'$ in
  $\sset{a_1, \dots, a_j}$. Then $z'-a'-x-y$ is a path, which
  contradicts an assumption of the theorem.
  
  It follows that $a'$ is anticomplete to $\sset{a_1, \dots, a_j}$.
  Since $z,z' \not \in R$ and have neighbors in $R$, it follows
that there is a vertex $t \in T$ that is anticomplete to $R$
(this is immediate if $R=Y_0$, and follows from the fact that
$L_P(z) \neq L$ if $R \neq Y_0$). Now
 $a'-z'-t-z-a_1-a_j$ is  a $P_6$  in
  $G$, a contradiction. This proves~\eqref{eq:mixedk4}.
  
\bigskip

By \eqref{eq:connectedseed} $P_{f',Q}$ satisfied \eqref{it:seed},
and by construction \eqref{it:y0} holds.
Now from \eqref{eq:mixedk4} with $j=2$ we deduce that no vertex of 
$(X' \cup Y') \setminus R$ is mixed on an edge of $(Y' \cup Y_0') \cap R$.

It remains to show that $\mathcal{L}$ is equivalent to $P$.
Clearly for every $P' \in \mathcal{L}$,
a precoloring extension of $P'$ is also a precoloring extension of $P$.

Let $d$ be a precoloring extension of $P$. We show that some $P' \in \mathcal{L}$ has a precoloring extension.
Let $j \in \{2,3,4\}$; define $S^{j,T}$ and $f'$ as follows
(starting with $S^{j,T}=\emptyset$):
\begin{itemize}
\item If $R=Y_0$ or $c_4 \in f(T)$ proceed as follows. 
While there is a vertex $z \in Z^{j+1}(T)$ complete to $S^{j,T}$ and such that
there is  clique
$\sset{a_1, \dots, a_j} \subseteq {Y_0}^{j+1}$
with $N(z) \cap \sset{a_1, \dots, a_j} = \sset{a_1}$,
choose such $z$ is such that $N(z) \cap R$  maximal 
and add it to $S^{j,T}$; set $f'(z)=d(z)$.
\item If $R \neq Y_0$ and  $c_4 \not \in f(T)$,
while there is $z \in Z^{j+1}(T)$ complete to 
$S^{j,T}$  such that there is  clique
$\sset{a_1, \dots, a_j} \subseteq {Y_0}^{j+1}$
with $N(z) \cap \sset{a_1, \dots, a_j} = \sset{a_1}$,
choose such $z$ with $d(z) \neq c_4$ and subject to that
with $N(z) \cap R$ maximal;
add $z$ to $S^{j,T}$ and set $f'(z)=d(z)$.
Let $X_0(z)$ be the set of all  $z' \in Z^{j+1}(T)$ such that
\begin{itemize}
\item $z'$ is complete to $S^{j,T} \setminus \{z\}$,
\item there is a clique
$\sset{b_1, \dots, b_j} \subseteq {Y_0}^{j+1}$
such that $N(z') \cap \sset{b_1, \dots, b_j} = \sset{b_1}$, and
\item  $N(z') \cap R$ is a proper subset of $N(z) \cap R$.
\end{itemize}
It follows from the choice of $z$ that   $d(z')=c_4$ for every 
$z' \in X_0(z)$.
When no such vertex $z$ exists, let 
$X_0^{j,T}=\bigcup_{z \in S^{j,T}}X_0(z)$; thus $d(z')=c_4$ for every $z' \in X_0^{j,T}$.
Define ${f'}_{j,T}(z')=c_4$ for every $z' \in X_0^{j,T}$, then ${f'}_{j,T}(z)=d(z)$ for
every $z \in X_0^{j,T}$.
\end{itemize}

Let $Q_j=(S^{j,T})$ and let ${f'}_j=\bigcup_{T}{f'}_{j,T}$.
It follows that 
$P_{f_2,Q_2}=(G,S',X_0',X',Y_0', Y', f \cup f')$ satisfies 
$d(v)=f_2(v)$ for every $v \in S' \cup X_0'$, and thus $d$ is
is a precoloring extension of $P_{f_2,q_2}$, as required.
This proves Lemma~\ref{sublem:it:mixedy0}.
\end{proof}

The next lemma is used to arrange the following axiom, which we restate: 
\begin{enumerate}
\item[\eqref{it:mixedy0}] No vertex $V(G) \setminus (Y_0 \cup X_0)$ is
  mixed on an edge of $Y_0$.
\end{enumerate}
\begin{lemma} \label{lem:it:y0}
There is a function $q: \mathbb{N} \rightarrow \mathbb{N}$ such that the following holds. Let $P = \spc$ be a seeded precoloring of a $P_6$-free graph $G$ with $P$ satisfying \eqref{it:conn}, \eqref{it:seed} and \eqref{it:y0}. Then there is an algorithm with running time $O(|V(G)|^{q(|S|)})$ that outputs an equivalent collection $\mathcal{L}$ for $P$ such that
\begin{itemize}
\item $|\mathcal{L}| \leq |V(G)|^{q(|S|)}$;
\item every $P' \in \mathcal{L}$ is a normal subcase of $P$;
\item every $P' = (G', S', X_0', X', Y_0', Y', f') \in \mathcal{L}$ with seed $S'$ satisfies $|S'| \leq q(|S|)$;  and
\item every $P' \in \mathcal{L}$ satisfies \eqref{it:conn},
  \eqref{it:seed}, \eqref{it:y0} and \eqref{it:mixedy0}.
\end{itemize}
Moreover, for every $P' \in \mathcal{L}$, given a precoloring extension of $P'$, we can compute a precoloring extension for $P$ in polynomial time. 
\end{lemma}
\begin{proof}
  Let $S^5 = \emptyset$. Let $Z=X \cup Y$. Since $P$ satisfies \eqref{it:y0},
it follows that every vertex of $Z$ has a neighbor in $S$.
While there is a vertex $z \in Z$
  complete to $S^5$ and a path $z-a-b-c$ with
  $a, b, c \in Y_0$, we add $z$ to $S^5$. If $|S^5| \geq 5$, then $G$
  contains a $K_5$ and thus it has no precoloring extension; set
$\mathcal{L}=\emptyset$ and stop. Thus we
  may assume that $|S^5| \leq 4$. Let 
$Y_0^5 = Y_0 \setminus N(S^5)$ and let $Z^5=Z \cup (Y_0 \setminus Y_0^5)$.
Since $S$ is connected, and since every vertex of $S^5$ has a neighbor in
$S$, it follows that $S \cup S^5$ is connected.

  \vspace*{-0.4cm}
  \begin{equation}\vspace*{-0.4cm} \label{eq:no3path}
    \longbox{\emph{There is no path $z-a-b-c$ with
$z \in Z^5$ and   $a, b, c \in Y_0^5$.}}
  \end{equation}

  Suppose for a contradiction that such a path exists, and
  suppose first that $z \in Z$.  By the choice of $S^5$, it follows that
  there exists a vertex $z' \in Z \cap S^5$ non-adjacent to $z$. Since
  $S \cup S^5$ is connected, there exists a path $Q$ connecting $z$
  and $z'$ with interior in $S \cup S^5$. Since $P$ satisfies \eqref{it:y0}
and by the construction of $S^5$, it follows that $Q^*$ is anticomplete
to $\{a,b,c\}$.
But now  $z'-Q-z-a-b-c$ is a path of length at least six in $G$, a contradiction.

  It follows that $z \in N(S^5) \setminus Z$, and thus
  $z \in Y_0 \setminus Y_0^5$. Let $s' \in S^5 \cap N(z)$. Then
$s'$ is anticomplete to $\{a,b,c\}$. Moreover,  $s' \in Z$, and so $s'$ has a 
neighbor $s \in S$. 
Since $P$ satisfies \eqref{it:y0},
$s$ is anticomplete to $Y_0$, and so $s$ is anticomplete to $\{z,a,b,c\}$.
But now 
  $s-s'-z-a-b-c$ is a $P_6$ in $G$, a contradiction. This
  proves \eqref{eq:no3path}.

  \bigskip

For every $f':S^5 \rightarrow [4]$ such that $f \cup f'$ is a proper
coloring of $G|(S \cup S^5)$, let
 $P_{f'}=(G,S \cup S^5, X_0, Z^5,Y_0^5,\emptyset, f \cup f')$.
Then $P_{f'}$ is a normal subcase of $P$ that 
satisfies \eqref{it:conn}-\eqref{it:y0}.

Let $\mathcal{M}_{f'}$ be the collection of seeded precolorings obtained
by applying Lemma~\ref{sublem:it:mixedy0} to $P_f'$ with $R=Y_0^5$, and let
$\mathcal{M}$ be the union of all such $\mathcal{M}_{f'}$.
By \eqref{eq:no3path} every $P'' \in \mathcal{M}$ satisfies
\eqref{it:seed}--\eqref{it:mixedy0}.

Finally let $\mathcal{L}$ be obtained from $\mathcal{M}$ by applying
Lemma~\ref{sublem:it:conn} to every member of $\mathcal{M}$.
Then every $P' \in \mathcal{L}$ satisfies \eqref{it:conn}--\eqref{it:mixedy0},
as required.
This proves Lemma~\ref{lem:it:y0}. 
\end{proof}

The purpose of Lemma~\ref{lem:lists} is to organize vertices according
to their lists (which, in turn, arise from the colors of their
neighbors in the seed) to satisfy the following axiom:
\begin{enumerate}
\item[\eqref{it:lists}] If $|L_{S,f}(v)| = 1$ and $v \not\in S$, then
  $v \in X_0$; if $|L_{S,f}(v)| = 2$, then $v \in X$; if $|L_{S,f}(v)| = 3$,
  then $v \in Y$; and if $|L_{S,f}(v)| = 4$, then $v \in Y_0$.
\end{enumerate}
Moreover, we will construct new seeded
precolorings in controlled ways from seeded precolorings satisfying
\eqref{it:conn}, \eqref{it:seed}, \eqref{it:y0}, and \eqref{it:mixedy0}, to arrange that these axioms as well as \eqref{it:lists} still hold for the new
instances.

\begin{lemma} 
  \label{lem:lists} \label{lem:preserve}
There is a constant $C$ such that the following holds. Let $P = \spc$ be a seeded precoloring of a $P_6$-free graph $G$ with $P$ satisfying  \eqref{it:conn},
  \eqref{it:seed}, \eqref{it:y0} and \eqref{it:mixedy0}, and let $P' = (G', S', X_0', X', Y_0', Y', f')$ be a normal subcase of $P$.  
Then there is an algorithm with running time $O(|V(G)|^{C})$ that outputs an equivalent  collection 
$\mathcal{L}$ for $\{P\}$ of seeded precoloring with $|\mathcal{L}| \leq 1$, 
such that if $\mathcal{L}=\{P''\}$ then  
\begin{itemize}
\item $P''$ is a normal subcase of $P'$, and 
\item $P''$ satisfies \eqref{it:conn}, \eqref{it:seed}, \eqref{it:y0},
  \eqref{it:mixedy0} and \eqref{it:lists}.
\item If $P'$ \eqref{it:123star}, then $P''$ 
satisfies  \eqref{it:123star}.
\item If $P'$ satisfies \eqref{it:mixedyl}, then $P''$ satisfies
\eqref{it:mixedyl}. 
\end{itemize}
Moreover, given a precoloring extension of $P''$, we can compute a precoloring extension for $P$ in polynomial time.
\end{lemma}

\begin{proof}
  Since $P'$ is a normal subcase of $P$, it follows that $P'$ satisfies
  \eqref{it:seed}. By moving vertices between $Y_0'$ and $X' \cup Y'$,
we may assume that $P'$ satisfies \eqref{it:y0}.
By Lemma~\ref{sublem:it:mixedyl}
$P'$ satisfies \eqref{it:mixedy0}.

  Let
  $Z_i = \sset{v \in V(G) \setminus (S' \cup X_0'): |L_{P'}(v)| = i}$.
  If $Z_0 \neq \emptyset$, then $P'$ has no precoloring extension, and
  we output this and $\mathcal{L}=\emptyset$ and stop. Thus, we may assume that $Z_0 = \emptyset$.
Let
  $f'' : Z_1 \rightarrow \sset{1,2,3,4}$ with $f''(v) = c$ if
  $L_{P'}(v) = \sset{c}$. 
Since $P'$ satisfies \eqref{it:y0}, it follows that $Y_0' = Z_4$, and 
so the seeded precoloring $\tilde{P}=(G', S', X_0' \cup Z_1, Z_2, Z_4, Z_3, f' \cup f'')$
satisfies \eqref{it:mixedy0}. For the same reason,
if $P'$ satisfies \eqref{it:123star},
then so does $\tilde{P}$, and if $P'$ satisfies \eqref{it:mixedyl},
then so does $\tilde{P}$.
Let
  $P''$  be obtained from the precoloring $\tilde{P}$
as in Lemma~\ref{sublem:it:conn}. It
  follows that $P''$ satisfies \eqref{it:conn}--\eqref{it:lists}, and $P''$ is 
a normal  subcase of $P'$. Clearly if $\tilde{P}$ satisfies \eqref{it:123star},
then so does $P''$, and if $\tilde{P}$ satisfies \eqref{it:mixedyl},
then so does $P''$.
 This proves Lemma~\ref{lem:lists}.
\end{proof}

In the next lemma we  establish \eqref{it:123star},  which we
restate:
\begin{enumerate}
\item[\eqref{it:123star}] There is a color $c \in \sset{1,2,3,4}$ such for
  every vertex $y \in Y$ with a neighbor in $Y_0$,
  $f(N(y) \cap S) = \sset{c}$. We let $L = \sset{1,2,3,4} \setminus \sset{c}$. 
\end{enumerate}

\begin{lemma} \label{lem:123star}
There is a function $q: \mathbb{N} \rightarrow \mathbb{N}$ such that the following holds. Let $P = \spc$ be a seeded precoloring of a $P_6$-free graph $G$ with $P$ satisfying \eqref{it:conn},
  \eqref{it:seed}, \eqref{it:y0}, \eqref{it:mixedy0} and
  \eqref{it:lists}. Then there is an algorithm with running time $O(|V(G)|^{q(|S|)})$ that outputs an equivalent collection $\mathcal{L}$ for $P$ such that
\begin{itemize}
\item $|\mathcal{L}| \leq |V(G)|^{q(|S|)}$;
\item every $P' \in \mathcal{L}$ is a normal subcase of $P$;
\item every $P' = (G', S', X_0', X', Y_0', Y', f') \in \mathcal{L}$ with seed $S'$ satisfies $|S'| \leq q(|S|)$;  and
\item every $P' \in \mathcal{L}$ satisfies \eqref{it:conn},
  \eqref{it:seed}, \eqref{it:y0}, \eqref{it:mixedy0}, \eqref{it:lists} and \eqref{it:123star}. 
\end{itemize}
Moreover, for every $P' \in \mathcal{L}$, given a precoloring extension of $P'$, we can compute a precoloring extension for $P$ in polynomial time. 
\end{lemma}

\begin{proof}
A seeded precoloring $P=\spc$ is {\em acceptable} if for
every precoloring extension $c$ of $P$ and for every 
non-adjacent  $y, y' \in Y \cap N(Y_0)$  with
$L_P(y) \neq L_P(y')$, we have  
$\{c(y), c(y')\} \not \subseteq L_P(y) \cap L_P(y')$.

First we construct a collection $\mathcal{M}$ of seeded precolorings that
is an equivalent collection for  $P$, and such that every member of $\mathcal{M}$
is acceptable.
We proceed as follows. 
Let $\mathcal{T}$ be the set of all pairs $(T,T')$ with
$T,T' \subseteq S$ and $|f(T)|=|f(T')|=1$ and $f(T) \neq f(T')$.
Write $\mathcal{T}=\{(T_1,T_1'), \ldots, (T_{t},T_{t}')\}$. 
Let $\mathcal{Q}$ be the set
of all $t$-tuples $Q=(Q_{T_1,T_1'}, \ldots, Q_{T_{t},T_{t}'})$ such that
$Q_{T_i,T_i'}=(P_{T_i,T_i'},M_{T_i,T_i'},N_{T_i,T_i'})$ where
\begin{itemize} 
\item $|P_{T_i,T_i'}|= |M_{T_i,T_i'}| \leq |N_{T_i,T_i'}| \leq 1$.
\item $P_{T_i,T_i'} \subseteq Y(T_i)$ and  
$N_{T_i,T_i'} \subseteq Y(T_i')$.
\item $M_{T_i,T_i'} \subseteq Y_0$.
\item $M_{T_i,T_i'}$ is complete to $P_{T_i,T_i'} \cup N_{T_i,T_i'}$.
\item $P_{T_i,T_i'}$ is anticomplete to $N_{T_i,T_i'}$.
\end{itemize}
Let $V(Q_{T_i,T_i'})=P_{T_i,T_i'} \cup M_{T_i,T_i'} \cup N_{T_i,T_i'}$ and let
$S(Q)=\bigcup_{i=1}^{t}V(Q_{T_i,T_i'})$. 
Let $(T_i,T_i') \in \mathcal{T}$. Define $Z(T_i,T_i')$ as follows.
\begin{itemize}
\item If $|P_{T_i,T_i'}|=|M_{T_i,T_i'}|=|N_{T_i,T_i'}| =0$ ,  then
$Z(T_i,T_i')=Y({T'}_i) \cap N(Y_0)$. 
\item If  $|P_{T_i,T_i'}|=|M_{T_i,T_i'}|=0$ and $|N_{T_i,T_i'}| =1$,  then 
$Z(T_i,T_i')=(Y({T'}_i) \cap N(Y_0)) \setminus N(N_{T_i,T_i'})$.
\item If  $|P_{T_i,T_i'}|=|M_{T_i,T_i'}|=|N_{T_i,T_i'}| =1$, then
  $Z(T_i,T_i')=\emptyset$.
\end{itemize}
Let $Z(Q)=\bigcup_{(T_i,T_i') \in \mathcal{T}}Z(T_i,T_i')$.
A function $f'$ 
is said to be {\em $Q$-admissible} if
$f':S(Q) \cup Z(Q) \rightarrow \{1, \ldots, 4\}$  and  for every 
$i \in \{1, \ldots, t\}$ it satisfies:
\begin{itemize}
\item  $f'(P_{T_i,T_i'}),f'(N_{T_i,T_i'}) \in [4] \setminus (f(T_i) \cup  f(T_i'))$. 
\item  If $Z(T_i,T_i') \subseteq Y(T_i')$, then $f'(Z(T_i,T_i'))=f(T_i)$.
\item  If $Z(T_i,T_i') \subseteq Y(T_i)$, then $f'(Z(T_i,T_i'))=f(T_i')$.
\item The coloring $f \cup f'$ of $G|(S \cup S(Q) \cup X_0 \cup Z(Q))$ is proper. 
\end{itemize}
For every $Q$-admissible function $f'$  with domain $S(Q) \cup Z(Q)$, let
$$P_{Q,f'} = (G, S \cup S(Q), X_0 \cup Z(Q), X, Y_0 \setminus (S(Q) \cup N(S(Q))), (Y \setminus (S(Q) \cup Z(Q))) \cup (N(S(Q)) \cap Y_0), f \cup f').$$
Then $P_{Q,f'}$ is a normal subcase of $P$.

Since every vertex in $X \cup Y$ has a neighbor in $S$, it follows that 
$P_{Q,f'}$ satisfies \eqref{it:seed}; by construction \eqref{it:y0} 
holds. By Lemma~\ref{sublem:it:mixedyl},  $P_{Q,f'}$ satisfies \eqref{it:mixedy0}.
Let $\mathcal{M}$ be the union of the collections obtained by
applying Lemma~\ref{lem:lists}, where the union is taken over all $Q,f'$ as 
above. Then every member of $\mathcal{M}$ satisfies
\eqref{it:conn}--\eqref{it:lists}.  

We show that there is a function $q_1: \mathbb{N} \rightarrow \mathbb{N}$
such that $|S \cup S(Q)| \leq q_1(|S|)$ and 
$|\mathcal{M}| \leq |V(G)|^{q_1(|S|)}$.
Since there are at most $2^{|S|}$ types,
it follows that $t \leq 2^{2|S|}$. 
Now, since for every 
$(T_i,T_i') \in \mathcal{T}$ we have that $|V(Q_{T_i,T_i'})| \leq 3$, 
it follows that for every $Q \in \mathcal{Q}$ we have 
$|S(Q)| \leq 3 \times 2^{t}$, and so 
$|S \cup S(Q)| \leq |S|+3 \times 2^{2|S|}$ and 
$|\mathcal{Q}| \leq |V(G)|^{3 \times 2^{2|S|}}$.
Finally,  for every $Q$, there are at most $4^{|S(Q)|}=4^{3t}$ possible 
precoloring of $S(Q)$, 
since every precoloring of $S(Q)$ extends to an admissible function
in a unique way, and   we deduce that 
$|\mathcal{M}| \leq 4^{3t} \times |\mathcal{Q}| \leq 
4^{3 \times 2^{2|S|}} \times |V(G)|^{3 \times 2^{2|S|}} \leq (4|V(G)|)^{3 \times 2^{2|S|}} $ 
as required.

  \vspace*{-0.4cm}
  \begin{equation}\vspace*{-0.4cm} \label{yinY}
    \longbox{\emph{Let $P' \in \mathcal{M}$ with $P'= (G, S', X_0', X', Y_0', Y', f')$. If $y \in Y'$ has a neighbor $z \in Y_0'$, then $y \in Y$.}}
  \end{equation}

Suppose that $y \not \in Y$. Then $y \in Y_0 \cap Y'$ and  there exist 
$s \in S' \setminus S$
such that $y$ is adjacent to $s$, contrary to Lemma~\ref{sublem:s'Y0}.
This proves \eqref{yinY}.

\bigskip

Next we show that every precoloring in $\mathcal{M}$ is acceptable.
Let $P' = (G, S', X_0', X', Y_0', Y', f') \in \mathcal{M}$,
and suppose there exist non-adjacent $y,y' \in N(Y_0') \cap Y'$
with $L_{P'}(y) \neq L_{P'}(y')$ 
and such that there exists a precoloring extension 
$c$ with $c(y), c(y') \in L_{P'}(y) \cap L_{P'}(y')$.
Let $z \in N(y) \cap Y_0'$ and $z' \in N(y') \cap Y_0'$.
Then $z,z' \in Y_0$, and so by Lemma~\ref{sublem:s'Y0}, $y,y' \in Y$,
$L_P(y)=L_{P'}(y)$ and  $L_P(y')=L_{P'}(y')$.
Let $T=T(y)$ and $T'=T(y')$ (in $P$). Then $T \cap T'=\emptyset$.
By Lemma~\ref{types} we may assume that $z=z'$.
Since $Y' \cap Y(T)$ and $Y' \cap Y(T')$ are both non-empty,
it follows that $|V(Q_{T,T'})|>1$.
Let $P_{T,T'}=\{p\}$, $M_{T,T}=\{m\}$ and $N_{T,T'}=\{n\}$. Since $z \in Y_0'$,
it follows that $z$ is anticomplete to $V(Q_{T,T'})$. Since 
$P'$ satisfies \eqref{it:lists}, 
$f'(p),f'(n) \in L_P(y) \cap L_P(y')$ and 
$|L_{P'}(y)|=|L_{P'}(y')|=3$, it follows that
$\{y,y',p,n\}$ is a stable set.
By symmetry, we may assume that $f'(m) \in L_{P'}(y)$, and hence $y$ is not adjacent to $m$. Let $s \in T \setminus T'$; then $z-y-s-p-m-n$ is a $P_6$ in $G$, a contradiction. This proves that every seeded precoloring in $\mathcal{M}$ 
is acceptable.

Next we show that $\mathcal{M}$ is equivalent to $P$. Clearly every precoloring 
extension of a member of $\mathcal{M}$ is a precoloring extension of $P$.
For the converse, let $c$ be a precoloring extension of $P$. For every
pair of types $(T,T') \in \mathcal{T}$ for which there
exist non-adjacent $y \in Y(T) \cap N(Y_0)$ and $y' \in Y(T') \cap N(Y_0)$,
such that $c(y), c(y') \not \in f(T) \cup f(T')$, choose such a pair $y,y'$ and 
let $z$ be a common neighbor of $y,y'$ in $Y_0$ (such $z$ exists by 
Lemma~\ref{types}); set
$P_{T,T'}=\{y\}$, $M_{T,T'}=\{z\}$ and $N_{T,T'}=\{y'\}$, and
define $f'(y)=c(y)$, $f'(y')=c(y')$ and $f'(z)=c(z)$.
Let $Z(T_i,T_i')=\emptyset$.

Now let $(T,T') \in \mathcal{T}$ be such that no such $y,y'$ exist.
Suppose that there exists $y \in Y(T') \cap N(Y_0)$ with $c(y) \neq f(T)$,
let $N_{T,T'}=\{y\}$, $P_{T,T'}=M_{T,T'}=\emptyset$, and let $f'(y)=c(y)$. 
Let $Z(T_i,T_i')= (Y(T) \cap N(Y_0)) \setminus N(y)$, and 
set $f'(v)=f(T')$ for every $v \in  Z(T_i,T_i')$.
Since $(T,T')$ does not
have the property described in the previous paragraph, it follows
that $c((Y(T) \cap N(Y_0)) \setminus N(y))=f(T')$, and so
$c(v)=f'(v)$ for every $v \in  Z(T_i,T_i')$.
Finally, suppose that $c( Y(T') \cap N(Y_0)) = f(T)$. Then
set  $P_{T,T'}=M_{T,T'}=N_{T,T'}=\emptyset$ and
$Z(T_i,T_i')= Y(T') \cap N(Y_0)$.
Define $f'(v)=f(T)$ for every $v \in Z(T_i,T_i')$.
Let $Q$ consist of all the triples $Q_{T,T'}=(P_{T,T'},M_{T,T'},N_{T,T'})$ as above.
Let $S(Q)=\bigcup_{(T,T') \in \mathcal{T}}V(Q_{T,T'})$, and 
$Z(Q)=\bigcup_{(T,T') \in \mathcal{T}}Z(T_i,T_i')$. Let
$$P_{Q,f'} = (G, S \cup S(Q), X_0 \cup Z(Q), X, Y_0 \setminus (S(Q) \cup N(S(Q))), (Y \setminus (S(Q) \cup Z(Q))) \cup (N(S(Q)) \cap Y_0), f \cup f').$$
Then $c$ is a precoloring extension of $P_{Q,f'}$. Moreover, 
$P_{Q,f'}$ was one of the seeded precoloring we considered in the process of
constructing $\mathcal{M}$, and so 
$\mathcal{M}$ contains the seeded precoloring obtained from
$P_{Q,f'}$ by applying Lemma~\ref{lem:lists}. It follows that $\mathcal{M}$
is an equivalent collection for $P$.

Let $P'=(G,S',X_0',X',Y_0',Y',f') \in \mathcal{M}$ be an acceptable seeded 
precoloring. 
For $c \in \sset{1,2,3,4}$ and a precoloring extension $d$ of $P'$,
we say that is \emph{$c$ is active for $L$ and $d$}  
if there exists a vertex $v \in Y' \cap N(Y_0')$ with $L_{P'}(v) = L$ and  
$d(v)=c$.

Define $\mathcal{L}_1(P')$ as follows. For every function
$g: Y' \cap N(Y_0') \rightarrow [4]$ such that
\begin{itemize}
\item $g(v) \in L_{P'}(v)$ for every $v \in Y \cap N(Y_0')$,
\item $|g(Y'_L \cap N(Y_0')|=1$ for every $L \in {[4] \choose 3}$, and
\item $f' \cup g$ is a proper coloring of $G|(S' \cup X_0' \cup  (Y' \cap N(Y_0')))$,
\end{itemize}
let 
$$P''_g= (G, S', X_0' \cup  (Y' \cap N(Y_0')), X', Y_0', Y' \setminus N(Y_0'), f' \cup g).$$
It is easy to check that $P''_g$ satisfies \eqref{it:seed}---\eqref{it:123star}.
Let $P'_g$ be obtained from $P''_g$ by applying Lemma~\ref{lem:lists}. 
Then $P'_g$ satisfies \eqref{it:conn}---\eqref{it:123star}.
Let  $\mathcal{L}_1(P')$ be the collections of all such $P'_g$.

Next we construct $\mathcal{L}_2(P')$.
For every $L \in  {[4] \choose 3}$, for every $y_1,y_2 \in {Y'}_L \cap N({Y'}_0)$, and for every $c_1, c_2 \in L$, define a function $g$ as follows.
Let $g(y_i)=c_i$. For every $L' \in {[4] \choose 3} \setminus  L$, 
let $Z(L')$ be the set of vertices $v \in {Y'}_{L'}$ such that
$v$ has a non-neighbor $n \in \{y_1,y_2\}$ with $g(n) \in L'$.
For every $v \in Z(L')$, let 
$g(v)$ be the unique element of $L' \setminus L$. Finally, let 
$Z=\bigcup_{L' \in {[4] \choose 3} \setminus L}  Z(L')$.

If $f' \cup g$ is 
a proper coloring of $G|(S \cup X_0 \cup \{y_1,y_2\})$, let
$$P''_{L,y_1,y_2,c_1,c_2}=(G,S \cup \{y_1,y_2\}, X_0 \cup Z, X, Y_0 \setminus N(\{y_1,y_2\}), Y \setminus (Z \cup \{y_1,y_2\}), f' \cup g).$$
It is easy to check that 
$P''_{L,y_1,y_2,c_1,c_2}$  satisfies \eqref{it:conn}---\eqref{it:123star}. 
Let $P'_{L,y_1,y_2,c_1,c_2}$ be obtained from $P''_{L,y_1,y_2,c_1,c_2}$ by applying
Lemma~\ref{lem:lists}.
Let $\mathcal{L}_2(P')$ be the collection of all  
$P'_{L,y_1,y_2,c_1,c_2}$ constructed this way; then every member of $\mathcal{L}_2$
satisfies \eqref{it:conn}---\eqref{it:123star}.

We claim that $\mathcal{L}(P')=\mathcal{L}_1(P') \cup \mathcal {L}_2(P')$ is an equivalent
collection for $\{P'\}$. Clearly a precoloring extension of an element of 
 $\mathcal{L}(P')$ is a precoloring extension of $P$.  Now let $c$ be a
precoloring extension of $P$. If for every $L \in  {[4] \choose 3}$
there is at most one active color for $L$ and $c$, then $c$ is a precoloring
extension of a member of $\mathcal{L}_1(P)$, so we may assume that
there is $L_0 \in  {[4] \choose 3}$ such that at least two colors
are active for $L$ and $c$. We may assume that $L=\{1,2,3\}$ and
the colors $1,2$ are active. Let $y_i \in {Y'}_{L_0}$ with $c(y)=i$.
We claim that $c$ is a precoloring extension of $P''_{L_0,y_1,y_2,1,2}$.
Let $L \in  {[4] \choose 3} \setminus L_0$. Since $P'$ is acceptable,
for every  $v \in Y'_L$ that
has  a non-neighbor $n \in \{y_1,y_2\}$ with $c(n) \in L'$, we have that
$c(v) \in L' \setminus L_0$. It follows that $c(v)=g(v)$, and the claim holds.
This proves that  $\mathcal{L}(P')$ is an equivalent collection for $\{P'\}$.

Finally, setting 
$$\mathcal{L}=\bigcup_{P' \in \mathcal{M}}\mathcal{L}(P'),$$ 
Lemma~\ref{lem:123star} follows. This completes the proof.
\end{proof}

The next lemma is used to arrange the following axiom, which we restate: 
\begin{enumerate}
\item[\eqref{it:mixedyl}] With $L$ as in \eqref{it:123star}, we let
  $Y_L^*$ be the subset of $Y_L$ of vertices that are in connected
  components of $G|(Y_0 \cup Y_L)$ containing a vertex of $Y_0$. Then
  no vertex of $Y \setminus Y_L^*$ has a neighbor in
  $Y_0 \cup Y_L^*$, and no vertex of $X$ is mixed on $Y_0 \cup Y_L^*$.
\end{enumerate}
\begin{lemma} \label{lem:mixedyl}
There is a function $q: \mathbb{N} \rightarrow \mathbb{N}$ such that the following holds. Let $P = \spc$ be a seeded precoloring of a $P_6$-free graph $G$ with $P$ satisfying 
\eqref{it:conn}, \eqref{it:seed}, \eqref{it:y0}, \eqref{it:mixedy0}, \eqref{it:lists} and \eqref{it:123star}. Then there is an algorithm with running time $O(|V(G)|^{q(|S|)})$ that outputs an equivalent collection $\mathcal{L}$ for $P$ such that
\begin{itemize}
\item $|\mathcal{L}| \leq |V(G)|^{q(|S|)}$;
\item every $P' \in \mathcal{L}$ is a normal subcase of $P$;
\item for every $P' = (G', S', X_0', X', Y_0', Y', f') \in \mathcal{L}$, 
$|S'| \leq q(|S|)$;  
\item every $P' \in \mathcal{L}$ satisfies 
\eqref{it:conn}, \eqref{it:seed}, \eqref{it:y0}, \eqref{it:mixedy0}, \eqref{it:lists}, \eqref{it:123star} and \eqref{it:mixedyl};
\end{itemize}
Moreover, for every $P' \in \mathcal{L}$, given a precoloring extension of $P'$, we can compute a precoloring extension for $P$ in polynomial time. 
\end{lemma}
\begin{proof}
We may assume that $G$ contains no $K_5$, for otherwise, $P$ does not have a precoloring extension and we output $\mathcal{L}=\emptyset$ and stop. 

With $L$ as in \eqref{it:123star} and $Y^*_L$ as in \eqref{it:mixedyl},
let $Y^* =  (X \cup (Y \setminus Y_L^*))  \cap N(Y_0 \cup Y_L^*)$. 
By the definition of $Y_L^*$, it follows that $L_P(y) \neq L$ for every 
$y \in Y^*$, and if $y \in Y^* \cap Y$, then $y$ is anticomplete to $Y_0$.
Let $\mathcal{T}=\{T_1, \ldots, T_t\}$ be the set of types of vertices in 
$Y^*$.  Let $L = \sset{c_1, c_2, c_3}$ and 
$\sset{c_4} = \sset{1,2,3,4} \setminus L$.
Let $\mathcal{Q}$ consist of all $t$-tuples 
$Q=((S_{T_1},R_{T_1}), \ldots, (S_{T_i},R_{T_t}))$ such that
\begin{itemize}
\item $|R_{T_i}| \leq |S_{T_i}| \leq 1$.
\item $S_{T_i} \cup R_{T_i} \subseteq Y^*(T_i)$.
\item $S_{T_i}$ is complete to $R_{T_i}$.
\end{itemize}
Let $V(Q)=\bigcup_{i=1}^t(S_{T_i} \cup R_{T_i})$.
For every $Q \in \mathcal{Q}$ and for every 
$f': V(Q) \rightarrow L$ with $f'(v) \in L_P(v) \setminus \{c_4\}$ for all $v \in V(Q)$, we proceed as follows.  
Let $\tilde{Y}^1_{Q,f'}$ be the set of all vertices $v$ in $Y^*$ such that 
$S_{T(v)}=\emptyset$. Let  $\tilde{Y}^2_{Q,f'}$ be the set of all vertices $v$ in $Y^*$ such that  $S_{T(v)} \neq \emptyset$ ,  $R_{T(v)}= \emptyset$ and
$v$ is complete to $S_{T(v)}$.
Let $\tilde{Y}_{Q,f'}=\tilde{Y}^1_{Q,f'} \cup \tilde{Y}^2_{Q,f'}$.
Let $f'(v)=c_4$ for every $v \in \tilde{Y}_{Q,f'}$
Since $V(Q) \subseteq Y^*$, it follows that  $G|(S \cup V(Q))$ is connected.
Suppose that $f \cup f'$ is a proper coloring of 
$G|(S \cup X_0 \cup V(Q) \cup \tilde{Y}_{Q,f'})$.
Let $\mathcal{L}'$ be obtained from the normal subcase 
$$(G, S \cup V(Q), X_0 \cup \tilde{Y}_{Q,f'}, X \setminus (\tilde{Y}_{Q,f'} \cup V(Q)), Y_0, Y \setminus (\tilde{Y}_{Q,f'} \cup V(Q)), f \cup f' \cup g)$$
of $P$ by applying Lemma~\ref{lem:lists}. 
Suppose that $\mathcal{L}'=\{P_{Q,f'}\}$. 
Write $P_{Q,f'} = (G, S', X_0', X', Y_0', Y', f')$. 
Then $P_{Q,f'}$ satisfies \eqref{it:conn}--\eqref{it:123star}. 
Furthermore, $P_{Q,f'}$  has a precoloring extension if and only if $P$ has a precoloring extension $d$ such that $d(v)=f'(v)$ for every $v \in V(Q)$,
and $d(v)=c_4$ for every $v \in Y^*$ 
such that  either
\begin{itemize}
\item $S_{T(v)}=\emptyset$, or
\item $S_{T(v)} \neq \emptyset$, $R_{T(v)}=\emptyset$, and $v$ is complete to 
$S_{T(v)}$.
\end{itemize}
Moreover, $|V(Q)| \leq 2|\mathcal{T}| \leq 2^{|S|+1}$.

Let $\mathcal{L}_1$ be the set of all seeded precolorings $P_{Q,f'}$ as above
(ranging over all $Q \in \mathcal{Q}$).
Then $\mathcal{L}_1$ is an equivalent collection for $P$, and 
$|\mathcal{L}_1| \leq (3|V(G)|)^{2^{|S|+1}}$.
Let $P' \in \mathcal{L}_1$ with $P' = (G, S', X_0', X', Y_0', Y', f')$. 
Since $P'$ satisfies \eqref{it:123star},
let $L$ be as in \eqref{it:123star} and let ${Y'}_L^*$ be as in 
\eqref{it:mixedyl}.

\vspace*{-0.4cm}\begin{equation}\vspace*{-0.4cm} \label{eq:no3pathyl}
  \longbox{\emph{There is no path $z-a-b-c$ with  
$z \in (X' \cup Y') \setminus {Y'}_L^*$ 
and $a, b, c \in {Y'}_L^* \cup Y_0'$.}}
\end{equation}

Suppose that such a path $z-a-b-c$ exists. First we show that $z \in X \cup Y$.
Suppose not, then $z \in Y_0$ and $z$ has a neighbor $s' \in S' \setminus S$.
Since $P$ satisfies \eqref{it:123star}, it follows that $s' \in X$.
Since $\{z,a,b,c\} \subseteq Y_0 \cup Y_L$, and since
$P$ satisfies \eqref{it:lists}, we deduce that there exists
$s \in T(s')$ with $f(s) \in L$. Consequently, $s$ is anticomplete to 
$\{z,a,b,c\}$. But now $s-s'-z-a-b-c$ is a $P_6$ in $G$, a contradiction.
This proves that $z \in X \cup Y$.

Since  $L_{S,f}(z) \neq L$, there exists $t \in T(z)$ with
$f(t) \in L$.  Since $z \not \in X_0'$, it follows that 
$S_{T(v)} \neq \emptyset$, and 
either
\begin{itemize}
\item $R_{T(z)} \neq \emptyset$ , or
\item $R_{T(z)}=\emptyset$, and $z$ is not complete to $S_{T(z)}$.
\end{itemize}
Let $S_{T(z)}=\{s\}$. Since $f'(s) \in L$, it follows that $s$ is anticomplete
to $\{a,b,c\}$. If $z$ is non-adjacent to $s$, then $s-t-z-a-b-c$ is a 
$P_6$, a contradiction. It follows that $z$ is adjacent to $s$, and therefore
$R_{T(z)}\neq \emptyset$; say $R_{T(z)}=\{r\}$. Since $s$ is adjacent
to $r$, it follows that $f'(z) \neq f'(r)$. Since $z \not \in X_0$, and
since \eqref{it:lists} holds, it follows that $z$ is non-adjacent to $r$.
Since $f'(r) \in L$, it follows that $r$ is anticomplete
to $\{a,b,c\}$. But now $r-s-z-a-b-c$ is a $P_6$ in $G$, a contradiction.
This proves~\eqref{eq:no3pathyl}. 

\bigskip

In view of \eqref{eq:no3pathyl}, let $\mathcal{L}_2(P')$ be the collection
of precolorings obtained from $P'$ by applying Lemma~\ref{sublem:it:mixedy0} with 
$R=Y_0' \cup {Y'}_L^*$. Let $P'' \in \mathcal{L}_2(P')$; write
 $P''=(G, S'',X_0'',X'',Y_0'',Y'',f'')$.
Then $P''$ satisfies \eqref{it:seed} and \eqref{it:y0} and  no vertex of 
$(X'' \cup Y'') \setminus R$ is mixed on $(Y'' \cup Y_0'') \cap R$.
By Lemma~\ref{sublem:it:mixedyl},  $P''$ satisfies \eqref{it:mixedy0}
and \eqref{it:123star}.

Let $\mathcal{L}_3(P'')$ be obtained by applying Lemma~\ref{lem:lists}
to  $P''$, and let $\tilde{P} \in \mathcal{L}_3(P'')$.
Write $\tilde{P}=(\tilde{G},\tilde{S},\tilde{X_0},\tilde{X}, \tilde{Y_0},\tilde{Y}, \tilde{f})$.
By  Lemma~\ref{lem:lists},  $\tilde{P}$
satisfies \eqref{it:conn}--\eqref{it:123star}. 
Since $P''$ satisfies \eqref{it:y0}, $\tilde{S}=S''$ and $\tilde{Y_0}=Y_0''$.
Define $\tilde{Y}_L^*$ as in \eqref{it:mixedyl}, then 
$\tilde{Y}_L^* =R \cap \tilde{Y}$.
Since no vertex of $(X'' \cup Y'') \setminus R$ is mixed on 
$(Y'' \cup Y_0'') \cap R$,
it follows that no vertex of $(\tilde{X} \cup \tilde{Y}) \setminus \tilde{Y}_L^*$
is mixed on $Y_0'' \cup \tilde{Y}_L^*$, and since $\tilde{P}$ satisfies
\eqref{it:123star}, we deduce that $\tilde{P}$ satisfies \eqref{it:mixedyl}.
Now setting

$$\mathcal{L}=\bigcup_{P_1 \in \mathcal{L}_1} \bigcup_{P_2 \in \mathcal{L}_2(P_1)}\mathcal{L}_3(P_2)$$

Lemma~\ref{lem:mixedyl} follows. 
\end{proof}

We are now ready to prove the final lemma of this section, used to prove the following axiom, which we restate: 
\begin{enumerate}
\item[\eqref{it:complete}] With $Y_L^*$ as in \eqref{it:mixedyl}, for every component $C$ of $G|(Y_0 \cup Y_L^*)$, there is a vertex $v$ in $X$ complete to $C$.
\end{enumerate}
\begin{lemma} 
\label{lem:complete}
There is a constant $c$ such that the following holds. Let $P = (G, S, X_0, X, Y_0,Y, f)$ be a seeded precoloring of a $P_6$-free graph $G$ satisfying \eqref{it:conn}, \eqref{it:seed}, \eqref{it:y0}, \eqref{it:mixedy0}, \eqref{it:lists}, \eqref{it:123star}, and \eqref{it:mixedyl}. 
Let $L$ be as in \eqref{it:123star} and 
let $Y_L^*$ as in \eqref{it:mixedyl}. 
There is an algorithm with running time $O(|V(G)|^c)$ that 
outputs an equivalent collection $\mathcal{L}$ of seeded precolorings,
such that $|\mathcal{L}| \leq 1$, and if $\mathcal{L}=\{P'\}$, then
\begin{itemize}
\item there is $Z \subseteq Y_0 \cup Y_L^*$ such that 
$P'=(G\setminus Z, S, X_0,X,Y_0 \setminus Z, Y \setminus Z, f)$, and
\item $P'$ satisfies \eqref{it:conn}--\eqref{it:complete}.
\end{itemize}
\end{lemma}

\begin{proof}
We may assume that $P$ does not satisfy \eqref{it:complete} for otherwise
we set $\mathcal{L}=\{P\}$. A component
$C$  of $G|(Y_0 \cup Y_L^*)$ is \emph{deficient} if  
no vertex of $X$ is complete to $V(C)$. 
Let $C$ be a deficient component.
It follows from
\eqref{it:mixedyl} that $X$ is anticomplete to $V(C)$.
Let $A = V(C) \cap Y_0$, $B = V(C) \setminus A$. For every vertex $v \in A \cup B$, let $L(v) = \sset{1,2,3,4} \setminus (f(N(v) \cap (S \cup X_0)))$. It follows that $L(v) \subseteq L$ for $v \in B$. Moreover, by \eqref{it:conn}, it follows that $B \neq \emptyset$. Let $L = \sset{c_1, c_2, c_3}$ and let $\sset{c_4} = \sset{1,2,3,4} \setminus L$. 

For every component $D$ of $G|A$, we proceed as follows. 

Let $\mathcal{P}(D)$ be the set of lists $L^* \subseteq \sset{1,2,3,4}$ with $|L^*| \leq 3$ such that $D$ can be colored with list assignment $L'(x) = L(x) \cap L^*$ for $x \in V(D)$. Since $G$ is $P_6$-free, it follows from Theorem~\ref{3colP7} that $\mathcal{P}(D)$ can be computed in polynomial time. 
Since $C$ is connected, it follows from \eqref{it:mixedy0} that some vertex of 
$B$ is complete to $D$. Consequently, in
 any precoloring extension of $P$, at most three colors appear in $D$,
and at least one color of $L$ does not appear in $D$.
Therefore, if $\mathcal{P}(D) = \emptyset$, or if $L \subseteq L'$ for
every $L' \in \mathcal{P}(D)$, 
then $P$ has no precoloring extension we set $\mathcal{L}=\emptyset$ and stop.
Let $\mathcal{P}^*(D)$ be the set of $L' \subseteq \sset{1,2,3,4}$ such that $L'\not\in \mathcal{P}(D)$, but for every proper superset $L'' \subseteq \sset{1,2,3,4}$ of $L'$ with $|L''| \leq 3$, we have that $L'' \in \mathcal{P}(D)$. 
Let $d \in V(D)$. We now replace $D$ by a stable set 
$R(D)=\{d(L^*)\}_{L^*}$   of
copies of $d$, one for each $L^* \in \mathcal{P}^*(D)$ with $c_4 \in L^*$, and 
set $L'(d(L^*))= \sset{1,2,3,4} \setminus L^*$. Then $L'(d(L^*)) \subseteq L$.
Let $C'$ denote the graph obtained by this process (repeated for every component
of $C|Y_0$) from $C$. 
Let $L'(v)=L(v)$ for every $v \in V(C) \setminus Y_0$.  Since $C'$ is obtained
from an induced subgraph of $G$ by replacing vertices with stable sets,
it follows that $C'$ is $P_6$-free.

We claim that $C$ has a proper  $L$-coloring if and only if $C'$ has a proper 
$L'$-coloring. Suppose that $C$ has a proper $L$-coloring $c$.
We need to show that $c|_{V(C) \setminus Y_0}$ can be extended to each $R(D)$.
We can consider each $D$ separately. 

Let $D$ be a component of $C|Y_0$.
Let $L^* = c(D)$. 
Let  $L^{**} = c(N(D))$. We claim that for every $r \in R(D)$, 
$L'(r) \setminus L^{**}  \neq \emptyset$. Suppose $L'(r) \subseteq L^{**}$.
Then $\{1,2,3,4\} \setminus L'(r) \in \mathcal{P}^*(D)$, but 
$L^* \subseteq \{1,2,3,4\} \setminus L^{**} \subseteq \{1,2,3,4\} \setminus L'(r)$,
a contradiction. This proves that for every $r \in R(D)$, there exists
$d(r)  \in L'(r) \setminus L^{**}$, and setting
$c(r)=d(r)$ we obtain a coloring of $C'$.

Next suppose that $C'$ has a proper $L'$-coloring $c$. Let $L^* = \sset{1,2,3,4} \setminus c(N(D))$. If $L^* \in \mathcal{P}(D)$, then we color $D$ with an $L$-coloring using only those colors in $L^*$; this is possible by the definition of $\mathcal{P}(D)$. Thus we may assume  that $L^* \not\in \mathcal{P}(D)$. Since $L(x) \subseteq L$ for all $x \in N(D) \subseteq B$, it follows that $c_4 \in L^*$. From the definition of $\mathcal{P}^*(D)$, it follows that 
some superset $L^{**}$ of $L^*$ is in $\mathcal{P}^*(D)$. Then
$L'(d'(L^{**}) = \sset{1,2,3,4} \setminus L^{**} \subseteq \sset{1,2,3,4} \setminus L^* = c(N(D))$. However,  $c(d') \in L'(d')$, and thus 
$c(d) \in c(N(D))=c(N(d))$,
contrary to the fact that $c$ is a proper coloring.
This proves that $C$ has a proper $L$-coloring
if and only if $C'$ has a proper $L'$-coloring. 

We have so far proved the following:
\begin{itemize}
\item $C'$ has a proper $L'$-coloring if and only if $C$ has a proper $L$-coloring;
\item $C'$ is $P_6$-free; and
\item for every $x \in V(C')$, we have that $L'(x) \subseteq L$.
\end{itemize}

By Theorem~\ref{3colP7}, we can decide in polynomial time if $C'$ has a proper $L'$-coloring, and thus if $C$ has a proper $L$-coloring. If not, then $P$ has no precoloring extension; we set $\mathcal{L}=\emptyset$ and stop. 
If $C$ has a proper $L$-coloring, then 
$(G \setminus V(C), S, X_0, X, Y_0 \setminus V(C), Y \setminus V(C), f)$ has a 
precoloring extension if and only if $P$ does. 

By repeatedly applying this algorithm to every deficient component $C$
of $G|(Y_0 \cup Y_L^*)$, 
and setting $Z=\bigcup V(C)$ where the union is taken over all such components, 
we set
$P' = (G \setminus Z, S, X_0, X, Y_0 \setminus Z, Y \setminus Z, f)$ 
and output $\mathcal{L}=\{P'\}$. Then $P'$ satisfies 
\eqref{it:conn}-\eqref{it:complete}, and Lemma~\ref{lem:complete} follows. 
\end{proof}

We call a seeded precoloring \emph{good} if it satisfies  \eqref{it:conn}, \eqref{it:seed}, \eqref{it:y0}, \eqref{it:mixedy0}, \eqref{it:lists}, \eqref{it:123star}, \eqref{it:mixedyl},  and \eqref{it:complete}. 

By applying Lemmas~\ref{lem:conn}, \ref{lem:seedy0}, \ref{lem:it:y0},
\ref{lem:lists},\ref{lem:123star}, \ref{lem:mixedyl} and \ref{lem:complete}, 
each to every seeded precoloring in the output of the previous one,
we finally derive the main theorem of Section~\ref{sec:axioms}. 
\begin{theorem} \label{thm:y0main} There is a constant $C$ such that the following holds. Let $G$ be a $P_6$-free graph, and let $(G,A,f)$ be a 4-precoloring of $G$. Then there exists a polynomial-time algorithm that computes a collection $\mathcal{L}$ of seeded precolorings such that 
\begin{itemize}
\item $\mathcal{L}$ is equivalent for $P$.
\item for every $(G',S',X_0',X',Y_0', Y',f') \in \mathcal{L}$, $G'$ is an induced subgraph of $G$, $A \subseteq X_0' \cup S'$ and $f'|_A=f|_A$.
\item every $P \in \mathcal{L}$ is good
\item every seeded precoloring in $\mathcal{L}$ has a seed of size at most $C$; 
\item $|\mathcal{L}| \leq |V(G)|^C$.
\end{itemize}
\end{theorem}
By Theorem~\ref{thm:y0main}, to solve the 4-precoloring extension problem in polynomial time, it is sufficient to solve the precoloring extension problem for good seeded precolorings of $P_6$-free graphs (with seed size bounded by a constant) in polynomial time.

\section{Establishing the Axioms on $Y$} \label{sec:axiomsy}

In the previous section, we arranged that components of $G|(Y_0 \cup Y)$ containing a vertex of $Y_0$ are well-behaved. In this section, we deal with components of $G|(Y_0 \cup Y)$ that do not contain a vertex of $Y_0$. 

Let $P$ be a starred precoloring.
We say that a collection $\mathcal{L}$ of starred precolorings is an \emph{equivalent collection} for $P$ if $P$ has a precoloring extension if and only if at least one of the starred precolorings in $\mathcal{L}$ does. 

The following are the axioms we want to establish for starred precolorings. 
\begin{enumerate}[(I)]
\item Every vertex $y$ in $Y$ satisfies $|L_P(y)| =3$. \label{it:size3}
\item Let $L_1, L_2 \subseteq \sset{1,2,3,4}$ with $|L_1| = |L_2| = 3$ and $L_1 \neq L_2$. Then there is no path $a-b-c$ with $L_P(a) = L_1$, $L_P(b) = L_P(c) = L_2$ with $a, b, c \in Y$. \label{it:122}
\item Let $L_1, L_2, L_3 \subseteq \sset{1,2,3,4}$ with $|L_1| = |L_2| = |L_3| = 3$ and $L_1 \neq L_2 \neq L_3 \neq L_1$. Then there is no path $a-b-c$ with $L_P(a) = L_1$, $L_P(b) = L_2$, $L_P(c) = L_3$ with $a, b, c\in Y$. \label{it:123}
\item Let $L_1 \subseteq \sset{1,2,3,4}$ with $|L_1| = 3$. Then there is no path $a-b-c$ with $L_P(b) = L_P(c) = L_1$ and $a \in X$, $b,c \in Y$. \label{it:x11}
\item Let $L_1, L_2 \subseteq \sset{1,2,3,4}$ with $|L_1| = |L_2| = 3$. Then there is no path $a-b-c$ with $L_P(b) = L_1, L_P(c) = L_2$ and $a \in X$ with $L_P(a) \neq L_1 \cap L_2$. \label{it:x12}
\item For every component $C$ of $G|Y$, for which there is a  vertex of $X$
is mixed on $C$, 
there exist $L_1, L_2 \subseteq \sset{1,2,3,4}$ with $|L_1| = |L_2| = 3$ such that $C$ contains a vertex $x_i$ with $L_P(x_i) = L_i$ for $i=1,2$, every vertex $x$ in $C$ satisfies $L_P(x) \in \sset{L_1, L_2}$, and every $x \in X$ mixed
on $C$ satisfies $L_P(x)=L_1 \cap L_2$. \label{it:l1l2}
\item For every component $C$ of $G|Y$ such that some vertex of $X$ is
mixed on $C$, and for $L_1, L_2$ as in \eqref{it:l1l2}, $L_P(v) = L_1 \cap L_2$ for every vertex $v \in X$ with a neighbor in $C$. \label{it:orthogonal}
\item $Y=\emptyset$. \label{it:yisempty}
\end{enumerate}

We begin by showing that starred precolorings exist, and we establish axiom \eqref{it:size3}. 
\begin{lemma} \label{lem:starred}
 Let $P$ be a good seeded precoloring of a $P_6$-free graph $G$. Then $$P' =(G, S, X_0, X, Y \setminus Y_L^*, Y_L^* \cup Y_0, f)$$ (with $Y_L^*$ as in \eqref{it:mixedyl}) is a starred precoloring satisfying \eqref{it:size3} and $P'$ has a precoloring extension if and only if $P$ does, and every precoloring extension of $P'$ is a precoloring extension of $P$.  
\end{lemma}
\begin{proof}
This is easily verified by checking the definition of a starred precoloring. 
\end{proof}

Our next goal is to establish axiom \eqref{it:122}, which we restate. 
\begin{enumerate}
\item[\eqref{it:122}] Let $L_1, L_2 \subseteq \sset{1,2,3,4}$ with $|L_1| = |L_2| = 3$ and $L_1 \neq L_2$. Then there is no path $a-b-c$ with $L_P(a) = L_1$, $L_P(b) = L_P(c) = L_2$ with $a, b, c \in Y$.
\end{enumerate}
This lemma will also be useful for proving \eqref{it:x11}. 
\begin{lemma} \label{lem:pairoflists}
There is a function $q: \mathbb{N} \rightarrow \mathbb{N}$ such that the following holds. Let $L_1 \subseteq \sset{1,2,3,4}$ with $|L_1| = 3$, and let $P = \sspc$ be a starred precoloring of a $P_6$-free graph $G$ with $P$ satisfying \eqref{it:size3}. Then there is an algorithm with running time $O(|V(G)|^{q(|S|)})$ that outputs an equivalent collection $\mathcal{L}$ for $P$ such that
\begin{itemize}
\item $|\mathcal{L}| \leq |V(G)|^{q(|S|)}$;
\item every $P' \in \mathcal{L}$ is a starred precoloring of $G$;
\item every $P' \in \mathcal{L}$ with seed $S'$ satisfies $|S'| \leq q(|S|)$; 
\item every $P' = (G, S', X_0', X', Y', Y^*, f') \in \mathcal{L}$ satisfies \eqref{it:size3} and $Y' \subseteq Y$; 
\item if there is no path $a-b-c$ with $L_{P}(a) \neq L_1'$, $L_{P}(b) = L_{P}(c) = L_1'$ with $a, b, c \in Y$  for some $L_1'$ with $|L_1'| = 3$,
   then there is no path $a-b-c$ with $L_{P'}(a) \neq L_1'$, $L_{P'}(b) = L_{P'}(c) = L_1'$ with $a, b, c \in Y'$; 
 \item if $P$ satisfies \eqref{it:122}, and 
   if there is no path $a-b-c$ with $L_{P}(a) \neq L_1'$, $L_{P}(b) = L_{P}(c) = L_1'$ with $a, b, c \in X \cup Y$  for some $L_1'$ with $|L_1'| = 3$,
   then there is no  path $a-b-c$ with $L_{P'}(a) \neq L_1'$, $L_{P'}(b) = L_{P'}(c) = L_1'$ with $a, b, c \in X' \cup Y'$; and 
\item there is no path $a-b-c$ with $L_{P'}(a) \neq L_1$, $L_{P'}(b) = L_{P'}(c) = L_1$ with $a, b, c \in X' \cup Y'$.
\end{itemize}
Moreover, for every $P' \in \mathcal{L}$, given a precoloring extension of $P'$, we can compute a precoloring extension for $P$ in polynomial time, if one exists.
\end{lemma}
\begin{proof}
Let $L_1 \subseteq \sset{1,2,3,4}$ with $|L_1| = 3$, and let $P = \sspc$ be a starred precoloring of a $P_6$-free graph $G$ with $P$ satisfying \eqref{it:size3}. We check in polynomial time if $G$ contains a $K_5$. If so, then $P$ does not have a precoloring extension and we output $\mathcal{L} = \emptyset$ as an equivalent collection. Therefore, for the remainder of the proof we may assume that $G$ contains no $K_5$. 

Let $\mathcal{L} = \emptyset$. Let $Y_1 = \sset{y \in Y: L_P(y) = L_1}$. Let $\mathcal{T} = \sset{T_1, \dots, T_r}$ be the set of types $T \subseteq S$ with $f(T) \neq \sset{1,2,3,4} \setminus L_1$ and $|f(T)| \leq 2$, and if $P$ satisfies \eqref{it:122}, $|f(T)| = 2$. 
Let $\mathcal{Q}$ be the set of all $r$-tuples of quadruples $((Q_1, R_1 c_1, d_1), \dots, (Q_r, R_r, c_r, d_r))$ such that for every $i \in \sset{1, \dots, r}$,
\begin{itemize}
\item $c_i, d_i \in L_1$; 
\item $1 \geq |Q_i| \geq |R_i|$ and $Q_i \cup R_i$ is a clique; and
\item $Q_i \cup R_i \subseteq (X \cup Y)(T_i)$. 
\end{itemize}

For every $Q = ((Q_1, R_1, c_1, d_1), \dots, (Q_r, R_r, c_r, d_r)) \in \mathcal{Q}$, we proceed as follows. Let $S'^Q = Q_1 \cup R_1 \cup \dots \cup Q_r \cup R_r$, and let $f': S' \rightarrow L_1$ be such that $f'(q_i) = c_i$ for all $i$ for which $Q_i = \sset{q_i}$, and $f'(r_i) = d_i$ for all $i$ for which $R_i = \sset{r_i}$. Let $$\tilde{Y}^Q = \bigcup_{i: Q_i = \emptyset} (X \cup Y)(T_i),$$
and let $g^Q: \tilde{Y}^Q \rightarrow \sset{1,2,3,4} \setminus L_1$ be the constant function. Let $$\tilde{Z}^Q = \bigcup_{i: R_i = \emptyset, Q_i \neq \emptyset} ((X \cup Y)(T_i) \cap N(Q_i)),$$ and let $g'^Q: \tilde{Z}^Q \rightarrow \sset{1,2,3,4} \setminus L_1$ be the constant function.

For $i \in \sset{1, \dots, r}$, let $\tilde{X}_i$ and $g''^Q_i$ be defined as follows. If $|f(T_i)| = 1$, we let $\tilde{X}_i = X(T_i) \cap N(Q_i) \cap N(R_i)$. If $|f(T_i)| = 2$, we let $\tilde{X}_i = X(T_i) \cap N(Q_i)$. We let $g''^Q(\tilde{X}_i) = \sset{1,2,3,4} \setminus (f'(T_i) \cup f'(Q_i) \cup f'(R_i))$.  Let $\tilde{X}^Q = \tilde{X}_1 \cup \dots \cup \tilde{X}_r$. 

Then, if $f \cup f' \cup g^Q \cup g'^Q \cup g''^Q$ is a proper coloring of $G|(S \cup S'^Q \cup X_0 \cup \tilde{Y}^Q \cup \tilde{Z}^Q \cup \tilde{X}^Q)$, we add the starred precoloring 
\begin{align*}
  P'^Q =& (G, S \cup S'^Q, \\
        &X_0 \cup \tilde{X}^Q \cup \tilde{Y}^Q \cup \tilde{Z}^Q, \\
        &X \setminus (\tilde{X}^Q \cup \tilde{Y}^Q \cup \tilde{Z}^Q \cup S'^Q),\\
        &Y \setminus (\tilde{X}^Q \cup \tilde{Y}^Q \cup \tilde{Z}^Q \cup S'^Q),\\
  &Y^*, f \cup f' \cup g^Q \cup g'^Q \cup g''^Q)
\end{align*}
to $\mathcal{L}$.

This starred precoloring satisfies \eqref{it:size3}. Every precoloring extension of $P'^Q$ is a precoloring extension of $P$. Moreover, suppose that $c$ is a precoloring extension of $P$. Let $Q = ((Q_1, R_1 c_1, d_1), \dots, (Q_r, R_r, c_r, d_r))$ be defined as follows:
\begin{itemize}
\item For every type $T_i \in \mathcal{T}$ such that $c((X \cup Y)(T_i)) \subseteq \sset{1,2,3,4} \setminus L_1$, we let $Q_i = R_i = \emptyset$ and $c_i, d_i \in L_1$ arbitrary. 
\item For every type $T_i \in \mathcal{T}$ such that there exist $x,y \in (X \cup Y)(T_i)$ with  $c(x), c(y) \in L_1$ and $xy \in E(G)$, we let $Q_i = \sset{x}$, $R_i = \sset{y}$ and $c_i = c(x), d_i = c(y)$.
\item For every type $T_i \in \mathcal{T}$ such that do not there exist $x,y$ as above, but there is a vertex $v \in (X \cup Y)(T_i)$ with $c(v) \in L_i$, we let $Q_i = \sset{v}, R_i = \emptyset, c_i = c(v), d_i = d(v)$. 
\end{itemize}
Note that if $|Q_i \cup R_i| < 2$, then every vertex $v$ in $(X \cup Y)(T_i)$ complete to $Q_i \cup R_i$ satisfies $c(v) \not\in L_1$, and so $g$ and $g'$ agree with $c$ on  $\tilde{Y}$ and $\tilde{Z}$, respectively. It follows that $P'^Q \in \mathcal{L}$ and $c$ is a precoloring extension of $P'^Q$. Consequently, that $\mathcal{L}$ is an equivalent collection for $P$. 

We now prove that every $P'^Q \in \mathcal{L}$ satisfies the claims of the lemma. Let $Q = ((Q_1, R_1 c_1, d_1), \dots, (Q_r, R_r, c_r, d_r))$ with $P'^Q \in \mathcal{L}$, and write $P' = P'^Q \in \mathcal{L}$ with $P' = (G, S', X_0', X', Y', Y^*, f')$. Let $Y_1' = \sset{y \in Y': L_{P'}(y) = L_1}$.

\vspace*{-0.4cm}\begin{equation}\vspace*{-0.4cm} \label{eq:long}
  \longbox{\emph{If there is no  path $a-b-c$ with $L_{P}(a) \neq L_1'$, $L_{P}(b) = L_{P}(c) = L_1'$ with $a, b, c \in Y$  for some $L_1'$ with $|L_1'| = 3$,
   then there is no path $a-b-c$ with $L_{P'}(a) \neq L_1'$, $L_{P'}(b) = L_{P'}(c) = L_1'$ with $a, b, c \in Y'$; and if $P$ satisfies \eqref{it:122}, and 
   if there is no  path $a-b-c$ with $L_{P}(a) \neq L_1'$, $L_{P}(b) = L_{P}(c) = L_1'$ with $a, b, c \in X \cup Y$  for some $L_1'$ with $|L_1'| = 3$,
   then there is no path $a-b-c$ with $L_{P'}(a) \neq L_1'$, $L_{P'}(b) = L_{P'}(c) = L_1'$ with $a, b, c \in X' \cup Y'$.}}
\end{equation}

Suppose not; and let $a-b-c$ be such a path. Since $b, c \in Y' \subseteq Y$, it follows that $L_{P}(b) = L_{P}(c) = L_1'$. By the assumption of \eqref{eq:long}, it follows that $L_{P}(a) \neq L_{P'}(a)$, and so $a \in Y \cap X'$. This implies that $|L_{P'}(a)| = 2$. Since $a \not\in Y'$, it follows that the first statement of \eqref{eq:long} is proved. 

Therefore, we may assume that \eqref{it:122} holds for $P$. Since $P$ satisfies \eqref{it:122}, it follows that $L_{P}(a) = L_1'$. Moreover, there is a vertex $s \in S' \setminus S$ with $f(s) \in L_1'$ and $as \in E(G)$. Since $b \in Y'$, it follows that $s-a-b$ is a path. But since $P$ satisfies \eqref{it:122}, it follows that $S' \setminus S \subseteq X$ by construction, and so $s \in X$. But then the path $s-a-b$ contradicts the assumption of \eqref{eq:long2}. This implies \eqref{eq:long2}.

\vspace*{-0.4cm}\begin{equation}\vspace*{-0.4cm} \label{eq:no3pathy1}
  \longbox{\emph{There is no path $z-a-b-c$ with $z \in (X' \cup Y') \setminus Y_1'$ and $a, b, c \in Y_1'$.}}
\end{equation}

Suppose not; and let $z-a-b-c$ as in \eqref{eq:no3pathy1}. It follows that $z \in X \cup Y$ and $a, b, c \in Y_1$. Let $T_i = N(z) \cap S \in \mathcal{T}$. Since $z \not\in X_0'$, it follows that $z\not\in \tilde{X}^Q \cup \tilde{Y}^Q \cup \tilde{Z}^Q$. Therefore, $Q_i\cup R_i$ contains a vertex $y$ non-adjacent to $z$.  Since $c_i, d_i \in L_1$, it follows that $y$ is anticomplete to $\{z, a, b, c\}$. Let $s \in T_i$ with $f(s) \in L_1$; then $s$ is a common neighbor of $y$ and $z$. It follows that $s$ is not adjacent to $a, b, c$. But then $y-s-z-a-b-c$ is a $P_6$ in $G$, a contradiction. This proves \eqref{eq:no3pathy1}. 

\bigskip

Let $\mathcal{L}_5 = \mathcal{L}$. We repeat the following procedure for $j=4,3,2$. For every $P' = (G, S', X_0', X', Y', Y^*, f') \in \mathcal{L}_{j+1}$, we proceed as follows. We let $\mathcal{L}_j(P') = \emptyset$. Let $Y_1' = \sset{y \in Y': L_{P'}(y) = L_1}$. Let $Y_1^*$ be the set of vertices $y$ in $(X' \cup Y') \setminus Y_1'$ such that there is a clique $\sset{a_1, \dots, a_j} \subseteq Y_1'$ and $N(y) \cap \sset{a_1, \dots, a_j} = \sset{a_1}$. Let $\mathcal{T}^j = \sset{T_1^j, \dots, T_{r_j}^j}$ be the set of all types $T \subseteq S'$ such that $f(T) \neq \sset{1,2,3,4} \setminus L_1$ and $|f(T)| \leq 2$, and if $P'$ satisfies \eqref{it:122}, $|f(T)| = 2$. Let $\mathcal{Q}(P')$ be the set of all $r_j$-tuples of quadruples $((Q_1, R_1 c_1, d_1), \dots, (Q_r, R_r, c_r, d_r))$ such that  for every $i \in \sset{1, \dots, r_j}$,
\begin{itemize}
\item $c_i, d_i \in L_1$; 
\item $1 \geq |Q_i| \geq |R_i|$ and $Q_i \cup R_i$ is a clique; and
\item $Q_i \cup R_i \subseteq (X \cup Y)(T_i)$. 
\end{itemize}

For every $Q = ((Q_1, R_1 c_1, d_1), \dots, (Q_r, R_r, c_r, d_r)) \in \mathcal{Q}$, we proceed as follows. Let $S'^Q = Q_1 \cup R_1 \cup \dots \cup Q_r \cup R_r$, and let $g^Q: S' \rightarrow L_1$ such that $g^Q(q_i) = c_i$ for all $i$ such that $Q_i = \sset{q_i}$, and $g^Q(r_i) = d_i$ for all $i$ such that $R_i = \sset{r_i}$.

For $i \in \sset{1, \dots, r_j}$, we let $Z_i$ be the set of vertices $z \in (X \cup Y)(T_i)$ such that one of the following holds:
\begin{itemize}
\item $Q_i = \emptyset$; 
\item $Q_i = \sset{q_i}$, and $N(q_i) \cap Y_1' \subsetneq N(z) \cap Y_2'$;
\item $Q_i = \sset{q_i}$, $R_i = \sset{r_i}$, $z$ is adjacent to $q_i$ and $N(r_i) \cap Y_2' \subsetneq N(z) \cap Y_1'$;
\end{itemize}
We let $\tilde{Z}^Q = Z_1 \cup \dots \cup Z_{r_j}$ and $g'^Q : \tilde{Z}^Q \rightarrow \sset{1,2,3,4} \setminus L_1$.
Let $$\tilde{X}^Q = \bigcup_{i: R_i = \emptyset, Q_i \neq \emptyset} ((X \cup Y)(T_i) \cap N(S_i)),$$ and let $g''^Q: \tilde{X}^Q \rightarrow \sset{1,2,3,4} \setminus L_1$ be the constant function.
Let
\begin{align*}
  P'^Q = &( G, S' \cup S^Q, X_0' \cup \tilde{Z}^Q \cup \tilde{X}^Q, \\
         &X' \setminus (S^Q \cup \tilde{Z}^Q \cup \tilde{X}^Q),\\
         &Y' \setminus (S^Q \cup \tilde{Z}^Q \cup \tilde{X}^Q), Y^*,\\
         &f' \cup g^Q \cup g'^Q \cup g''^Q).
\end{align*}
If $f' \cup g^Q \cup g'^Q \cup g''^Q$ is proper coloring of $G|(S' \cup S^Q \cup \tilde{Z}^Q \cup \tilde{X}^Q)$, then we add $P'^Q$ to $\mathcal{L}_j(P')$.

It follows that for every $Q \in \mathcal{Q}(P')$, every precoloring extension of $P'^Q$ is a precoloring extension of $P'$. Moreover, suppose that $c$ is a precoloring extension of $P'$. We define $Q =((Q_1, R_1 c_1, d_1), \dots, (Q_r, R_r, c_r, d_r))$ as follows:
\begin{itemize}
\item For every type $T_i \in \mathcal{T}$ such that $c((X \cup Y)(T_i)) \cap L_1 = \emptyset$, we let $Q_i = R_i = \emptyset$ and $c_i, d_i \in L_1$ arbitrary. 
\item For every type $T_i \in \mathcal{T}$ such that  $c((X \cup Y)(T_i)) \cap L_1 \neq \emptyset$, we let $v$ a vertex $v \in (X \cup Y)(T_i)$ with $c(v) \in L_1$ with $N(v) \cap Y_1$ maximal. We let $Q_i = \sset{v}, c_i = c(v)$. If there is a vertex $w$ in $N(v) \cap (X \cup Y) (T_i)$ with $c(w) \in L_1$, then we choose such a vertex with $N(w) \cap Y_1$ maximal and let $R_i = \sset{w}, d_i = c(w)$; otherwise we let $R_i = \emptyset$ and $d_i \in L_1$ arbitrary. 
\end{itemize}
The second bullet implies that $c(x) \not\in L_1$ for every $x \in (X \cup Y)(T_i)$ such that $Q_i = \sset{q_i}$ and $N(q_i) \cap Y_1' \subsetneq N(v) \cap Y_1'$. Similarly, $c(x) \not\in L_1$ for every $x \in (X \cup Y)(T_i) \cap N(Q_i)$ such that $R_i = \sset{r_i}$ and $N(r_i) \cap Y_1' \subsetneq N(v) \cap Y_1'$. It follows that $Q \in \mathcal{Q}(P')$, $P'^Q \in \mathcal{L}_j(P')$, and $c$ is a precoloring extension of $P'^Q$. Thus $\mathcal{L}_j(P')$ is an equivalent collection for $P'$. By construction, $P'^Q$ satisfies \eqref{it:size3} for every $Q \in \mathcal{Q}(P')$. 

Now let $$\mathcal{L}_j = \bigcup_{P' \in \mathcal{L}_{j+1}} \mathcal{L}_j(P').$$
Since $\mathcal{L}_{j+1}$ is an equivalent collection for $P$ and since $\mathcal{L}_j$ is the union of equivalent collections for every $P' \in \mathcal{L}_{j+1}$, it follows that $\mathcal{L}_j$ is an equivalent collection for $P$.

Let $P' \in \mathcal{L}_{j+1}$. Let $Q =((Q_1, R_1 c_1, d_1), \dots, (Q_r, R_r, c_r, d_r)) \in \mathcal{Q}(P')$, and let $P'^Q = (G, S'', X_0'', X'', Y'', Y^*, f'') \in \mathcal{L}_j(P')$. Let $Y_1'' = \sset{y \in Y'': L_{P'^Q}(y) = L_1}$. From the previous step ($j+1$) of our argument, we may assume that~\eqref{eq:mixedk4y} and~\eqref{eq:no3pathy1} hold for $j+1$ for $P'$ and $Y_1'$. This is true when $j=4$ as well, since $G$ contains no $K_5$. 

\vspace*{-0.4cm}\begin{equation}\vspace*{-0.4cm} \label{eq:mixedk4y}
  \longbox{\emph{There is no vertex $z \in (X'' \cup Y'') \setminus Y_1''$ with $N(z) \cap \sset{a_1, \dots, a_j} = \sset{a_1}$ for a clique $\sset{a_1, \dots, a_j} \subseteq Y_1''$.}}
\end{equation}

Suppose for a contradiction that $z$ is such a vertex. Write $P' = (G, S', X_0', X', Y', Y^*, f')$. Let $Y_1' = \sset{y \in Y': L_{P'}(y) = L_1}$ for $i=1,2$. Suppose first that $z \in Y_1'$. Then $z$ has a neighbor $s \in S'' \setminus S'$. It follows that $f''(s) \in L_1$ and $s \not\in Y_1'$. Consequently, $s$ is anticomplete to $\{a_1, \dots, a_j\}$. But then the path $s-z-a_1-a_j$ contradicts the fact that \eqref{eq:no3pathy1} holds for $P'$. 

It follows that $z \in (X' \cup Y') \setminus Y_1'$ and $\sset{a_1, \dots, a_j} \subseteq Y_1'$. Let $i$ such that $S' \cap N(z) = T_i$. Since $z \not\in X_0''$, it follows $Q_i \neq \emptyset$; say $Q_i = \sset{q_i}$. If $z$ is non-adjacent to $q_i$, let $s = q_i$. Otherwise, it follows that $R_i = \sset{r_i}$, say; let $s = r_i$. In both cases, it follows that $s$ is non-adjacent to $z$. 

Since $a_1, \dots, a_j \not\in X''$, it follows that $s$ is non-adjacent to $a_1, \dots, a_j$. The definition of $Z_i$ implies  that $N(s) \cap Y_2' \not\subset N(z) \cap Y_2'$. Since $a_1 \in (N(z) \setminus N(s)) \cap Y_1'$, we deduce that there exists a vertex $y \in (N(z) \setminus N(s)) \cap Y_1'$.

Let $s' \in T_i$ with $f(s') \in L_1$. Then, $s'$ is non-adjacent to $a_1, \dots, a_j$.  But $y-s-s'-z-a_1-a_j$ is not a $P_6$ in $G$, and thus $y$ has a neighbor in $\sset{a_1, \dots, a_j}$. But $y$ is not complete to $\sset{a_1, \dots, a_j}$, since $P'$ satisfies \eqref{eq:mixedk4y} for $j+1$. It follows that $y$ is mixed on $\sset{a_1, \dots, a_j}$, and thus by Lemma~\ref{lem:mixed} there is a path $y-a-b$ with $a, b \in \sset{a_1, \dots, a_j}$. But then $s-y-a-b$ is a path, contrary to the fact that $P'$ satisfies \eqref{eq:no3pathy1}. This concludes the proof of \eqref{eq:mixedk4y}. 

\vspace*{-0.4cm}\begin{equation}\vspace*{-0.4cm} \label{eq:no3pathy1reprise}
  \longbox{\emph{There is no path $z-a-b-c$ with $z \in (X'' \cup Y'') \setminus Y_1'$ and $a, b, c \in Y_1''$.}}
\end{equation}

Suppose not; and let $z-a-b-c$ be such a path. Since $Y_1'' \subseteq Y_1'$, 
the fact that $P'$ satisfies \eqref{eq:no3pathy1} implies
that $z \not\in X' \cup Y'$, and thus $z \in Y_1'$. Thus $z$ has a neighbor $s \in S'' \setminus S'$ with $f(s) \in L_1$. It follows that $s \in X' \cup Y'$, and thus $s-z-a-b$ is a path, contrary to the fact that \eqref{eq:no3pathy1} holds for $P'$. This proves \eqref{eq:no3pathy1reprise}.

\vspace*{-0.4cm}\begin{equation}\vspace*{-0.4cm} \label{eq:long2}
  \longbox{\emph{If there is no  path $a-b-c$ with $L_{P'}(a) \neq L_1'$, $L_{P'}(b) = L_{P'}(c) = L_1'$ with $a, b, c \in Y'$  for some $L_1'$ with $|L_1'| = 3$,
   then there is no path $a-b-c$ with $L_{P''}(a) \neq L_1'$, $L_{P''}(b) = L_{P''}(c) = L_1'$ with $a, b, c \in Y''$; and if $P'$ satisfies \eqref{it:122}, and 
   if there is no path $a-b-c$ with $L_{P'}(a) \neq L_1'$, $L_{P'}(b) = L_{P'}(c) = L_1'$ with $a, b, c \in X \cup Y$  for some $L_1'$ with $|L_1'| = 3$,
   then there is no path $a-b-c$ with $L_{P'}(a) \neq L_1'$, $L_{P''}(b) = L_{P''}(c) = L_1'$ with $a, b, c \in X'' \cup Y''$.}}
\end{equation}

Suppose not; and let $a-b-c$ be such a path. Since $b, c \in Y'' \subseteq Y'$, it follows that $L_{P'}(b) = L_{P'}(c) = L_1'$. By the assumption of \eqref{eq:long2}, it follows that $L_{P'}(a) \neq L_{P''}(a)$, and so $a \in Y' \cap X''$. This implies that $|L_{P''}(a)| = 2$. Since $a \not\in Y''$, it follows that the first statement of \eqref{eq:long2} is proved. 

Therefore, we may assume that \eqref{it:122} holds for $P'$. Since $P'$ satisfies \eqref{it:122}, it follows that $L_{P'}(a) = L_1'$. Moreover, there is a vertex $s \in S'' \setminus S'$ with $f'(s) \in L_1'$ and $as \in E(G)$. Since $b \in Y''$, it follows that $s-a-b$ is a path. But since $P'$ satisfies \eqref{it:122}, it follows that $S'' \setminus S' \subseteq X'$ by construction, and so $s \in X'$. But then the path $s-a-b$ contradicts the assumption of \eqref{eq:long2}. This implies \eqref{eq:long2}.

\bigskip

It follows that \eqref{eq:no3pathy1} and \eqref{eq:long2}  holds for $P'^Q$ for every $P' \in \mathcal{L}_{j+1}$ and $Q \in \mathcal{Q}(P')$. Moreover, by construction, $\mathcal{L}_j$ is an equivalent collection for $P$. If $j > 2$, we repeat the procedure for $j-1$; otherwise, we stop. 

At termination, we have constructed an equivalent collection $\mathcal{L}_2$ for $P$ and every $P' = (G,S',X_0',X',Y',Y_0',f') \in \mathcal{L}_2$ satisfies \eqref{it:size3} and \eqref{eq:mixedk4y} for $j=2$, and thus the last bullet of the lemma. The third-to-last and second-to-last bullets of the lemma follow from \eqref{eq:long} and \eqref{eq:long2}. Thus, $\mathcal{L}_2$ satisfies the properties of the lemma, and hence, the lemma is proved. 
\end{proof}

\begin{lemma}  \label{lem:122}
There is a function $q: \mathbb{N} \rightarrow \mathbb{N}$ such that the following holds. Let $P = \sspc$ be a starred precoloring of a $P_6$-free graph $G$ with $P$ satisfying \eqref{it:size3}. Then there is an algorithm with running time
$O(|V(G)|^{q(|S|)})$ that outputs an equivalent collection $\mathcal{L}$ for 
$P$ such that
\begin{itemize}
\item $|\mathcal{L}| \leq |V(G)|^{q(|S|)}$;
\item every $P' \in \mathcal{L}$ is a starred precoloring of $G$;
\item every $P' \in \mathcal{L}$ with seed $S'$ satisfies $|S'| \leq q(|S|)$;  and
\item every $P' \in \mathcal{L}$ satisfies \eqref{it:size3} and \eqref{it:122}. 
\end{itemize}
Moreover, for every $P' \in \mathcal{L}$, given a precoloring extension of $P'$, we can compute a precoloring extension for $P$ in polynomial time, if one exists. 
\end{lemma}
\begin{proof}
Let $\mathcal{L} = \sset{P}$. We repeat the following for every pair $L_1, L_2$ of distinct lists of size three contained in $\sset{1,2,3,4}$.
We apply Lemma~\ref{lem:pairoflists} to every starred precoloring $P' \in \mathcal{L}$, and replace $\mathcal{L}$ by the union of the equivalent collections 
produced by Lemma~\ref{lem:pairoflists}. Then
we move to the next pair of lists. 
\end{proof}

The next lemma is a simple tool that we will use to establish further axioms. 
\begin{lemma}\label{lem:nunv} Let $G$ be a $P_6$-free graph with $u, v \in V(G)$ such that $V(G) = \sset{u,v} \cup N(u) \cup N(v)$, $uv \not\in E(G)$, $N(u) \cap N(v) = \emptyset$, and $N(u), N(v)$ stable. Then there is a partition $A_0, A_1, \dots, A_k$ of $N(u)$ and a partition $B_0, B_1, \dots, B_k$ of $N(v)$ with $k \geq 0$ such that 
\begin{itemize}
\item $A_0$ is complete to $N(v)$;
\item $B_0$ is complete to $N(u)$; and 
\item for $i = 1, \dots, k$, $A_i, B_i \neq \emptyset$ and $A_i$ is complete to $N(v) \setminus B_i$ and $B_i$ is complete to $N(u) \setminus A_i$, and $A_i$ is anticomplete to $B_i$. 
\end{itemize}
\end{lemma}
\begin{proof}
Let $G, u, v$ as in the lemma. The result holds if $N(u) = \emptyset$ or $N(v) = \emptyset$; thus we may assume that both sets are non-empty. Let $a \in N(u), b \in N(v)$. If $ab \in E(G)$, we let $A_0 = \sset{a}, B_0 = \sset{b}$; otherwise, we let $A_1 = \sset{a}, B_1 = \sset{b}$. Now let $A_0, A_1, \dots, A_k, B_0, B_1, \dots, B_k$ be chosen such that their union is maximal subject to satisfying the conditions of the lemma. If their union is $V(G) \setminus \sset{u,v}$, then there is nothing to show; thus we may assume that there is a vertex $x \not\in \sset{u,v}$ not contained in their union. Without loss of generality, we may assume that $x \in N(v)$. 

If $x$ is complete to $A = A_0 \cup A_1 \cup \dots \cup A_k$, we can add $x$ to $B_0$, contrary to the maximality of our choice of sets. Suppose first that $x$ is complete to $A_1 \cup \dots \cup A_k$. Let $A_{k+1} = A_0 \setminus N(x)$. Then $A_{k+1}$ is non-empty, since $x$ has a non-neighbor in $A$. But then $A_0 \setminus A_{k+1}, A_1, \dots, A_k, A_{k+1}, B_0, B_1, \dots, B_k, \sset{x}$ satisfies the conditions of the lemma and has strictly larger union; a contradiction. 

It follows that $x$ has a non-neighbor in $A \setminus A_0$; without loss of generality we may assume that there is $y \in A_1$ non-adjacent to $x$. 
Let $w \in B_1$.
Suppose that $x$ has a neighbor $z \in A_1$. Then $w-v-x-z-u-y$ is a $P_6$ in $G$, a contradiction. It follows that $x$ has no neighbor in $A_1$. If $x$ is complete to $A \setminus A_1$, we can add $x$ to $B_1$ and enlarge the structure, a contradiction; hence $x$ has a non-neighbor $z$ in $A \setminus A_1$. It follows that $z$ is adjacent to $w$. But then $x-v-w-z-u-y$ is a $P_6$ in $G$, a contradiction. This concludes the proof of the lemma. 
\end{proof}

The purpose of the following lemmas is to establish the following axiom, which we restate: 
\begin{enumerate}
\item[\eqref{it:123}] Let $L_1, L_2, L_3 \subseteq \sset{1,2,3,4}$ with $|L_1| = |L_2| = |L_3| = 3$ and $L_1 \neq L_2 \neq L_3 \neq L_1$. Then there is no  path $a-b-c$ with $L_P(a) = L_1$, $L_P(b) = L_2$, $L_P(c) = L_3$ with $a, b, c\in Y$.
\end{enumerate}
\begin{lemma}\label{lem:fixedl1l2l3} There is a function $q: \mathbb{N} \rightarrow \mathbb{N}$ such that the following holds. Let $L_1, L_2, L_3 \subseteq \sset{1,2,3,4}$ with $|L_1| = |L_2| = |L_3| = 3$  and $L_1 \neq L_2 \neq L_3 \neq L_1$. Let $P = \sspc$ be a starred precoloring of a $P_6$-free graph $G$ with $P$ satisfying \eqref{it:size3} and \eqref{it:122}. Then there is an
algorithm with running time $O(|V(G)|^{q(|S|)})$ 
that outputs an equivalent collection $\mathcal{L}$ for $P$ such that
\begin{itemize}
\item $|\mathcal{L}| \leq |V(G)|^{q(|S|)}$;
\item every $P' \in \mathcal{L}$ is a starred precoloring of $G$;
\item every $P' \in \mathcal{L}$ with seed $S'$ satisfies $|S'| \leq q(|S|)$;  
\item every $P' \in \mathcal{L}$ satisfies \eqref{it:size3} and \eqref{it:122};  
\item every $P' = (G, S', X_0', X', Y', Y^*, f') \in \mathcal{L}$ satisfies that there is no  path $a-b-c-d$ with $L_{P'}(a) = L_1$, $L_{P'}(b) = L_{P'}(d) = L_2$, $L_{P'}(c) = L_3$ with $a, b, c, d\in Y'$; and 
\item if $P$ satisfies the previous bullet for $L_1, L_2, L_3$ and for $L_3, L_2, L_1$, then every $P' = (G, S', X_0', X', Y', Y^*, f') \in \mathcal{L}$ satisfies that there is no  path $a-b-c$ with $L_{P'}(a) = L_1$, $L_{P'}(b) = L_2$, $L_{P'}(c) = L_3$ with $a, b, c \in Y'$.
\end{itemize}
Moreover, for every $P' \in \mathcal{L}$, given a precoloring extension of $P'$, we can compute a precoloring extension for $P$ in polynomial time. 
\end{lemma}

\begin{proof}
We say that \emph{the conditions of the last bullet hold for $P$} if $P$ satisfies the second-to-last bullet for $L_1, L_2, L_3$ and $L_3, L_2, L_1$. 
  
Let $Y_i = \sset{y \in Y: L_P(y) = L_i}$ for $i = 1,2,3$. Let $\mathcal{T} = \sset{T_1, \dots, T_r}$ be the set of types $T \subseteq S$ with $f(T) = \sset{1,2,3,4} \setminus L_1$. We let $\mathcal{Q}$ be the set of all $r$-tuples $(Q_1, \dots, Q_r)$, where for each $i$, $Q_i = (S_i^1, S_i^2, R_i^1, R_i^2, c_i^1, c_i^2, c_i^3, c_i^4)$ such that the following hold:
\begin{enumerate}
\item $\sset{c_i^1, c_i^2} \subseteq \sset{1,2,3,4}$. 
\item $1 \geq |S_i^1| \geq |S_i^2| \geq  |R_i^1| \geq |R_i^2|$. 
\item $S_i^1 \neq \emptyset$ if and only if one of the following holds: \label{empty}
  \begin{itemize}
  \item there is a path $a-b-c-d$ with $a \in Y_1, b, d \in Y_2, c \in Y_3$ and $N(a) \cap S = T_i$; or
  \item the conditions of the last bullet hold for $P$ and there is a path $a-b-c$ with $a \in Y_1, b \in Y_2, c \in Y_3$ and $N(a) \cap S = T_i$. 
  \end{itemize}
\item $S_i^1 \cup S_i^2$ is a stable set, and 
$S_i^1 \cup S_i^2 \subseteq Y_1(T_i)$.
\item If $S_i^1 = \{s_i^1\}$, then $s_i^1$ has a neighbor in $Y_2$.
\item If $S_i^2 = \{s_i^2\}$, then $s_i^2$ has a neighbor in $Y_2$.
\item If $S_i^2 \neq \emptyset$, then $\sset{c_i^1, c_i^2} = L_1 \setminus (L_2 \cap L_3)$ and $c_i^1 \in L_3, c_i^2 \in L_2$. 
\item $R_i^1 \subseteq (N(S_i^1) \setminus N(S_i^2)) \cap Y_2.$
\item $R_i^2 \subseteq (N(S_i^2) \setminus N(S_i^1)) \cap Y_3.$
\item $R_i^1 \cup R_i^2$ is a stable set.
\item $\sset{c_i^3, c_i^4} \subseteq L_2 \cap L_3$. 
\end{enumerate}

We let $S'^Q = \bigcup_{i=1}^r (S_i^1 \cup S_i^2)$ and $T'^Q = \bigcup_{i=1}^r (R_i^1 \cup R_i^2)$. Define $f'^Q : S'^Q \cup T'^Q \rightarrow \sset{1,2,3,4}$ by setting $f'^Q(v) = c_i^j$ if $S_i^j = \sset{v}$ for $j = 1,2$ and $f'^Q(v) = c_i^{j+2}$ if $R_i^j = \sset{v}$ for $j=1,2$. Let $S'_1$ be the set of $v \in (T'^Q \cup S'^Q)$ such that $f'^Q(v) \in L_2 \cap L_3$. Let $S_2'$ be the set of $v \in (T'^Q \cup S'^Q)$ such that $f'^Q(v) \in L_2 \setminus L_3$, and let $S_3'$ be the set of $v \in (T'^Q \cup S'^Q)$ such that $f'^Q(v) \in L_3 \cap L_2$. 
Let
$$\tilde{X}^Q = (N(S_1') \cap (Y_{1} \cup Y_{2} \cup Y_{3})) \cup (N(S_2') \cap (Y_{1} \cup Y_{2})) \cup (N(S_3') \cap (Y_{1} \cup Y_{2})) \cup (N(T'^Q) \cap (Y_{2} \cup Y_{3})).$$

For $i \in \sset{1, \dots, r}$, we further define $\tilde{Z}_i = \emptyset$ if $|S_i^1 \cup S_i^2| < 2$ or $|R_i^1| > 0$, and $\tilde{Z}_i = (N(S_i^1) \setminus N(S_i^2)) \cap Y_2$ otherwise. We let $\tilde{Z}^Q = \bigcup_{i=1}^r \tilde{Z}_i$. Let $g^Q: \tilde{Z} \rightarrow L_2 \setminus L_3$ be the constant function.
For $i \in \sset{1, \dots, r}$, we let $\tilde{Y}_i = \emptyset$ if $|S_i^1 \cup S_i^2| < 2$ or $|R_i^1 \cup R_i^2| \neq 1$, and $\tilde{Y}_i = (N(S_i^2) \setminus (N(S_i^1) \cup N(R_i^1))) \cap Y_3$ otherwise. We let $\tilde{Y}^Q = \bigcup_{i=1}^r \tilde{Y}_i$. Let $g'^Q: \tilde{Y} \rightarrow L_3 \setminus L_2$ be the constant function. 
For $i \in \sset{1, \dots, r}$, we let $\tilde{W}_i = \emptyset$ if $|S_i^1 \cup S_i^2| \neq 1$ or $c_i^1 \in L_1 \cap L_2 \cap L_3$, and $\tilde{W}_i = Y_1(T_i) \setminus S_i^1$ otherwise. We let $\tilde{W}^Q = \bigcup_{i=1}^r \tilde{W}_i$. We define $g''^Q: \tilde{W} \rightarrow L_1$ by setting $g''(\tilde{W}_i \setminus N(S_i^1)) = \sset{c_i^1}$ and $g''(\tilde{W}_i \cap N(S_i^1)) = L_1 \setminus (\sset{c_i^1} \cup (L_2 \cap L_3)$. 

Let $P'^Q$ be the starred precoloring 
$$(G, S \cup S'^Q \cup T'^Q, X_0 \cup \tilde{W}^Q \cup \tilde{Y}^Q \cup \tilde{Z}^Q, X \cup \tilde{X}^Q, Y \setminus (S'^Q \cup T'^Q \cup \tilde{W}^Q \cup \tilde{X}^Q \cup \tilde{Y}^Q \cup \tilde{Z}^Q), Y^*, f \cup f'^Q \cup g^Q \cup g'^Q \cup g''^Q).$$
Since $P$ satisfies \eqref{it:122}, it follows that $P'$ satisfies \eqref{it:122} as well. We let $\mathcal{L} = \sset{P'^Q: Q \in \mathcal{Q},  f \cup f'^Q \cup g^Q \cup g'^Q \cup g''^Q \mbox{ is a proper coloring}}$.

\vspace*{-0.4cm}\begin{equation}\vspace*{-0.4cm} \label{eq:equiv}
  \longbox{\emph{$\mathcal{L}$ is an equivalent collection for $P$.}}
\end{equation}

Let $L_1 = \sset{c^1, c^2, c^3}, L_2 = \sset{c^1, c^2, c^4}$ and $L_3 = \sset{c^1, c^3, c^4}$. Let $Y_1^*$ denote the set of vertices in $Y_1$ with a neighbor in $Y_2$. 
Every precoloring extension for $P'^Q \in \mathcal{L}$ is a precoloring extension for $P$. Now suppose that $P$ has a precoloring extension $c : V(G) \rightarrow \sset{1,2,3,4}$. We define an $r$-tuple $(Q_1, \dots, Q_r)$, where for each $i$, $Q_i = (S_i^1, S_i^2, R_i^1, R_i^2, c_i^1, c_i^2, c_i^3, c_i^4)$. For $i \in \sset{1, \dots, r}$, we define $Q_i = (S_i^1, S_i^2, R_i^1, R_i^2, c_i^1, c_i^2, c_i^3, c_i^4)$ as follows:
\begin{itemize}
\item If neither bullet of \ref{empty} is satisfied, we let $Q_i = (\emptyset, \emptyset, \emptyset, \emptyset, c^1, c^1, c^1, c^1)$. 
\item If $Y_1^*(T_i)$ contains a vertex $v$ with $c(v) = c^1$, we let $Q_i = (\sset{v}, \emptyset, \emptyset, \emptyset, c^1, c^1, c^1, c^1)$. 
\item If $Y_1^*(T_i)$ contains a vertex $v$ with $c(v) = c^2$ such that $c(Y_1^*(T_i) \setminus N(v)) \subseteq \sset{c^3}$, we let $Q_i = (\sset{v}, \emptyset, \emptyset, \emptyset, c^2, c^1, c^1, c^1)$.  
\item If $Y_1^*(T_i)$ contains a vertex $v$ with $c(v) = c^3$ such that $c(Y_1^*(T_i) \setminus N(v)) \subseteq \sset{c^2}$, we let $Q_i = (\sset{v}, \emptyset, \emptyset, \emptyset, c^3, c^1, c^1, c^1)$.
\item Let $u, v \in Y_1^*(T_i)$ such that $c(u) = c^2, c(v) = c^3$ and $uv \not\in E(G)$. We let $A = (N(u) \setminus N(v)) \cap Y_2$ and $B = (N(v) \setminus N(u)) \cap Y_3$. We proceed as follows:
  \begin{itemize}
  \item If $c(A) \subseteq L_2 \setminus L_3$, we let $Q_i = (\sset{u}, \sset{v}, \emptyset, \emptyset, c^2, c^3, c^1, c^1)$.
  \item If there is a vertex $x \in A$ such that $c(x) \in L_2 \cap L_3$ and $c(B \setminus N(x)) \subseteq L_3 \setminus L_2$, we let $Q_i = (\sset{u}, \sset{v}, \sset{x}, \emptyset, c^2, c^3, c(x), c^1)$.
  \item If there is $x \in A$ and $y \in B$ such that $c(x), c(y) \in L_2 \cap L_3$ and $xy \not\in E(G)$, we let $Q_i = (\sset{u}, \sset{v}, \sset{x}, \sset{y}, c^2, c^3, c(x), c(y))$. 
  \end{itemize}
\end{itemize}
It follows from the definitions of $\tilde{Y}^Q, \tilde{Z}^Q, \tilde{W}^Q$ that $c|_{(\tilde{Y}^Q \cup \tilde{Z}^Q \cup \tilde{W}^Q)} = g^Q|_{\tilde{Z}^Q} \cup g'^Q|_{\tilde{Y}^Q} \cup g''^Q|_{\tilde{W}^Q}$. It follows that $Q \in \mathcal{Q}$ and $c$ is a precoloring extension of $P'^Q$. Thus $\mathcal{L}$ is an equivalent collection for $P$, which proves \eqref{eq:equiv}. 

\bigskip

Let $Q \in \mathcal{Q}$ and let $P'^Q \in \mathcal{L}$ with $P'^Q = (G, S', X_0', X', Y', Y^*, f')$, and let $Y_i' = \sset{y \in Y': L_{P'}(y) = L_i}$ for $i = 1,2,3$. We claim the following. 

\vspace*{-0.4cm}\begin{equation}\vspace*{-0.4cm} \label{eq:empty}
  \longbox{\emph{For every $i \in \sset{1, \dots, r}$ such that $S_i^1 = \sset{u}, S_i^2 = \sset{v}$, we have that $N(u) \cap (Y_2' \cup Y_3')$ is anticomplete to $N(v) \cap (Y_2' \cup Y_3')$.}}
\end{equation}

From the properties of $Q$, we know that $f'(u) \in L_1 \cap L_3$ and $f'(v) \in L_1 \cap L_2$. Since $u,v \in S'$,  it follows that $N(u) \cap Y_3' = \emptyset$, since $N(u) \cap Y_3 \subseteq \tilde{X}^Q$; similarly, $N(v) \cap Y_2' = \emptyset$. We let $A = (N(u) \setminus N(v)) \cap Y_2$ and $B = (N(v) \setminus N(u)) \cap Y_3$. It follows that $v$ is anticomplete to $A$ and $u$ is anticomplete to $B$. Let $a_1, \ldots, a_t$ be the components of $G|A$, and
let $b_1, \ldots, b_s$ be the components of $G|B$. Since $P$ satisfies \eqref{it:122}, it follows that for every $i \in [t]$ and $j \in [s]$,
$V(a_i)$ is either complete or anticomplete to $V(b_j)$.

Let $H$ be the graph with vertex set  
$\{u,v\} \cup \{a_1, \ldots, a_t\} \cup \{b_1, \ldots b_s\}$; 
where $N_H(u)=\{a_1, \ldots, a_t\}$, $N_H(v)=\{b_1, \ldots, b_s\}$,
the sets $\{a_1, \ldots, a_t\}$ and $\{b_1, \ldots, b_s\}$ are stable,
and $a_i$ is adjacent to $b_j$ if and only if $V(a_i)$ is complete to $V(b_j)$ 
in $G$.  Apply \ref{lem:nunv} to $H$, $u$ and $v$ to  obtain a 
partition
$A'_0, A'_1, \dots, A'_k$ of $\{a_1, \ldots, a_t\}$ and 
a partition $B'_0, B'_1, \dots, B'_k$ of $\{b_1, \ldots, b_t\}$. 
For $i \in [k]$, let $A_i=\bigcup_{a_j \in A_i}V(a_j)$ and 
$B_i=\bigcup_{b_j \in B_i}V(b_j)$. 

It follows from the definition of $H$ that in $G$,
\begin{itemize}
\item $A_0$ is complete to $B$;
\item $B_0$ is complete to $A$; and 
\item for $j = 1, \dots, k$, $A_j, B_j \neq \emptyset$ and $A_j$ is complete to $B \setminus B_j$ and $B_j$ is complete to $A \setminus A_j$, and $A_j$ is anticomplete to $B_j$.
\end{itemize}

If $R_i^1 = \emptyset$, then $A \subseteq \tilde{Z}^Q$, and so $A \cap Y' = \emptyset$, and \eqref{eq:empty} follows. Thus $R_i^1 \neq \emptyset$. Suppose that $R_i^2 = \emptyset$. Then one of the following holds:
\begin{itemize}
\item $R_i^1 \subseteq A_0$, and so $B \subseteq \tilde{X}^Q$; or
\item $R_i^1 \subseteq A_j$ for some $j > 0$, and so $B \setminus B_j \subseteq \tilde{X}^Q$ and $B_j \subseteq \tilde{Y}^Q$.
\end{itemize}
It follows that $N(v) \cap Y_2' = \emptyset$, and \eqref{eq:empty} follows. 
Thus we may assume that $R_i^2 \neq \emptyset$, then there exists a $j > 0$ such that $R_i^1 \subseteq A_j$ and $R_i^2 \subseteq B_j$, and so $(A \setminus A_j) \cup (B \setminus B_j) \subseteq \tilde{X}^Q$, and again, \eqref{eq:empty} holds.

\vspace*{-0.4cm}\begin{equation}\vspace*{-0.4cm} \label{eq:1232}
  \longbox{\emph{There is no path $z-a-b-c$ with $z \in Y_1'$, $a, c \in Y_2'$ and $b \in Y_3'$.}}
\end{equation}

Suppose that $z-a-b-c$ is such a path. Let $i \in \sset{1, \dots, r}$ such that $N(z) \cap S = T_i$. Since $z \not\in X_0'$, it follows that $S_i^1 \neq \emptyset$. Write $S_i^1=\{u\}$.  Let $s \in T_i$; then  
$f'(s) \in L_2 \cup L_3$,  and therefore $s$ is  anticomplete to $\{a,b,c\}$.

Suppose that $S_i^2 = \emptyset$.  Then $f'(u) \in L_2 \cap L_3$, and thus $u$ 
is  non-adjacent to $z, a, b, c$. Now $u-s-z-a-b-c$ is a  $P_6$ in $G$, a 
contradiction. 
Thus it follows that $S_i^2 = \sset{v}$, and $z$ is non-adjacent to $u$ and $v$. By construction, it follows that $f'(u) \in L_2 \setminus L_3$, and $f'(v) \in L_3 \setminus L_2$. 
Since neither $u-s-z-a-b-c$ nor $v-s-z-a-b-c$ is a $P_6$  in $G$, it follows that $u, v$ each have a neighbor in $\sset{a,b,c}$. Since neighbors of $u$ in $Y_2$ are in $\tilde{X}^Q$, it follows that $u$ is non-adjacent to $a$ and $c$, and hence $u$ is adjacent to $b$.  Since neighbors of $v$ in $Y_3$ are in $\tilde{X}^Q$, it follows that $v$ is non-adjacent to $b$, and $v$ is adjacent to $a$ or $c$. This contradicts \eqref{eq:empty}, and thus \eqref{eq:1232} follows. 

\vspace*{-0.4cm}\begin{equation}\vspace*{-0.4cm} \label{eq:123}
  \longbox{\emph{If the conditions of the last bullet hold for $P$, then there is no path $z-a-b$ with $z \in Y_1'$, $a \in Y_2'$ and $b \in Y_3'$.}}
\end{equation}

Suppose not, and let $z-a-b$ be such a path. Let $i \in \sset{1, \dots, r}$ such that $N(z) \cap S = T_i$. Let $s \in T_i$. Then $f'(s) \in L_2 \cap L_3$, since $f'(s) \not\in L_1$, and hence $s$ is anticomplete to $a, b$.
Since $z \not\in X_0'$, it follows that $S_i^1 \neq \emptyset$, say $S_i^1 = \sset{u}$.  Suppose first that $S_i^2 = \emptyset$. Since $z \not\in X_0'$, it follows that $f'(u) \in L_2 \cap L_3$, and thus $u$ is non-adjacent to $z, a, b$.   By construction, $u$ has a neighbor $y$ in $Y_2$, and since $u$ is anticomplete to $a, b$, it follows that $y \neq a, b$. Since $y-u-s-z-a-b$ is not a $P_6$ in $G$, it follows that $y$ has a neighbor in $\sset{z,a,b}$. Since $P$ satisfies \eqref{it:122}, it follows that $u-y-a$ is not a path and so $y$ is not adjacent to $a$. Since $P$ satisfies the second-to-last bullet for $L_1, L_2, L_3$, it follows that $u-y-b-a$ is not a path, and so $u$ is not adjacent to $b$. But then $u$ is adjacent to $z$; and $b-a-z-u$ is a  path contrary to the second-to-last bullet for $L_3, L_2, L_1$. This is a contradiction, and hence $S_i^2 \neq \emptyset$, say $S_i^2 = \sset{u}$. 

By construction, it follows that $f'(u) \in L_2 \setminus L_3$, and $f'(v) \in L_3 \setminus L_2$. If one of $u, v$ has no neighbor in $\sset{a,b}$, then we reach a contradiction as above. Since neighbors of $u$ in $Y_2$ are in $\tilde{X}^Q$, it follows that $u$ is adjacent to $b$, but not $a$. Since neighbors of $v$ in $Y_3$ are in $\tilde{X}^Q$, it follows that $v$ is adjacent to $a$, but not $b$. This contradicts \eqref{eq:empty}, and proves \eqref{eq:123}.

\bigskip
We now replace every $P' \in \mathcal{L}$ by $P''$ satisfying \eqref{it:size3} by moving vertices with lists of size less than three from $Y'$ to $X'$. It follows that $P''$ still satisfies \eqref{it:122} and \eqref{eq:1232}. This concludes the proof of the lemma. 
\end{proof}

\begin{lemma} \label{lem:123}
There is a function $q: \mathbb{N} \rightarrow \mathbb{N}$ such that the following holds. Let $P = \sspc$ be a starred precoloring of a $P_6$-free graph $G$ with $P$ satisfying \eqref{it:size3} and \eqref{it:122}. Then there is 
an algorithm with running time $O(|V(G)|^{q(|S|)})$
that outputs an equivalent collection $\mathcal{L}$ for $P$ such that
\begin{itemize}
\item $|\mathcal{L}| \leq |V(G)|^{q(|S|)}$;
\item every $P' \in \mathcal{L}$ is a starred precoloring of $G$;
\item every $P' \in \mathcal{L}$ with seed $S'$ satisfies $|S'| \leq q(|S|)$;  and
\item every $P' \in \mathcal{L}$ satisfies \eqref{it:size3}, \eqref{it:122} and \eqref{it:123}. 
\end{itemize}
Moreover, for every $P' \in \mathcal{L}$, given a precoloring extension of $P'$, we can compute a precoloring extension for $P$ in polynomial time.
\end{lemma}
\begin{proof} 
Let $\mathcal{L}=\{P\}$. For every triple $(L_1,L_2,L_3)$ of distinct lists 
of size three included in $[4]$ we  repeat the following. Apply 
Lemma~\ref{lem:fixedl1l2l3} to every member of $\mathcal{L}$; replace 
$\mathcal{L}$ by the union of the collections thus obtained, and 
move to the next triple of lists.  At the end of this process we 
have  an equivalent collection $\mathcal{L}$ for $P$, in which every starred 
precoloring satisfies the second-to-last bullet of Lemma~\ref{lem:fixedl1l2l3} for every $(L_1, L_2, L_3)$. 

Repeat the procedure described in the previous paragraph.
Since the second-to-last bullet of the conclusion of Lemma~\ref{lem:fixedl1l2l3} holds for each starred precoloring we input this time, it follows that the last bullet of Lemma~\ref{lem:fixedl1l2l3} holds for the output for every 
$(L_1, L_2, L_3)$. Thus \eqref{it:123} holds; this concludes the proof. 
\end{proof}

Let $P=\sspc$ be a starred precoloring.
For $W \subseteq V(G)$ and $L \subseteq [4]$, we say that $W$ {\em meets}
$L$ if $L_P(w)=L$ for some $w \in W$.
We now have the following convenient property.
\begin{lemma} 
\label{lem:Ycomps}
Let $P=\sspc$ be a starred precoloring of a $P_6$-free graph $G$
satisfying \eqref{it:size3}, \eqref{it:122} and \eqref{it:123}. 
Let $L_1, L_2,L_3,L_4$ be the subsets of $[4]$ of size three.
Let $C$ be  a component of $G|Y$ that meets at least three of the lists 
$L_1, L_2,L_3,L_4$.
For $i \in [4]$, let $C_i=\{v \in V(C) \; : \; L_P(v)=L_i\}$.
Then for every $i \neq j$, $C_i$ is complete to $C_j$.
\end{lemma}

\begin{proof}
Let $P=p_1- \ldots-p_k$ be a path such that
for some $i \neq j$ $p_1 \in C_i$, $p_k \in C_j$, $p_1$ is non-adjacent to $p_k$,
and subject to that with $k$ minimum. Since $P$ satisfies \eqref{it:122},
it follows that $p_2 \not \in C_i$; say $p_2 \in C_l$. 
Since $P$ satisfies \eqref{it:122} and \eqref{it:123}, it follows that
$p_3 \in C_i$. Similarly, $p_4 \not \in C_i$. By the minimality of $k$,
we deduce that $k=4$. By \eqref{it:123} applied to $p_2-p_3-p_4$,
we deduce that $l=j$. Let $C'$ be a component of $C|(C_i \cup C_j)$ with
$p_1, \ldots, p_4 \in V(C')$. Since $C$ is connected, and
since $V(C) \neq C_i \cup C_j$, there exists $c \in C_l$
with $l \neq i,j$ such that $c$ has a neighbor in $C'$. Since
$P$ satisfies \eqref{it:122} and \eqref{it:123}, it follows from 
Lemma~\ref{lem:mixed} that $c$ is complete to  $C'$.  But now
$p_1-c-p_4$ contradicts the fact that $P$ satisfies \eqref{it:123}.
This proves Lemma~\ref{lem:Ycomps}.
\end{proof}

Our next goal is to establish axiom \eqref{it:x11}, which we restate. 
\begin{enumerate}
\item[\eqref{it:x11}] Let $L_1 \subseteq \sset{1,2,3,4}$ with $|L_1| = 3$. Then there is no path $a-b-c$ with $L_P(b) = L_P(c) = L_1$ and $a \in X$, $b, c \in Y$.
\end{enumerate}

\begin{lemma} \label{lem:x11}
There is a function $q: \mathbb{N} \rightarrow \mathbb{N}$ such that the following holds. Let $P = \sspc$ be a starred precoloring of a $P_6$-free graph $G$ with $P$ satisfying \eqref{it:size3}. Then there is an
algorithm with running time $O(|V(G)|^{q(|S|)})$
that outputs an equivalent collection $\mathcal{L}$ for $P$ such that
\begin{itemize}
\item $|\mathcal{L}| \leq |V(G)|^{q(|S|)}$;
\item every $P' \in \mathcal{L}$ is a starred precoloring of $G$;
\item every $P' \in \mathcal{L}$ with seed $S'$ satisfies $|S'| \leq q(|S|)$;  and
\item every $P' \in \mathcal{L}$ satisfies \eqref{it:size3}, \eqref{it:122}, \eqref{it:123} and \eqref{it:x11}. 
\end{itemize}
Moreover, for every $P' \in \mathcal{L}$, given a precoloring extension of $P'$, we can compute a precoloring extension for $P$ in polynomial time. 
\end{lemma}
\begin{proof}
Let $\mathcal{L}=\{P\}$. For every list $L \subseteq \sset{1,2,3,4}$ of size three, apply Lemma \ref{lem:pairoflists} to every member of $\mathcal{L}$, 
replace $\mathcal{L}$ by the union of the equivalent collections thus obtained,
and move to the next list. At the end of the process we obtained
the required equivalent collection for $\{P\}$.
\end{proof}

We now begin to establish the following axiom, which we restate below.
\begin{enumerate}
\item[\eqref{it:x12}] Let $L_1, L_2 \subseteq \sset{1,2,3,4}$ with $|L_1| = |L_2| = 3$. Then there is no  path $a-b-c$ with $L_P(b) = L_1, L_P(c) = L_2$ and $a \in X$ with $L_3=L_P(a) \neq L_1 \cap L_2$.
\end{enumerate}

We define the following auxiliary statement: 

\vspace*{-0.4cm}\begin{equation}\vspace*{-0.4cm} \label{it:star}
  \longbox{\emph{Let $L_1, L_2 \subseteq \sset{1,2,3,4}$ with $|L_1| = |L_2| = 3$. Then there is no  path $a-b-c-d$ with $L_P(b) = L_P(d) = L_1, L_P(c) = L_2$ and $a \in X$ with $L_3 = L_P(a) \neq L_1 \cap L_2$.}}
\end{equation}

\begin{lemma}\label{lem:fixedl1l2l3x} There is a function $q: \mathbb{N} \rightarrow \mathbb{N}$ such that the following holds. 
Let $L_1, L_2\subseteq \sset{1,2,3,4}$ with $|L_1| = |L_2| = 3$  and $L_1 \neq L_2$, and let $L_3 \subseteq \sset{1,2,3,4}$ with $|L_3| = 2$ and $L_3 \neq L_1 \cap L_2$. Let $P = \sspc$ be a starred precoloring of a $P_6$-free graph $G$ with $P$ satisfying \eqref{it:size3}, \eqref{it:122}, \eqref{it:123} and \eqref{it:x11}. Then there is an
algorithm with running time $O(|V(G)|^{q(|S|)})$
that outputs an equivalent collection $\mathcal{L}$ for $P$ such that
\begin{itemize}
\item $|\mathcal{L}| \leq |V(G)|^{q(|S|)}$;
\item every $P' \in \mathcal{L}$ is a starred precoloring of $G$;
\item every $P' \in \mathcal{L}$ with seed $S'$ satisfies $|S'| \leq q(|S|)$;  
\item every $P' \in \mathcal{L}$ satisfies \eqref{it:size3}, \eqref{it:122}, \eqref{it:123} and \eqref{it:x11};  
\item every $P' \in \mathcal{L}$ satisfies \eqref{it:star} for every three lists $L_1', L_2', L_3'$ such that $P$ satisfies \eqref{it:star} for $L_1', L_2', L_3'$;   
\item if $P$ satisfies \eqref{it:star} for every three lists, then every $P' \in \mathcal{L}$ satisfies \eqref{it:x12} for every three lists $L_1', L_2', L_3'$ such that $P$ satisfies \eqref{it:x12} for $L_1', L_2', L_3'$;   
\item every $P' \in \mathcal{L}$ satisfies \eqref{it:star} for $L_1, L_2, L_3$. 
\item if $P$ satisfies \eqref{it:star} for every three lists $L_1', L_2', L_3'$ such that $|L_1'| = |L_2'| = 3, L_1' \neq L_2', |L_3'| = 2, L_3' \neq L_1' \cap L_2'$, then every $P' = (G, S', X_0', X', Y', Y^*, f') \in \mathcal{L}$ satisfies that there is no  path $a-b-c$ with $L_{P'}(a) = L_3$, $L_{P'}(b) = L_1$, $L_{P'}(c) = L_2$ with $a \in X$, $b, c \in Y'$.
\end{itemize}
Moreover, for every $P' \in \mathcal{L}$, given a precoloring extension of $P'$, we can compute a precoloring extension for $P$ in polynomial time. 
\end{lemma}
\begin{proof}
  Let $\mathcal{L} = \emptyset$. Let $Y_i = \sset{y \in Y: L_P(y) = L_i}$ for $i = 1,2$, and let $X_3$ be the set of vertices $v$ in $X$ with list $L_3$  such that $v$ starts a path $v-b-c-d$ ($v-b-c$ if the condition of the last bullet holds for $P$) with $v \in X, b,d \in Y_1, c \in Y_2$. 
Let $L_4, L_5$ be the two three-element lists in $\sset{1,2,3,4}$ that are not $L_1, L_2$, and let $Y_i = \sset{y \in Y: L_P(y) = L_i}$ for $i = 4,5$. We call a component $C$ of $G|Y$ \emph{bad} if $V(C) \cap Y_1 \neq \emptyset, V(C) \cap Y_2 \neq \emptyset$ and $V(C) \cap Y_i \neq \emptyset$ for some $i \in \sset{4,5}$. 

Let $\mathcal{T} = \sset{T_1, \dots, T_r}$ be the set of types $T \subseteq S$ with $f(T) = \sset{1,2,3,4} \setminus L_3$. We let $\mathcal{Q}$ be the set of all $r$-tuples $(Q_1, \dots, Q_r)$, where for each $i$, $$Q_i = (S_i^1, S_i^2, R_i^1, R_i^2, R_i^3, R_i^4, C_i^1, C_i^2, X_i^{1,1},X_i^{1,2},X_i^{2,1}, X_i^{2,2}, f_i, case_i)$$ such that the following hold:
\begin{enumerate}
\item $f_i : S_i^1 \cup S_i^2 \cup R_i^1 \cup R_i^2 \cup R_i^3 \cup R_i^4 \cup X_i^{1,1} \cup X_i^{1,2} \cup X_i^{2,1} \cup X_i^{2,2} \rightarrow \sset{1,2,3,4}$. 
\item $f_i(S_i^1 \cup S_i^2) \subseteq L_3$.
\item $1 \geq |S_i^1| \geq |S_i^2| \geq  |R_i^1| \geq |R_i^2| \geq |R_i^3| \geq |R_i^4|$.
\item $S_i^1 \cup S_i^2$ is a stable set and $S_i^1 \cup S_i^2 \subseteq X_3(T_i)$. 
\item If $S_i^1 = \emptyset$, then $X_3(T_i) = \emptyset$.
\item If $S_i^2 \neq \emptyset$, then $f_i(S_i^1 \cup S_i^2) = L_3$ and $L_3 \cap L_1 \cap L_2  = \emptyset$.
\item For $j = 1,2$, if $S_i^j = \sset{s_i^j}$ and $s_i^j$ is mixed on a bad component, then $C_i^j$ is the vertex set of a bad component on which $s_i^j$ is mixed; otherwise, $C_i^j = \emptyset$.
\item For $j,k = 1,2$, $|X_i^{j,k}| \leq 1$, and $|X_i^{j,k}| = 1$ if and only if $C_i^j \neq \emptyset$.
\item For $j = 1,2$, if $C_i^j \neq \emptyset$, then there exist $p \neq q$ such that $X_i^{j,1} \cap C_i^j \cap Y_p \neq \emptyset$ and $X_i^{j,2} \cap C_i^j \cap Y_q \neq \emptyset$.
\item For $j = 1,2,3,4$, $f_i(R_i^j) \subseteq L_1 \cap L_2$.
\item $case_i \in \sset{\emptyset, (a), (b), (c), (d), (e), (f)}$.
\item $case_i \in \sset{\emptyset, (a), (b)}$ if and only if $R_i^j = \emptyset$ for all $j \in \sset{1,2,3,4}$.
\item $case_i \in \sset{(c), (d), (e)}$ if and only if $R_i^3, R_i^4 = \emptyset$ and $R_i^1, R_i^2 \neq \emptyset$.
\item $case_i = (f)$ if and only if $R_i^j \neq \emptyset$ for all $j \in \sset{1,2,3,4}$. 
\item If $S_i^2 = \emptyset$, then $case_i = \emptyset$.
\item If $case_i \neq \emptyset$, then let $\sset{u,v} = S_i^1 \cup S_i^2$ such that $u \in S_i^1$ if and only if $f_i(u) \in L_1$; then $R_i^1, R_i^3 \subseteq N(u) \cap (Y_2 \setminus N(v))$ and $R_i^2, R_i^4 \subseteq N(v) \cap (Y_1 \setminus N(u))$.
\item If $case_i = (c)$, $R_i^1$ is anticomplete to $R_i^2$.
\item If $case_i \in \sset{(d), (e)}$, $R_i^1$ is complete to $R_i^2$.
\item If $case_i = (f)$, then $R_i^1$ is complete to $R_i^2$ and anticomplete to $R_i^4$, and $R_i^3$ is anticomplete to $R_i^2$ and anticomplete to $R_i^4$. 
\end{enumerate}
We let $$S'^Q = \bigcup_{i \in \sset{1, \dots, r}} (S_i^1 \cup  S_i^2 \cup  R_i^3 \cup R_i^4 \cup X_i^{1,1} \cup X_i^{1,2} \cup X_i^{2,1} \cup X_i^{2,2}) \cup \bigcup_{i \in \sset{1, \dots, r}, case_i \neq (c)} (R_i^1 \cup R_i^2),$$
and let $f'^Q  = f_1 \cup \dots \cup f_r$. 

For every $i \in \sset{1, \dots, r}$, we let $\tilde{Y}_i = \bigcup_{j,k \in \sset{1,2}} \bigcup_{p \in \sset{1,2,4,5}, X_i^{j,k} \cap C_i^j \cap Y_p \neq \emptyset}(C_i^j \cap Y_p)$, and we let $h_i(C_i^j \cap Y_p \cap \tilde{Y}_i) \subseteq f_i(X^{j,k}_i)$.
Let $\tilde{Z}_i = (C_i^1 \cup C_i^2) \setminus \tilde{Y}_i$. Let $\tilde{Y}^Q = \bigcup_{i \in \sset{1, \dots, r}} \tilde{Y}_i$ and $\tilde{Z}^Q = \bigcup_{i \in \sset{1, \dots, r}} \tilde{Z}_i$ and $h^Q = h_1 \cup \dots \cup h_r$.

Let $S'_1$ be the set of $v \in S'^Q$ such that $f'(v) \in L_1 \cap L_2$; let $S_2'$ be the set of $v \in S'^Q$ such that $f'(v) \in L_1 \setminus L_2$, and let $S_3'$ be the set of $v \in S'^Q$ such that $f'(v) \in L_2 \setminus L_1$. Let
$$\tilde{X}^Q = (N(S_1') \cap (Y_{1} \cup Y_{2})) \cup (N(S_2') \cap (Y_{1})) \cup (N(S_3') \cap (Y_{2})).$$

Let $\tilde{W}_i = X_3(T_i)$ if $S_i^1 = \sset{v}, S_i^2 = \emptyset$ and $f'(v) \not\in L_1 \cap L_2 \cap L_3$, and $\tilde{W}_i = \emptyset$ otherwise. If $\tilde{W}_i \neq \emptyset$, we let $g''_i : \tilde{W}_i \rightarrow L_3$ such that $g''(y) = f'(v)$ is $y$ if non-adjacent to $v$, and $g''(y)$ is the unique color in $L_3 \setminus (\sset{f'(v)})$ otherwise. Let $\tilde{W}^Q = \bigcup_{i \in \sset{1, \dots, r}} \tilde{W}_i$ and let $g''^Q = g_1''^Q \cup \dots \cup g_r''^Q$.

Let $\tilde{V}^Q$ be the set of vertices $v$ in $X$ with list $L_3$ such that $S'^Q$ contains a neighbor $s$ of $v$, and let $h'^Q : \tilde{V} \rightarrow L_3$ such that $h'^Q(v) \in L_3 \setminus (f'(s))$.

Let $\tilde{U}_i$ be the set of all vertices $x \in X_3(T_i)$ such that $S_i^1 = \sset{v}$ and such that $f'(v) \in L_1 \cap L_2$ and $N(v) \cap Y_1 \subsetneq N(x) \cap Y_1$, and let $g_i: \tilde{U}_i \rightarrow L_3 \setminus (L_1 \cap L_2)$. Let $\tilde{U}^Q = \bigcup_{i \in \sset{1, \dots, r}} \tilde{U}_i$ and $g^Q = g_1 \cup \dots \cup g_r$.

Let $\tilde{U'}_i$ be the set of all vertices $x \in X_3(T_i)$ such that $S_i^1 = \sset{u}, S_i^2 = \sset{v}$ such that $xu \not\in E(G)$, and $N(v) \cap Y_1 \subsetneq N(x) \cap Y_1$, and let $g'_i: \tilde{U}_i \rightarrow \sset{f'(u)}$. Let $\tilde{U'}^Q = \bigcup_{i \in \sset{1, \dots, r}} \tilde{U'}_i$ and $g'^Q = g'_1 \cup \dots \cup g'_r$.

Finally, we define $\tilde{T}_i$ as follows: If $case_i = \emptyset$, then $\tilde{T}_i = \emptyset$. Otherwise, let $\sset{u,v} = S_i^1 \cup S_i^2$ such that $f'(u) \in L_1$, and let $A = N(u) \cap (Y_2 \setminus N(v))$ and $B = N(v) \cap (Y_1 \setminus N(u))$.  If $case_i = $
\begin{enumerate}[(a)]
  \item then $\tilde{T}_i = A$; 
  \item then $\tilde{T}_i = B$; 
  \item then $\tilde{T}_i = (A \cap N(R_i^2)) \cup (B \cap N(R_i^1))$; 
  \item then $\tilde{T}_i = B \setminus N(R_i^1)$; 
  \item then $\tilde{T}_i = A \setminus N(R_i^2)$; 
  \item then $\tilde{T}_i = \emptyset$.
  \end{enumerate}

We let $\tilde{T}^Q = \bigcup_{i \in \sset{1, \dots, r}} \tilde{T}_i$ and let $h''^Q: \tilde{T}^Q \rightarrow (L_1 \setminus L_2) \cup (L_2 \setminus L_1)$ be the unique function such that $h''^Q(v) \in L_P(v)$ for all $v \in \tilde{T}^Q$. 

The following statement could be proved using Lemma \ref{lem:nunv}, but we give a shorter proof here: 

\vspace*{-0.4cm}\begin{equation}\vspace*{-0.4cm} \label{eq:eq16}
  \longbox{\emph{Let $i$ such that $\sset{u,v} = S_i^1 \cup S_i^2$ and $f(u) \in L_1$. Let $R = R_i^1 \cup R_i^2 \cup R_i^3 \cup R_i^4$ if $case_i \neq (c)$ and $R = \emptyset$ otherwise. Then $(N(u) \cap Y_2) \setminus (N(v) \cup \tilde{T}_i \cup N(R)))$ is anticomplete to $(N(v) \cap Y_1) \setminus (N(u) \cup \tilde{T}_i \cup N(R)))$.}}
\end{equation}

Let $A' = A \setminus (\tilde{T}_i \cup N(R))$, $B' = B \setminus (\tilde{T}_i \cup N(R))$; then it suffices to prove that $A'$ is anticomplete to $B'$. If $case_i = (a), (b), (d), (e)$, this follows since $A'$ or $B'$ is empty in each of these cases. In case $(f)$, we have that $G|(\sset{u,v} \cup R)$ is a six-cycle. Since the graph arising from a six-cycle by adding a vertex with exactly one neighbor in the cycle contains a $P_6$, it follows that $A', B' = \emptyset$. In case $(c)$, we let $x'y'$ be an edge from $A'$ to $B'$, and we let $x \in A_i^1, y \in A_i^2$. Then $x-u-x'-y'-v-y$ is a $P_6$ in $G$, a contradiction. Again it follows that $A'$ is anticomplete to $B'$, and \eqref{eq:eq16} follows.

Let $P'^Q$ be the starred precoloring obtained from 
\begin{align*}  
(G, &S \cup S'^Q\\
 &X_0 \cup \tilde{Y}^Q \cup \tilde{W}^Q \cup \tilde{V}^Q \cup \tilde{U}^Q \cup \tilde{U'}^Q \cup \tilde{T}^Q\\
 &(X \setminus (\tilde{W}^Q \cup \tilde{V}^Q \cup \tilde{U}^Q \cup \tilde{U'}^Q))\cup \tilde{Z}^Q \cup \tilde{X}^Q\\
 &Y \setminus (\tilde{Y}^Q \cup \tilde{Z}^Q \cup \tilde{X}^Q \cup \tilde{T}^Q) \\
 &Y^*, f \cup f'^Q \cup h^Q \cup h'^Q \cup h''^Q \cup g^Q \cup g_i'^Q \cup g''^Q)
\end{align*}
by moving every vertex with a list of size at most two $X$, and every vertex with a list of size at most one to $X_0$. Since $P$ satisfies \eqref{it:122} and \eqref{it:123}, it follows that $P'^Q$ satisfies \eqref{it:122} and \eqref{it:123} as well. Moreover, $P'^Q$ satisfies \eqref{it:size3}.

We let $$\mathcal{L} = \sset{P'^Q : Q \in \mathcal{Q},  f \cup f'^Q \cup h^Q \cup h'^Q \cup h''^Q \cup g^Q \cup g_i'^Q \cup g''^Q \mbox{ is a proper coloring}}.$$ 

\vspace*{-0.4cm}\begin{equation}\vspace*{-0.4cm} \label{eq:equivalent}
  \longbox{\emph{$\mathcal{L}$ is an equivalent collection for $P$.}}
\end{equation}

For every $P'^Q \in \mathcal{L}$, every precoloring extension of $P'^Q$ is a precoloring extension of $P$. Conversely, let $c$ be a precoloring extension of $P$, and define $Q = (Q_1, \dots, Q_r)$, where for each $i$, $$Q_i = (S_i^1, S_i^2, R_i^1, R_i^2, R_i^3, R_i^4, C_i^1, C_i^2, X_i^{1,1},X_i^{1,2},X_i^{2,1}, X_i^{2,2}, f_i', case_i)$$
is defined as follows:
\begin{itemize}
\item If $X_3(T_i) = \emptyset$, then $Q_i = (\emptyset, \emptyset, \emptyset, \emptyset, \emptyset, \emptyset, \emptyset, \emptyset, \emptyset, \emptyset, \emptyset, \emptyset, f_i, \emptyset)$, where $f_i$ is the empty function.
\item If $X_3(T_i)$ contains a vertex $v$ with $c(v) \in L_1 \cap L_2$, we choose $v$ with $N(v) \cap Y_1$ maximal and let $S_i^1 = \sset{v}$, $case = \emptyset$. In this case, we let $S_i^2 = \emptyset$. 
\item If $X_3(T_i)$ contains no vertex $v$ with $c(v) \in L_1 \cap L_2$, we let $u \in X_3(T_i)$ with $N(u) \cap Y_1$ maximal, and set $S_i^1 = \sset{u}$. If there is a vertex $v \in X_3(T_i)$ with $c(v) \neq c(u)$ and $uv \not\in E(G)$, we choose $v$ with $N(v) \cap Y_1$ maximal and set $S_i^2 = \sset{v}$; otherwise we let $S_i^2 = \emptyset$.
\item If $S_i^2 = \emptyset$, we let $case_i = \emptyset$ and $R_i^j = \emptyset$ for $j = 1,2,3,4$. Otherwise, we let $\sset{u,v} = S_i^1 \cup S_i^2$ such that $c(u) \in L_1$. We let $A = N(u) \cap (Y_2 \setminus N(v))$ and $B = N(v) \cap (Y_1 \setminus N(u))$.  Let $a_1, \ldots, a_t$ be the components of $G|A$, and
let $b_1, \ldots, b_s$ be the components of $G|B$. Since $P$ satisfies \eqref{it:122}, it follows that for every $i \in [t]$ and $j \in [s]$,
$V(a_i)$ is either complete or anticomplete to $V(b_j)$.

Let $H$ be the graph with vertex set  
$\{u,v\} \cup \{a_1, \ldots, a_t\} \cup \{b_1, \ldots b_s\}$; 
where $N_H(u)=\{a_1, \ldots, a_t\}$, $N_H(v)=\{b_1, \ldots, b_s\}$,
the sets $\{a_1, \ldots, a_t\}$ and $\{b_1, \ldots, b_s\}$ are stable,
and $a_i$ is adjacent to $b_j$ if and only if $V(a_i)$ is complete to $V(b_j)$ 
in $G$.  Apply \ref{lem:nunv} to $H$, $u$ and $v$ to  obtain a 
partition
$A'_0, A'_1, \dots, A'_k$ of $\{a_1, \ldots, a_t\}$ and 
a partition $B'_0, B'_1, \dots, B'_k$ of $\{b_1, \ldots, b_t\}$. 
For $i \in [k]$, let $A_i=\bigcup_{a_j \in A_i}V(a_j)$ and 
$B_i=\bigcup_{b_j \in B_i}V(b_j)$.

It follows from the definition of $H$ that in $G$,
\begin{itemize}
\item $A_0$ is complete to $N(v)$;
\item $B_0$ is complete to $N(u)$; and 
\item for $j = 1, \dots, k$, $A_j, B_j \neq \emptyset$ and $A_j$ is complete to $N(v) \setminus B_j$ and $B_j$ is complete to $N(u) \setminus A_j$, and $A_j$ is anticomplete to $B_j$. 
\end{itemize}

 If $A_0 = B_0 = \emptyset$ and $k=1$, then $A$ is anticomplete to $B$, and we let $case_i = \emptyset$. Otherwise, we consider the following cases, setting $case_i = $ 
\begin{enumerate}[(a)]
\item if $c(A) \subseteq L_2 \setminus L_1$;  \label{it:casea1}
\item if $c(B) \subseteq L_1 \setminus L_2$;  \label{it:casea2}
\item if there is an $i \in \sset{1, \dots, k}$ such that $c(A \setminus A_i) \subseteq L_2 \setminus L_1$, and $c(B\setminus B_i) \subseteq L_1 \setminus L_2$;  \label{it:casea3}
\item if there exist $x \in A$, $y \in B$ adjacent such that $c(x), c(y) \in L_2 \cap L_1$ and $c(B \setminus N(x)) \subseteq L_1 \setminus L_2$; \label{it:caseb1}
\item if there exist $x \in A$, $y \in B$ adjacent such that $c(x), c(y) \in L_2 \cap L_1$ and $c(A \setminus N(y)) \subseteq L_2 \setminus L_1$; \label{it:caseb2}
\item if there exist $x,x' \in A, y,y' \in B$, with $x, y$ adjacent, $x'$ non-adjacent to $y$, $y'$ non-adjacent to $x$, and (consequently) $x'$ adjacent to $y'$, and $c(x),c(y),c(x'),c(y') \in L_2 \cap L_1$. \label{it:caseb3}
\end{enumerate}
It is easy to verify that one of these cases occurs.

With the notation as above, if $case_i = $
\begin{enumerate}[(a)]
\item then we let $R_i^{j} = \emptyset$ for $j = 1,2,3,4$;
\item then we let $R_i^{j} = \emptyset$ for $j = 1,2,3,4$;
\item then we let $x \in A_i, y \in B_i$ and set $R_i^1 = \sset{x}$, $R_i^2 = \sset{y}$, $R_i^3 = R_i^4 = \emptyset$;
\item then we let $R_i^1 = \sset{x}$, $R_i^2 = \sset{y}$, $R_i^3 = R_i^4 = \emptyset$;
\item then we let $R_i^1 = \sset{x}$, $R_i^2 = \sset{y}$, $R_i^3 = R_i^4 = \emptyset$;
\item then we let $R_i^1 = \sset{x}$, $R_i^2 = \sset{y}$, $R_i^3 = \sset{x'}$, $R_i^4 = \sset{y'}$.
\end{enumerate}
\item For $j = 1,2$, we proceed as follows. If $S_i^j = \emptyset$ or the vertex $v \in S_i^j$ is not mixed on a bad component, then we let $X_i^{j,1} = X_i^{j,2} = C_i^j = \emptyset$. Otherwise, let $v \in S_i^j$ and let $C$ be a bad component of $G|Y$ on which $v$ is mixed. We set $C_i^j = V(C)$. By Lemma \ref{lem:Ycomps} applied to $C$, it follows that for $p \neq q$, $V(C) \cap Y_p$ is complete to $V(C) \cap Y_q$. Since $Y_p \cap V(C) \neq \emptyset$ for at least three different $p \in \sset{1,2,4,5}$, it follows that there exist $p,q \in \sset{1,2,4,5}$ with $p \neq q$ such that $|c(V(C) \cap Y_p)| = 1$ and $|c(V(C) \cap Y_q)| = 1$. Let $X_i^{j,1} \subseteq V(C) \cap Y_p$, $X_i^{j,2} \subseteq V(C) \cap Y_q$, such that $|X_i^{j,k}| = 1$ for $k = 1,2$.
\end{itemize}
We let $f_i = c|_{S_i^1 \cup S_i^2 \cup R_i^1 \cup R_i^2 \cup  R_i^3 \cup R_i^4 \cup X_i^{1,1} \cup X_i^{1,2} \cup X_i^{2,1} \cup X_i^{2,2}}$. 
It follows from the definition of $Q$ that $Q \in \mathcal{Q}$. Moreover, $c$ is a precoloring extension of $P'^Q$ by the definition of $Q$ and $P'^Q$. This proves \eqref{eq:equivalent}. 

\bigskip

Let $P' \in \mathcal{L}$ with $P' = (G, S', X_0', X', Y', Y^*, f')$ such that $P' = P'^Q$ for $Q = (Q_1, \dots, Q_r)$, where for each $i$, $$Q_i = (S_i^1, S_i^2, R_i^1, R_i^2, R_i^3, R_i^4, C_i^1, C_i^2, X_i^{1,1},X_i^{1,2},X_i^{2,1}, X_i^{2,2}, f_i, case_i).$$ Let $Y_i' = \sset{y \in Y': L_{P'}(y) = L_i}$ for $i = 1,2$. We claim the following. 

\vspace*{-0.4cm}\begin{equation}\vspace*{-0.4cm} \label{eq:claimx}
  \longbox{\emph{$P'$ satisfies \eqref{it:x11}.}}
\end{equation}

Suppose not; and let $x-a-b$ be a path with $x \in X'$ and $a, b \in Y'$ with $L_{P'}(a) = L_{P'}(b) = L$. Since $P$ satisfies \eqref{it:122} and \eqref{it:x11}, it follows that $x \not\in X$, and so $x \in Y$ and $L_P(x) = L$. Moreover, since $x \in X' \setminus X$, it follows that $x$ has a neighbor $s' \in S' \setminus S$ with $f'(s') \in L$. Since $P$ satisfies \eqref{it:122} and \eqref{it:x11}, and since $s'$ is adjacent to $x$ but not $a$, it follows that $s' \in Y$ and $L_P(s') = L$. Since $s'$ has a neighbor $x \in Y$ with a neighbor $a \in Y'$, it follows that $x \not\in \tilde{Y}^Q \cup \tilde{Z}^Q$. Since $s' \not\in X$, it follows that $s' \not\in S_i^1 \cup S_i^2$, and hence there exists $i \in \sset{1, \dots, r}$ such that $s' \in R_i^j$ for some $j \in \sset{1,2,3,4}$. Thus $L_P(s') \in \sset{L_1, L_2}$. Let $\sset{u,v} = S_i^1 \cup S_i^2$ such that $s' \in N(u) \setminus N(v)$. It follows that $case_i \in \sset{(d), (e), (f)}$, and hence there is a vertex $t' \in R_i^1 \cup R_i^2 \cup R_i^3 \cup R_i^4$ such that $t'$ is adjacent to $s'$ and $v$, but $t'$ is not adjacent to $u$, and $L_P(t') \in \sset{L_1, L_2} \setminus \sset{L}$, and $f'(t') \in L_1 \cap L_2$. But then $t'-s'-x$ or $t'-x-a$ is a path (since $a\in Y'$ it follows that $a$ is not adjacent to $t'$); contrary to the fact that \eqref{it:122} holds for $P$. This is a contradiction, and \eqref{eq:claimx} follows. 

\vspace*{-0.4cm}\begin{equation}\vspace*{-0.4cm} \label{eq:claimstar}
  \longbox{\emph{If $P$ satisfies \eqref{it:star} for lists $L_1', L_2', L_3'$, then $P'$ satisfies \eqref{it:star} for $L_1', L_2', L_3'$.}}
\end{equation}

Suppose not; and let $x-a-b-c$ be a path such that $L_{P'}(x) = L_3'$ with $|L_3'| = 2$ and $L_3' \neq L_1' \cap L_2'$, $L_{P'}(a) = L_1' = L_{P'}(c)$, $L_{P'}(b) = L_2'$. Since $P$ satisfies \eqref{it:122}, \eqref{it:star} for $L_1', L_2', L_3'$, and \eqref{it:123}, it follows that $L_P(x) = L_2'$. Consequently, $x$ has a neighbor $s'$ in $S' \setminus S$ with $f'(s') \in L_2'$. Since $L_3' \neq L_1' \cap L_2'$, it follows that $f'(s') \in L_1'$. Thus $s'-x-a-b-c$ is a path. Suppose first that $s' \in Y$. It follows that $s' \not\in S_i^1 \cup S_i^2$. Since $s'$ has a neighbor $x \in Y$ with a neighbor $a \in Y'$, it follows that $x \not\in \tilde{Y}^Q \cup \tilde{Z}^Q$. This implies that there exist $i \in \sset{1,\dots, r}$ and $j \in \sset{1,2,3,4}$ such that $s' \in R_i^j$. Since $P$ satisfies \eqref{it:122} and \eqref{it:123}, it follows that $L_P(s') = L_1'$. Let $\sset{u,v} = S_i^1 \cup S_i^2$ such that $u$ is adjacent to $s'$ and $v$ is not. It follows that $case_i \in \sset{(d), (e), (f)}$, and hence there is a vertex $t' \in R_i^1 \cup R_i^2 \cup R_i^3 \cup R_i^4$ such that $t'$ is adjacent to $s'$ and $v$, but $t'$ is not adjacent to $u$, and $f'(t') \in L_1' = L_P(s')$. Since $t'-s'-x-a-b-c$ is not a $P_6$ in $G$, it follows that $f'(t') \not\in L_1' \cap L_2'$. Therefore, $L_P(t') \not\in \sset{L_1', L_2'}$. Since $P$ satisfies \eqref{it:123}, it follows that $t'$ is adjacent to $x$ (since $t'-s'-x$ is not a path). Since $f'(t') \in L_1'$, it follows that $t'$ is not adjacent to $a$. Now $t'-x-a$ is a path in $G$, contrary to the fact that $P$ satisfies \eqref{it:123}. Thus, $s' \in X$. 

Suppose that $L_P(s') \neq L_1' \cap L_2'$. Then $s'$ has a neighbor $s$ in $S$ with $f'(s) \in L_1' \cap L_2'$. Now $s-s'-x-a-b-c$ is a $P_6$ in $G$, a contradiction. It follows that $s' \in X$ and $L_P(s') = L_1' \cap L_2'$. Since $s' \in S_i^1 \cup S_i^2$, it follows that there is a path $s'-y-z$ with $y \in L_1, z \in L_2$, and $L_P(s') \neq L_1 \cap L_2$. It follows that either $L_1 \not\in \sset{L_1', L_2'}$ or $L_2 \not\in \sset{L_1', L_2'}$. Since $z-y-s'-x-a-b-c$ is not a $P_7$ in $G$, it follows that $G|\sset{z,y,x,a,b,c}$ is connected. Let $w \in \sset{y,z}$ such that $L_P(w) \not\in \sset{L_1', L_2'}$. Since $P$ satisfies \eqref{it:123}, it follows that $w$ is complete to $x, a, b, c$. But then $x-w-c$ is a path, contrary to the fact that \eqref{it:123} holds for $P$. This implies \eqref{eq:claimstar}.

\vspace*{-0.4cm}\begin{equation}\vspace*{-0.4cm} \label{eq:claimy}
  \longbox{\emph{If $P$ satisfies \eqref{it:x12} for lists $L_1', L_2', L_3'$, and $P$ satisfies \eqref{it:star} for all lists, then $P'$ satisfies \eqref{it:x12} for $L_1', L_2', L_3'$.}}
\end{equation}

Suppose not; and let $x-a-b$ be a path such that $L_{P'}(x) = L_3'$ with $|L_3'| = 2$ and $L_3' \neq L_1' \cap L_2'$, $L_{P'}(a) = L_1'$, $L_{P'}(b) = L_2'$. Since $P$ satisfies \eqref{it:122}, \eqref{it:x12} for $L_1', L_2', L_3'$, and \eqref{it:123}, it follows that $L_P(x) = L_2'$. Consequently, $x$ has a neighbor $s'$ in $S' \setminus S$ with $f'(s') \in L_2'$. Since $L_3' \neq L_1' \cap L_2'$, it follows that $f'(s') \in L_1'$. Thus $s'-x-a-b$ is a path. Suppose first that $s' \in X$. Then there exist $i \in \sset{1, \dots, r}$ and $j \in \sset{1,2}$ such that $s' \in S_i^j$. It follows that $L_P(s') = L_1' \cap L_2'$, since $P$ satisfies \eqref{it:star} for all lists. By construction, it follows that there is a path $s'-y-z$ with $y \in L_1, z \in L_2$, and $L_P(s') \neq L_1 \cap L_2$. We choose such $y, z \in C_i^j$ if $C_i^j \neq \emptyset$. Since $z-y-s'-x-a-b$ is not a six-vertex path in $G$, it follows that $G|\sset{z,y,x,a,b}$ is connected. Since $C_i^j \cap Y' = \emptyset$ by construction, it follows that $C_i^j = \emptyset$, and so $s'$ is not mixed on a bad component. Since $L_1 \cap L_2 \neq L_1' \cap L_2'$, it follows that either $L_1 \not\in \sset{L_1', L_2'}$ or $L_2 \not\in \sset{L_1', L_2'}$.  Let $w \in \sset{y,z}$ such that $L_P(w) \not\in \sset{L_1', L_2'}$. Then $G|\sset{z,y,x,a,b}$ is contained in a component of $G|Y$ containing vertices with lists $L_1', L_2'$ and $L_P(w)$, hence a bad component. But since $s'-x-a-b$ is a path, $s'$ is mixed on this bad component, a contradiction. It follows that $s' \in Y$. 

Since $P$ satisfies \eqref{it:122} and \eqref{it:123}, it follows that $L_P(s') = L_1'$. Since $s'$ has a neighbor $x \in Y$ with a neighbor $a \in Y'$, it follows that $s' \not\in \tilde{Y}^Q \cup \tilde{Z}^Q$. Thus, there exist $i \in \sset{1, \dots, r}$ and $j \in \sset{1,2,3,4}$ such that $s' \in R_i^j$. By construction, it follows that $L_P(s') \in \sset{L_1, L_2}$. Let $\sset{u,v} = S_i^1 \cup S_i^2$ such that $u$ is adjacent to $s'$ and $v$ is not. It follows that $case_i \in \sset{(d), (e), (f)}$, and hence there is a vertex $t' \in R_i^1 \cup R_i^2 \cup R_i^3 \cup R_i^4$ such that $t'$ is adjacent to $s'$ and $v$, but $t'$ is not adjacent to $u$, and $f'(t') \in L_1 \cap L_2$. 

Suppose first that $\sset{L_1', L_2'} = \sset{L_1, L_2}$. Then $f'(s'), f'(t') \in L_1' \cap L_2'$. Let $s \in S$ be a common neighbor of $u,v$ with $f'(s) \in L_1 \cap L_2$. Since $s-u-s'-x-a-b$ is not a $P_6$ in $G$, it follows that $u$ is adjacent to $a$. Since $t'-v-s-u-a-b$ is not a $P_6$ in $G$, it follows that $v$ has a neighbor in $\sset{u,a,b}$. Since $f'(v) \in L_1$, it follows that $v$ is non-adjacent to $a$. Thus $v$ is adjacent to $b$. Since $a, b \not\in \tilde{T}^Q$, it follows that $case_i = (f)$. By symmetry, we may assume that $s' \in R_i^1, t' \in R_i^2$. Let $x' \in R_i^3,y' \in R_i^4$.  Then $x', y'$ are non-adjacent to $a, b$. But then $x'-u-a-b-v-y$ is a $P_6$ in $G$, a contradiction. It follows that $\sset{L_1', L_2'} \neq \sset{L_1, L_2}$.

Consequently, $L_P(t') \not\in \sset{L_1', L_2'}$. Since $P$ satisfies \eqref{it:123}, it follows that $t'-s'-x$ is not a path, and so $t'$ is adjacent to $x$. Since $f'(t') \in L_1'$, it follows that $t'$ is not adjacent to $a$. Now $t'-x-a$ is a path, contrary to the fact that \eqref{it:123} holds for $P$. This proves \eqref{eq:claimy}.

\vspace*{-0.4cm}\begin{equation}\vspace*{-0.4cm} \label{eq:claimmain1}
  \longbox{\emph{$P'$ satisfies \eqref{it:star} for $L_1, L_2, L_3$.}}
\end{equation}

Suppose not; and let $z-a-b-c$ be a path with $L_{P'}(z) = L_3$, $L_{P'}(a) = L_{P'}(c) = L_1$, $L_{P'}(b) = L_2$. Suppose first that $z \in X$. Let $i$ such that $T_i = N(z) \cap S$. Then $S_i^1 \neq \emptyset$. Let $s' \in S_i^1 \cup S_i^2$, and let $s$ be a common neighbor of $s'$ and $z$ in $S$ with $f(s) \in L_1 \cap L_2$. Since $s'-s-z-a-b-c$ is not a path, it follows that $z,a,b,c$ contains a neighbor of $s'$ for every $s' \in S_i^1 \cup S_i^2$. But $z$ is anticomplete to $S_i^1 \cup S_i^2$, for otherwise, $z \in \tilde{V}^Q$. If $S_i^2 = \emptyset$, then, since $z \not\in X_0'$, it follows that $f(s') \in L_1 \cap L_2$ and so $z$ is anticomplete to $a,b,c$, a contradiction. Therefore, $S_i^2 \neq \emptyset$. But then $S_i^1 \cup S_i^2 = \sset{u,v}$ with $f'(u) \in L_2 \setminus L_1$, say. Since $a, b, c \in Y'$, it follows that $u$ is adjacent to $a$ or $c$, and $v$ is adjacent to $b$; and no other edges between $u,v$ and $a,b,c$ exist. Now, $Y'$ contains an edge between $N(u) \cap (Y_1 \setminus N(v))$ and $N(v) \cap (Y_2 \setminus N(u))$; but this contradicts \eqref{eq:eq16}.

Since $P$ satisfies \eqref{it:122} and \eqref{it:123}, it follows that $L_P(z) = L_2$. Then $z$ has a neighbor $s' \in S' \setminus S$ with $f'(s') \in L_1 \cap L_2$ (for if $f'(s') \not\in L_1$, then $L_{P'}(z) = L_1 \cap L_2 \neq L_3$), and $s'-z-a-b-c$ is a path. Suppose first that $s' \in Y$. Since $P$ satisfies \eqref{it:122} and \eqref{it:123}, it follows that $L_P(s') = L_1$. Moreover, by construction, $s'$ has a neighbor $t' \in S'$ with $L_P(t') = L_2$ and $f'(t') \in L_1 \cap L_2$. But then $t'-s'-z-a-b-c$ is a $P_6$ in $G$, a contradiction. It follows that $s' \in X$. 

Since $s' \in X$, it follows that $L(s') = L_3$, and so $s'$ has a neighbor $s \in S$ with $f(s) \in L_1 \cap L_2$. But then $s-s'-z-a-b-c$ is a $P_6$ in $G$, a contradiction. This proves \eqref{eq:claimmain1}. 

\vspace*{-0.4cm}\begin{equation}\vspace*{-0.4cm} \label{eq:claimmain2}
  \longbox{\emph{If $P$ satisfies \eqref{it:star} for every three lists, then $P'$ satisfies \eqref{it:x12} for $L_1, L_2, L_3$.}}
\end{equation}

Suppose not; and let $z-a-b$ be a path with $L_{P'}(z) = L_3$, $L_{P'}(a) = L_1$, $L_{P'}(b) = L_2$.

Suppose first that $z \in X$. Let $i \in \sset{1, \dots, r}$ such that $T_i = N(z) \cap S$. By construction, it follows that $S_i^1 \neq \emptyset$. Let $s' \in S_i^1 \cup S_i^2$, and let $s$ be a common neighbor of $s'$ and $z$ in $S$ with $f(s) \in L_1 \cap L_2$. Let $c$ be a neighbor of $s'$ in $Y_1$; by construction, we may choose $c$ to be non-adjacent to $z$. Then $c \neq a,b$ (since $b \not\in Y_1$). Since $c-s'-s-z-a-b$ is not a path, it follows that either $s'$ or $c$ has a neighbor in $\sset{a,b}$. Since $P$ satisfies \eqref{it:x11}, it follows that $s'-c-a$ is not a path. Since $P$ satisfies \eqref{it:star} for all lists, it follows that $z-a-b-c$ is not a path. Consequently, $s'$ has a neighbor in $\sset{a,b}$. It follows that $f'(s') \not\in L_1 \cap L_2$. Therefore, $S_i^1 \cup S_i^2 = \sset{u,v}$ and both $u,v$ have a neighbor in $\sset{a,b}$. Since $a, b in Y'$, it follows that both $a, b$ have a non-neighbor in $\sset{u,v}$. This is a contradiction by \eqref{eq:eq16}. 

Since $z \in Y$ and $P$ satisfies \eqref{it:122} and \eqref{it:123}, it follows that $L_P(z) = L_2$. Consequently, $z$ has a neighbor $s'$ in $S' \setminus S$ with $f'(s') \in L_2$. Since $L_3 \neq L_1 \cap L_2$, it follows that $f'(s') \in L_1$. Thus $s'-z-a-b$ is a path. Since $s'$ has a neighbor $z \in Y$ with a neighbor $a \in Y'$, it follows that $s' \not\in \tilde{Y}^Q \cup \tilde{Z}^Q$.  Suppose first that $s' \in X$. Then there exist $i \in \sset{1, \dots, r}$ and $j \in \sset{1,2}$ such that $s' \in S_i^j$. It follows that $L_P(s') = L_1 \cap L_2$ since $P$ satisfies \eqref{it:star} for all lists. But $S_i^j \subseteq X_3$ and so $L_P(s') \neq L_1 \cap L_2$, a contradiction. It follows that $s' \in Y$. 

Since $P$ satisfies \eqref{it:122} and \eqref{it:123}, it follows that $L_P(s') = L_1$, and there exist $i \in \sset{1, \dots, r}$ and $j \in \sset{1,2,3,4}$ such that $s' \in R_i^j$. Moreover, $S_i^1 \cup S_i^2 = \sset{u, v}$. By symmetry, we may assume that $u$ is adjacent to $s'$ and $v$ is not. It follows that $case_i \in \sset{(d),(e),(f)}$, and hence there is a vertex $t' \in R_i^1 \cup R_i^2 \cup R_i^3 \cup R_i^4$ such that $t'$ is adjacent to $s'$ and $v$, but $t'$ is not adjacent to $u$. By construction, it follows that $f'(s'), f'(t') \in L_1 \cap L_2$. Let $s \in T_i$ with $f'(s) \in L_1 \cap L_2$. Since $s-u-s'-z-a-b$ is not a $P_6$ in $G$, it follows that $u$ is adjacent to $a$ or to $z$. Note that if $uz \in E(G)$, then $z$ is adjacent to both $s'$ and $u$, both of which are in $S'$ and $f(s', u) \subseteq L_1$. This implies that $z \in X_0'$. It follows that $u$ is adjacent to $a$. Since $t'-v-s-u-a-b$ is not a $P_6$ in $G$, it follows that $v$ is adjacent to $b$. This contradicts \eqref{eq:eq16} and concludes the proof of \eqref{eq:claimmain2}. 

\bigskip
The statement of the lemma follows; we have proved every claim in \eqref{eq:claimx}, \eqref{eq:claimstar}, \eqref{eq:claimy}, \eqref{eq:claimmain1} and \eqref{eq:claimmain2}. 
\end{proof}

\begin{lemma} \label{lem:x12}
There is a function $q: \mathbb{N} \rightarrow \mathbb{N}$ such that the following holds. Let $P = \sspc$ be a starred precoloring of a $P_6$-free graph $G$ with $P$ satisfying \eqref{it:size3}, \eqref{it:122}, \eqref{it:123} and \eqref{it:x11}. Then there is an
algorithm with running time $O(|V(G)|^{q(|S|)})$
that outputs an equivalent collection $\mathcal{L}$ for $P$ such that
\begin{itemize}
\item $|\mathcal{L}| \leq |V(G)|^{q(|S|)}$;
\item every $P' \in \mathcal{L}$ is a starred precoloring of $G$;
\item every $P' \in \mathcal{L}$ with seed $S'$ satisfies $|S'| \leq q(|S|)$;  and
\item every $P' \in \mathcal{L}$ satisfies \eqref{it:size3}, \eqref{it:122}, \eqref{it:123}, \eqref{it:x11} and \eqref{it:x12}. 
\end{itemize}
Moreover, for every $P' \in \mathcal{L}$, given a precoloring extension of $P'$, we can compute a precoloring extension for $P$ in polynomial time, if one exists. 
\end{lemma}
\begin{proof}
Let $\mathcal{L}=\{P\}$. For every triple $(L_1,L_2,L_3)$ of lists of size three, we repeat the following. Apply Lemma \ref{lem:fixedl1l2l3x} to 
every member of $\mathcal{L}$, replace $\mathcal{L}$ with the union of the equivalent collections thus obtained, and move to the next triple. At the end 
of thus process \eqref{it:star} holds for every $P' \in \mathcal{L}$.

Now repeat the procedure of the previous paragraph.
Since at this stage  all inputs satisfy \eqref{it:star} for every triple of 
lists, it follows that  \eqref{it:x12} holds for every starred precoloring of 
the output. 
\end{proof}

We now observe that the next axiom, which we restate, holds. 
\begin{enumerate}
\item[\eqref{it:l1l2}]  For every component $C$ of $G|Y$, for which there is a  vertex of $X$
is mixed on $C$, 
there exist $L_1, L_2 \subseteq \sset{1,2,3,4}$ with $|L_1| = |L_2| = 3$ such that $C$ contains a vertex $x_i$ with $L_P(x_i) = L_i$ for $i=1,2$, every vertex $x$ in $C$ satisfies $L_P(x) \in \sset{L_1, L_2}$, and every $x \in X$ mixed
on $C$ satisfies $L_P(x)=L_1 \cap L_2$. 
\end{enumerate}

\begin{lemma}
\label{lem:l1l2}
Let $P=\sspc$ of a $P_6$-free graph $G$ satisfying 
\eqref{it:size3}-\eqref{it:x12}, and let $C$ be a component of $G|Y$
such that some vertex $x \in X$ is mixed on $C$. Then  $C$ meets
exactly two lists $L_1,L_2$, and $L_P(x)=L_1 \cap L_2$.
\end{lemma}

\begin{proof}
Since $P$ satisfies \eqref{it:x11}, Lemma~\ref{lem:mixed}
implies that $C$ meets more than one list.
By Lemma~\ref{lem:mixed}, there exist
$a,b$ in $C$ such that $x-a-b$ is a path. By \eqref{it:x11} $L_P(a)\neq L_P(b)$,
and by \eqref{it:x12} $L_P(x)=L_P(a) \cap L_P(b)$. 
Let $c \in V(C)$ be such that $L_P(c) \neq L_P(a), L_P(b)$.
By Lemma~\ref{lem:Ycomps} $c$ is complete to $\{a,b\}$.
But then $x$ is mixed on one of $\{a,c\}$, $\{b,c\}$, contrary
to \eqref{it:x12}. This proves Lemma~\ref{lem:l1l2}.
\end{proof}

The following lemma establishes that: 
\begin{enumerate}
\item [\eqref{it:orthogonal}] For every component $C$ of $G|Y$ such that some vertex of $X$ is
mixed on $C$, and for $L_1, L_2$ as in \eqref{it:l1l2}, $L_P(v) = L_1 \cap L_2$ for every vertex $v \in X$ with a neighbor in $C$. 
\end{enumerate}
\begin{lemma} \label{lem:orthogonal}
There is a function $q: \mathbb{N} \rightarrow \mathbb{N}$ such that the following holds. Let $P = \sspc$ be a starred precoloring of a $P_6$-free graph $G$ with $P$ satisfying \eqref{it:size3}, \eqref{it:122}, \eqref{it:123}, \eqref{it:x11}, \eqref{it:x12} and \eqref{it:l1l2}. Then there is an
algorithm with running time $O(|V(G)|^{q(|S|)})$
that outputs an equivalent collection $\mathcal{L}$ for $P$ such that
\begin{itemize}
\item $|\mathcal{L}| \leq |V(G)|^{q(|S|)}$;
\item every $P' \in \mathcal{L}$ is a starred precoloring of $G$;
\item every $P' \in \mathcal{L}$ with seed $S'$ satisfies $|S'| \leq q(|S|)$;  and
\item every $P' \in \mathcal{L}$ satisfies \eqref{it:size3}, \eqref{it:122}, \eqref{it:123}, \eqref{it:x11}, \eqref{it:x12}, \eqref{it:l1l2} and \eqref{it:orthogonal}. 
\end{itemize}
Moreover, for every $P' \in \mathcal{L}$, given a precoloring extension of $P'$, we can compute a precoloring extension for $P$ in polynomial time, if one exists. 
\end{lemma}
\begin{proof}Let $\mathcal{R}=\{T_1, \ldots, T_r\}$ be the set of all 
$T \subseteq S$ with  $|f(T)|=2$, let
$S = \sset{s_1, \dots, s_s}$, and let $\mathcal{T} = \sset{(L_1^1, L_2^1), \dots, (L_1^t, L_2^t)}$ be the set of all pairs $(L_1, L_2)$ with $|L_1| = |L_2| = 3$ and $L_1 \neq L_2$. We let $\mathcal{Q}$ be the set of all $(rst + 1)$-tuples $Q = (Q_{1,1,1}, \dots, Q_{r,s,t}, f')$, where $i \in [r], j \in [s]$ and $k \in [t]$, and for each $i,j,k$ the following statements hold:
\begin{itemize}
\item $Q_{i,j,k} \subseteq X(T_i)$ and $|Q_{i,j,k}| \leq 1$;
\item $Q_{i,j,k} = \emptyset$ if $[4] \setminus f(T_i) = L_1^k \cap L_2^k$ or $f(s_j) \in f(T_i)$;
\item if $Q_{i,j,k} = \sset{x}$, then there is a component $C$ of $G|Y$ such that
  \begin{itemize}
  \item $s_j$ has a neighbor in $V(C)$;
  \item some vertex of $X$ is mixed on $C$, and $C$ meets $L_1^k, L_2^k$ as 
in \eqref{it:l1l2};
  \item $x$ has neighbors in $V(C)$ 
  \end{itemize}
  and $x$ has the maximum number  of such components $C$ among all vertices 
in $X(T_i)$; 
\item if $Q_{i,j,k} = \emptyset$, then no vertex $x \in X(T_i)$ and component $C$ as above exist,

\item  Let $\tilde{Q} = \bigcup_{i \in \sset{1, \dots, r}, j \in \sset{1, \dots, s}, k \in \sset{1, \dots, t}} Q_{i,j,k}$,  then $f' : \tilde{Q} \rightarrow \sset{1,2,3,4}$ satisfies that $f' \cup f$ is a proper coloring of $G|(S \cup X_0 \cup \tilde{Q})$. 
\end{itemize}

For $Q \in \mathcal{Q}$, we construct a starred precoloring $P^Q$ from $P$ as follows. We let $\tilde{Z}^Q$ be the set of vertices $z$ in $X \setminus \tilde{Q}$ such that $\tilde{Q}$ contains a neighbor $x$ of $z$ with $f'(x) \in L_P(z)$, and let $g^Q : \tilde{Z}^Q \rightarrow \sset{1,2,3,4}$ be the unique function such that $g^Q(z) \in L_P(z) \setminus f'(N(z) \cap \tilde{Q})$. We let $\tilde{X}^Q$ be the set of vertices $z$ in $Y$ such that $\tilde{Q}$ contains a neighbor $x$ of $z$ with $f'(x) \in L_P(z)$. 

We let
$$P^Q = (G, S \cup \tilde{Q}, X_0 \cup \tilde{Z}^Q, (X \setminus (\tilde{Z}^Q \cup \tilde{Q})) \cup \tilde{X}^Q, Y \setminus \tilde{X}^Q, Y^*, f \cup f' \cup g^Q),$$
and let $\mathcal{L} = \sset{P^Q : Q \in \mathcal{Q}, f \cup f' \cup g^Q \mbox{ is a proper coloring}}$. It is easy to check that $\mathcal{L}$ is an equivalent collection for $P$. 

Let $Q \in \mathcal{Q}$, and let 
$P^Q = (G', S', X_0', X', Y', Y^*, f')$.
By construction, $P^Q$ satisfies \eqref{it:size3}. Since $P$ satisfies \eqref{it:122}, \eqref{it:123}, so does $P^Q$. 
Since $P$ satisfies \eqref{it:122}, it follows that $P^Q$
satisfies  \eqref{it:x11}.

\vspace*{-0.4cm}\begin{equation}\vspace*{-0.4cm} \label{eq:satx12}
  \longbox{\emph{$P^Q$ satisfies \eqref{it:x12}.}}
\end{equation}

Suppose not; and let $a-b-c$ be a path with $a \in X'$, $b, c \in Y'$ such that $L_{P^Q}(a) = L_3, L_{P^Q}(b) = L_1, L_{P^Q}(c) = L_2$ and $L_1 \neq L_2$, $L_3 \neq L_1 \cap L_2$. Since $P$ satisfies \eqref{it:x12}, it follows that $a \in Y$. Since $P$ satisfies \eqref{it:122} and \eqref{it:123}, it follows that $L_P(a) = L_2$, and there is a vertex $x \in \tilde{Q}$, say $x \in Q_{i,j,k}$ such that $x$ is adjacent to $a$ and $f'(x) \in L_P(a)$.  Since $c \in Y'$, it follows that $x$ is not adjacent to $c$. Since $x$ is mixed on a component of $G|Y$ 
meeting $L_1$ and $L_2$, and since $P$ satisfies \eqref{it:l1l2}, 
it follows that $L_P(x) = L_1 \cap L_2$. Thus $x-a-b-c$ is a path, and there is a component $C$ of $G|Y$ such that $V(C)$ meets $L_1^k, L_2^k$ and $x$ has a neighbor in $C$ and $L_1^k \cap L_2^k \neq L_P(x) = L_1 \cap L_2$. It follows that $a, b,c \not\in V(C)$, and so $V(C)$ is anticomplete to $a, b, c$. By symmetry, we may assume that $L_1^k \not\in \sset{L_1, L_2}$. Let $d \in V(C)$ with $L_P(d) = L_1^k$. Since $P$ satisfies \eqref{it:x12} and \eqref{it:x11}, and since $x$ has a neighbor in $C$, it follows that $x$ is complete to $C$ and thus adjacent to $d$. Since $L_P(d) \not\in \sset{L_1, L_2}$, it follows that there is a vertex $s \in S$ with $f(s) \in L_1 \cap L_2$ and $s$ adjacent to $d$. But then $c-b-a-x-d-s$ is a  $P_6$ in $G$, a contradiction. This proves \eqref{eq:satx12}. 

\bigskip

Now by Lemma~\ref{lem:l1l2}, $P^Q$ satisfies \eqref{it:l1l2}.

\vspace*{-0.4cm}\begin{equation}\vspace*{-0.4cm} \label{eq:orthogonal}
  \longbox{\emph{$P^Q$ satisfies \eqref{it:orthogonal}.}}
\end{equation}

Suppose not. Let $C$ be a component of $G'|Y'$ such that some
vertex of $X'$ is mixed on $C$, and 
with $L_1, L_2$ as in \eqref{it:l1l2}, and let $v \in X'$ with $ N(v)\cap C\neq \emptyset $ such that $L_{P^Q}(v) \neq L_1 \cap L_2$.

Since $L_{P^Q}(v) \neq L_1 \cap L_2$, 
we may assume that $[4] \setminus L_1 \subseteq L_P(v)$.
Let $s \in S$ with $f(s)=[4] \setminus L_1$, such that $s$ has a neighbor in 
$C$. Since $P^Q$ satisfies \eqref{it:l1l2}, it follows that $v$ is complete
to $C$.

We claim that every $x \in X' \cap Y$ is complete to $C$.
Suppose that $x \in Y \cap X'$ is mixed on $C$. 
Since $P^Q$ satisfies \eqref{it:l1l2}, it follows that
$L_{P^Q}(x)=L_1 \cap L_2$.  By symmetry, we may assume that $L_P(x) = L_1$, and
therefore, $x$ has a neighbor $s$ in $\tilde{Q} \cap X$ 
and $f(s)=L_1 \setminus L_2$. But then $s$ is mixed on the component
$\tilde{C}$ of $G|Y$ containing $V(C) \cup \{x\}$, $\tilde{C}$ meets $L_1$ and $L_2$, and  $L_P(s) \neq L_1 \cap L_2$, contrary to the fact that $P$ 
satisfies \eqref{it:l1l2}. This proves the claim. Now since some vertex of $X'$ 
is mixed on $C$, it follows that some vertex of $X$  is mixed on $C$.

Next we claim that $v \in X$. Suppose $v \in Y$. Then there is a component
$\tilde{C}$ of $G|Y$ such that $V(C) \cup \{v\} \subseteq V(\tilde{C})$. 
Since some $x \in X$ is mixed on $C$, and since $P$ satisfies 
\eqref{it:l1l2}, we deduce that $L_P(v) \in \{L_1,L_2\}$. Consequently,
$v$  has a neighbor $s$ in $\tilde{Q}$. Therefore $q \in X$.
Since $v$ is complete to $C$, it follows that $v$ has a neighbor
$n$ in $C$ with $L_P(n)=L_P(v)$. But then $x$ is mixed on the edge
$vn$, contrary to the fact that $P$  satisfies \eqref{it:x11}. This proves that 
$v \in X$.

By construction, $Q$ contains an entry $Q_{i,j,k}$ with $T_i=T(v)$, $s_j = s$ and $(L_1^k, L_2^k) = (L_1, L_2)$, and in view of the claims of the
previous two paragraphs,   $Q_{i,j,k} \neq \emptyset$. 
Write  $Q_{i,j,k} = \sset{z}$. Let $C'$ be a component of $G|Y$ meeting 
both $L_1$ and $L_2$, such that some
vertex of $X$ is mixed on $C'$, and both $s$ and $z$ have a neighbor in $C'$. 
Since $f'(z) \in L_1 \cup L_2$, it follows that $z$ is not complete 
to $C$. Since $L_P(z) \neq L_1 \cap L_2$, it follows from the fact that
$P$ satisfies \eqref{it:l1l2} that $z$ is not mixed on either of $C,C'$.
Consequently, $z$ is complete to $C'$, and $z$ is anticomplete to $C$.
Now by the maximality of $z$  we may assume that $v$ is
anticomplete to $C'$. Since $[4] \setminus L_1 \subseteq L_P(z)=L_P(v)$,
it follows that $s$ is anticomplete to $\{z,v\}$.

Let $a \in V(C) \cap N(s)$ and $a' \in V(C') \cap N(s)$. Since
each of $C,C'$ meets $L_2$, we can also choose
$b \in V(C) \setminus N(s)$ and $b' \in V(C') \setminus N(s)$.
$L_P(z) \neq L_1 \cap L_2$, there exists $t \in T_i$
with $f(t) \in L_1 \cap L_2$. Then $t$ is anticomplete to $V(C) \cup V(C')$.
It $t$ is non-adjacent to $s$, then $s-a-v-t-z-a'$ is a $P_6$
in $G$, so $t$ is adjacent to $s$. If $a$ is non-adjacent to $b$,
then  $b-v-a-s-a'-z$ is a $P_6$, so $a$ is adjacent to $b$. 
But now $b-a-s-t-z-a'$ is a $P_6$, a contradiction.
Thus, \eqref{eq:orthogonal} follows. 

\bigskip
This concludes the proof of the Lemma~\ref{lem:orthogonal}.
\end{proof}

We are now ready to prove the final axiom. 
\begin{enumerate}
\item[\eqref{it:yisempty}] $Y = \emptyset$. 
\end{enumerate}
\begin{lemma} \label{lem:yisempty}
There is a function $q: \mathbb{N} \rightarrow \mathbb{N}$ such that the following holds. Let $P = \sspc$ be a starred precoloring of a $P_6$-free graph $G$ with $P$ satisfying \eqref{it:size3}, \eqref{it:122}, \eqref{it:123}, \eqref{it:x11}, \eqref{it:x12}, \eqref{it:l1l2}, \eqref{it:orthogonal}. 
Then there is an
algorithm with running time $O(|V(G)|^{q(|S|)})$
that outputs collection $\mathcal{L}$ of starred precolorings such that
\begin{itemize}
\item if we know for every $P' \in \mathcal{L}$ whether $P'$ has a precoloring extension or not, then we can decide if $P$ has a precoloring extension in polynomial time;
\item $|\mathcal{L}| \leq |V(G)|^{q(|S|)}$;
\item every $P' \in \mathcal{L}$ is a starred precoloring of $G$;
\item every $P' \in \mathcal{L}$ with seed $S'$ satisfies $|S'| \leq q(|S|)$;  and
\item every $P' \in \mathcal{L}$ satisfies \eqref{it:yisempty}. 
\end{itemize}
Moreover, for every $P' \in \mathcal{L}$, given a precoloring extension of $P'$, we can compute a precoloring extension for $P$ in polynomial time, if one exists. 
\end{lemma}
\begin{proof}
Let $P = \sspc$.  For every component $C$ of $G \setminus (S \cup X_0)$,
Let $P_C$ be the starred precoloring
$$(G|(V(C) \cup S \cup X_0), S, X_0, X \cap V(C), Y \cap V(C), Y^*\cap V(C), f). $$
Then $P_C$ satisfies \eqref{it:size3}--\eqref{it:orthogonal}.
Let $\mathcal{L}_0$ be the collection of all such starred precolorings $P_C$. 
Clearly $P$ has a precoloring extension if and only if every member of $\mathcal{L}_0$ does, so from now on we focus on constructing an equivalent collection for each $P_C$ separately. To simplify notation, from now we will simply assume that $G \setminus (X_0 \cup S)$ is connected.

In the remainder of the proof we either find that $P$ has no precoloring extension, output $\mathcal{L}=\emptyset$ and stop, or 
construct two disjoint subsets 
$U$ and $W$ of $Y$,  and a subset $\tilde{X}_0$ of $X$ such that
\begin{itemize}
\item $U \cup W=Y$, 
\item No vertex of $X$ is mixed on a component of $G|W$,
\item For every component $C$ of $G|W$, some vertex of $X \cup X_0 \cup S$ is complete to $C$.
\item There is a set $F$ with $|F| \leq 2^6$ of colorings of $G|\tilde{X}_0$  
that contains every coloring of $G|\tilde{X}_0$ that extends to  a 
precoloring extension of $P$, and $F$ can be computed in polynomial time. 
\item $P$ has a precoloring extension  if and only if for some $f' \in F$
$$(G \setminus U,  S, X_0 \cup \tilde{X}_0, X \setminus \tilde{X}_0, W, Y^*, f \cup f')$$  has a precoloring extension.
\end{itemize}
Having constructed such $U,W, \tilde{X}_0$ and $F$,  for each $f' \in F$ we set
$$P_{f'}=(G \setminus U, S, X_0 \cup \tilde{X}_0, X \setminus \tilde{X_0}, \emptyset, Y^* \cup W, f \cup f')$$  
and output the collection $\mathcal{L}=\{P_{f'}\}_{f' \in F}$, which has the 
desired  properties.

Start with $U=W=\tilde{X}_0=\emptyset$. 
For $v \in Y$, let  $M(v)=L_P(v) \setminus f(N(v) \cap (S \cup X_0))$. For
$L \subseteq [4]$, we denote by $M_L$ the list assignment
$M_L(v)=M(v) \cap L$.
To construct $U, W$ and $\tilde{X}_0$, we first examine each component of $G|Y$ 
separately. Every time we enlarge $U$, we will ``restart'' the algorithm with
$(G,S,X_0,X,Y,Y^*,f)$ replaced by $(G \setminus U,S,X_0,X,Y \setminus U,Y^*,f)$.
Since we only do this when $U$ is enlarged, there will be at most $|V(G)|$ such 
iterations, and so it is enough to ensure that each iteration can be done in
polynomial time.

Let $C$ be a component of $G|Y$. If no vertex of $X$ is mixed on 
$C$, and some vertex of $S \cup X_0 \cup X$ is complete to  $C$, we add $V(C)$ 
to  $W$. So we may assume that either some vertex of $X$ is mixed on $C$, or 
no vertex of $X$ is complete to $C$.
Let $C_i=\{v \in V(C) \; :\; L_P(v)=[4] \setminus \{i\}\}$.
Since $P$ satisfies \eqref{it:size3}, it follows that $V(C)=\bigcup_{i=1}^4 C_i$,

Suppose first  that $C$ meets exactly one list  $L$. Since $P$ satisfies
\eqref{it:l1l2}, it follows that no vertex of $X$ is mixed on $C$, and so 
$N(V(C)) \subseteq S \cup X_0$. By Theorem~\ref{3colP7}, we can test in 
polynomial time if $(C,M)$ is colorable.
If not, then $P$ has no precoloring extension, we set $\mathcal{L}=\emptyset$
and stop. If $(C,M)$ is colorable, then deleting $V(C)$ does not
change the existence of a precoloring extension for $P$, and we add $V(C)$ to 
$U$.

Now suppose that $C$ meets at least three lists. 
By Lemma~\ref{lem:Ycomps}  $C_i$ is complete to $C_j$ for every $i \neq j$.
Since $P$ satisfies \eqref{it:l1l2}, it follows that no vertex of $X$ is mixed 
on $C$, and so  $N(V(C)) \subseteq S \cup X_0$.  Since $C_i$ is non-empty
for at least three values of $i$, it follows that in every proper coloring of
$C$, at most two colors appear in $C_i$, and for $i \neq j$ the sets of colors that appear in $C_i$ and $C_j$ are disjoint.
By Theorem~\ref{Edwards},  for every $L \subset [4]$ with  $|L| \leq 2$ and for 
every $i$, we can test in polynomial time if $(C|C_i, M_L)$ is colorable.
If there exist disjoint lists $L_1, \ldots, L_4$ such that
$(G_i, M_{L_{i}})$ is colorable for all $i$, then deleting $V(C)$ does not
change the existence of a precoloring extension for $P$, and we add $V(C)$ to 
$U$. If no such $L_1, \ldots, L_i$ exist, then $P$ has no precoloring 
extension, we set $\mathcal{L}=\emptyset$ and stop.

Thus we may assume that $C$ meets exactly two lists, say $V(C)=C_3 \cup C_4$.
Let $A_1, \ldots, A_k$ be the components of $C|C_3$ and
$A_{k+1}, \ldots, A_t$ be the components of $C|C_4$. Since $P$ satisfies 
\eqref{it:122}, for every $i \in [k]$ and $j \in \{k+1, \ldots, t\}$, $A_i$ is 
either complete or anticomplete to $A_j$, and since $P$ satisfies \eqref{it:x11}, for every $i \in [t]$ no vertex of 
$X$ is mixed on $A_i$. Since $P$ satisfies \eqref{it:orthogonal}, if $x \in X$ 
has a  neighbor in $C$, then $L_P(x)=\{1,2\}$. By Theorem~\ref{3colP7}, for every $A_i$ and for every $L \subseteq [4]$
with $|L \cap \{1,2\}| \leq 1$, we can test in polynomial time if
$(A_i,M_L)$ is colorable. If $(A_i,M_L)$ is colorable, we say that the set
$M_L \cap \{1,2\}$ {\em works} for $A_i$.
Suppose that $\emptyset$ works for $i$. We may assume $i=1$. 
It follows that $(A_1,M)$ can be colored with color $3$.
Since $N(V(A_1)) \subseteq S \cup X_0 \cup X_{\{1,2\}} \cup C_4$, it follows that
deleting $A_i$ does not change the existence of a precoloring extension for $P$, and so we add $V(A_i)$ to $U$.  Thus we may assume that $\emptyset$ does not
work for any $i$. 

Since $C$ is connected and both $C_3,C_4$ are non-empty, it follows
that for every $i$ there is $j$ such that $A_i$ is complete to $A_j$,
and so in every proper coloring of $C$, at most one of the colors
$1,2$ appears in each $V(A_i)$. Since $\emptyset$ does not work for
any $i$, it follows that in every precoloring extension of $P$,
exactly one of the colors $1,2$ appears in each $V(A_i)$, and both $1$
and $2$ appear in $V(C)$. If some $x \in X$ is complete to $C$, then
$x \in X_{\{1,2\}}$, and so $G$ has no precoloring extension; we set
$\mathcal{L}=\emptyset$, and stop. Thus we may assume that no vertex
of $X$ is complete to $V(C)$.

Let $X_C$ be the set of vertices of $X$ that are mixed on 
$V(C)$. Then $X_C \subseteq X_{\{1,2\}}$, and 
$N(V(C)) \subseteq S \cup X_0 \cup X_C$.
Let $A_C=\{a_1, \ldots, a_t\}$.
Let $H_C$ be the graph with vertex set $X_C \cup A_C$, 
where 
\begin{itemize}
\item $a_ia_j \in E(H_C)$ if and only if $A_i$ is complete to $A_j$,
\item for $x \in X_C$, $xa_i \in E(H_C)$ if and only if $x$ is 
complete to $A_i$, and 
\item $H_C|(X_C)=G|(X_C)$.
\end{itemize}
Let $T_C(a_i)$ be the  the union of all the sets that work for $i$.
Suppose first that $X_C=\emptyset$. Then $N(V(C)) \subseteq S \cup X_0$.
By Theorem~\ref{Edwards} we can test in polynomial time if $(H_C,T_C)$
is colorable.  If $(H_C, T_C)$ is not colorable, then
$P$ has no precoloring extension; we output $\mathcal{L}=\emptyset$
and stop. Thus we may assume that $(H_C,T_C)$ is colorable.
Since  $N(V(C)) \subseteq S \cup X_0$, deleting $V(C)$ does
not change the existence of a precoloring extension, and
we add $V(C)$ to $U$. Thus we may assume that $X_C \neq \emptyset$.

Now let $C^1, \ldots, C^l$ be all the components  of $G|Y$ 
for which $V(C^i)=C^i_3 \cup C^i_4$ and $X_C \neq \emptyset$.
Let $H$ be the graph with vertex set 
$\bigcup_{i=1}^l V(H_{C^i})$ and such that
$uv \in E(H)$ if and only if either 
\begin{itemize}
\item $uv \in E(H_{C^i})$ for some $i$, or 
\item $u,v \in X$ and $uv \in E(G)$.
\end{itemize}
Let $T(v)=T_C(v)$ if $v \in V(H) \setminus X$, 
and let $T(v)=M(v)$ if $v \in V(H) \cap X$.
By Theorem~\ref{Edwards}, we can test in polynomial time if
$(H,T)$ is colorable. If $(H, T)$ is not colorable, then
$P$ has no precoloring extension; we output $\mathcal{L}=\emptyset$
and stop. Thus we may assume that $(H,T)$ is colorable. Note that
$T(v) \subseteq \{1,2\}$ for every $v \in V(H)$.

Next we will show $H$ is connected, and therefore $(H,T)$ has at most
two proper colorings, and we can compute the set of all proper 
colorings of $(H,T)$
in polynomial time. Suppose that $H$ is not connected. Since each $C^i$
is connected, it follows that $H|A_{C^i}$ is connected for all $i$,
and since for every $i$, every vertex of $X_{C^i}$ has a neighbor in $A_{C^i}$, 
it follows that $H|V(H_{C^i})$ is connected for every $i$. Let $D_1,D_2$
be distinct components of $H$. Since $G \setminus (S \cup X_0)$ is
connected, there is exist $p,q \in [l]$ such that $V(H_{C^p}) \subseteq D_1$,
$V(H_{C^q}) \subseteq D_2$, and there is a path  $P=p_1- \ldots -p_m$ in 
$G \setminus (S \cup X_0)$ with  $p_1 \in V(C^p) \cup X_{C^p}$, 
$p_m \in  V(C^q) \cup X_{C^q}$, and $P^*$ is disjoint from 
$\bigcup_{i=1}^l (V(C^i) \cup X_{C^i})$. Since for every $i$,
$N(V(C^i)) \subseteq S \cup X_0 \cup X_{C^i}$, it follows
that $p_1 \in X_{C^p}$ and $p_m \in X_{C^q}$,  and $P^*$ is anticomplete
to $V(C^p) \cup V(C^q)$.
By Lemma~\ref{lem:mixed}, 
there exist $a_p,b_p \in V(C^p)$ such that $p_m-a_p-b_p$ is a path,
and there exist $a_q,b_q \in V(C^q)$ such that $p_m-a_q-b_q$ is a path.
But now $b_p-a_p-p_1-P-p_m-a_q-b_q$ is a path of length at least six in $G$, a contradiction. This proves that $H$ is connected.

Let $\tilde{X}_0^{3,4}=V(H) \cap X$, and let $F^{3,4}$ be the set
of all proper colorings of $(G|\tilde{X}_0^{3,4},M)$ that extend to a coloring
of $(H,T)$. Then $|F^{3,4}| \leq 2$, and we can compute $F^{3,4}$ in
polynomial time. Let $U^{3,4}=\bigcup_{i=1}^lV(C^i)$.
Since for each $i$, $N(C^i) \subseteq \tilde{X}_0^{3,4} \cup S \cup X_0$, it
follows that 

\vspace*{-0.4cm}
  \begin{equation}\vspace*{-0.4cm} \label{Uij}
    \longbox{\emph{
$P$ has a precoloring extension if and only
if for some $f' \in F^{3,4}$
$$(G \setminus U^{3,4}, S, X_0 \cup \tilde{X}_0^{3,4}, X \setminus  \tilde{X}_0^{3,4}, Y \setminus U^{3,4}, Y^*, f \cup f')$$ 
has a precoloring extension.}}
\end{equation}

For every $i,j \in [4]$ with $i \neq j$ define 
$U^{i,j}$, $F^{i,j}$ and  $\tilde{X}_0^{i,j}$ similarly. 
Let $\tilde{X}_0=\bigcup \tilde{X}_0^{i,j}$. Let $F$ be the set of all
functions $f':\tilde{X}_0 \rightarrow [4]$ such that
$f'|_{\tilde{X_0}^{i,j}} \in F^{i,j}$. Then $|F| \leq 2^6$.
Let $U'=\bigcup U^{i,j}$.

If follows from \eqref{Uij} that $P$ has a precoloring extension if and only if
$$(G \setminus U', S, X_0 \cup \tilde{X}, X \setminus  \tilde{X}_0, Y \setminus U', Y^*, f \cup f')$$ 
has a precoloring extension for some $f' \in F$.
Now we add $U'$ to $U$, and Lemma~\ref{lem:yisempty} follows.
\end{proof}

We are now ready to prove our the main result, which we restate: 
\begin{theorem}
There exists an integer $C>0$ and  a polynomial-time algorithm with the following specifications.
\\
\\
{\bf Input:}  A 4-precoloring $(G,X_0,f)$ of a $P_6$-free graph $G$.
\\
\\
{\bf Output:}  A collection $\mathcal{L}$ of  excellent starred  precolorings of $G$ such 
that 
\begin{enumerate}
\item $|\mathcal{L}| \leq |V(G)|^C$,
\item for every $(G',S',X_0',X',\emptyset,Y^*,f') \in \mathcal{L}$ 
\begin{itemize}
\item $|S'| \leq C$, 
\item $X_0 \subseteq S' \cup X_0'$,
\item $G'$ is an induced subgraph of $G$, and
\item $f'|_{X_0}=f|_{X_0}$.
\end{itemize}
\item if we know for every $P \in \mathcal{L}$ whether $P$ has a precoloring extension, then we can decide in polynomial time if $(G, X_0, f)$ has a 4-precoloring extension; and  
\item given a precoloring extension for every $P \in \mathcal{L}$ such
  that $P$ has a precoloring extension, we can compute a 4-precoloring
  extension for $(G, X_0, f)$ in polynomial time, if one exists.
\end{enumerate}
\end{theorem}

\begin{proof}
Let $(G, X_0, f)$ be a 4-precoloring of a $P_6$-free graph $G$. We apply Theorem~\ref{thm:y0main} to $(G, X_0, f)$ to obtain a collection $\mathcal{L}_0$ of good seeded precolorings with the desired properties. Then we apply Lemma \ref{lem:starred} to each seeded precoloring in $\mathcal{L}_0$ to obtain a starred precoloring satisfying \eqref{it:size3}; let $\mathcal{L}_1$ be the collection thus obtained. Next, starting with $\mathcal{L}_1$,  apply Lemma \ref{lem:122}, Lemma \ref{lem:123}, Lemma \ref{lem:x11}, Lemma \ref{lem:x12}, Lemma \ref{lem:l1l2}, 
Lemma \ref{lem:orthogonal} and Lemma \ref{lem:yisempty} to each element 
in the output of the previous one, to finally obtain a collection
$\mathcal{L}$. Then $\mathcal{L}$ is an equivalent collection for $P$, and 
every element of $\mathcal{L}$ 
satisfies \eqref{it:122}, \eqref{it:123}, \eqref{it:x11}, \eqref{it:x12}, \eqref{it:l1l2}, \eqref{it:orthogonal} and \eqref{it:yisempty}. Finally, \eqref{it:yisempty} implies that each starred precoloring in $\mathcal{L}$ is excellent, as claimed. 
\end{proof}

\section{Acknowledgments}
This material is based upon work supported in part by the U. S. Army  Research 
Laboratory and the U. S. Army Research Office under    grant number 
W911NF-16-1-0404. The authors are also grateful to Pierre Charbit, Bernard Ries,
Paul Seymour, Juraj Stacho and Maya Stein for many useful discussions.

\end{document}